\documentclass[11pt,a4paper,twoside,leqno]{amsart}
\pdfoutput=1
\usepackage{comment}

\usepackage{amsmath}
\usepackage{amsthm}
\usepackage{thmtools}
\usepackage{braket}
\usepackage{mathtools} %

\usepackage{enumitem}
\setlist[enumerate]{label=\textnormal{(\roman*)}}

\usepackage{fullpage}
\usepackage[dvipsnames]{xcolor}
\definecolor{britishracinggreen}{rgb}{0.0, 0.26, 0.15}
\definecolor{cobalt}{rgb}{0.0, 0.28, 0.67}

\usepackage[english]{babel}
\usepackage[utf8]{inputenc}
\usepackage[T1]{fontenc}
\usepackage{libertine}
\usepackage[libertine]{newtxmath}
\usepackage[scaled=0.83]{beramono}
\usepackage{mathrsfs}
\usepackage{soul}
\usepackage{microtype,etoolbox}
\DeclareSymbolFont{usualmathcal}{OMS}{cmsy}{m}{n}
\DeclareSymbolFontAlphabet{\mathcal}{usualmathcal}
\DeclareMathAlphabet\BCal{OMS}{cmsy}{b}{n}
\usepackage{hyperref}
\hypersetup{
  hypertexnames=false,
  colorlinks,
  citecolor=britishracinggreen,
  filecolor=black,
  linkcolor=cobalt,%
  urlcolor=black
}
\usepackage[capitalize]{cleveref}

\numberwithin{equation}{section}

\newcommand\dash{\nobreakdash-\hspace{0pt}}

\def\into{\hookrightarrow}

\def\Sch{\mathrm{Sch}}
\def\Sets{\mathrm{Sets}}

\def\LL{\mathbf{L}}

\def\RR{\mathbf R}

\newcommand\qcoh{\ensuremath{\mathrm{qcoh}}}

\newcommand\Ocal{\ensuremath{\mathcal{O}}}
\newcommand\Xcal{\ensuremath{\mathcal{X}}}
\newcommand\Ycal{\ensuremath{\mathcal{Y}}}

\newcommand\bounded{\ensuremath{\mathrm{b}}}
\newcommand\Et{\ensuremath{\mathrm{\textnormal{\'{e}t}}}}

\DeclareMathOperator\derived{\mathbf{D}}

\DeclareMathOperator{\QCoh}{Qcoh}
\DeclareMathOperator\Qcoh\QCoh %
\DeclareMathOperator\coh{coh}

\DeclareMathOperator{\Perf}{Perf}
\DeclareMathOperator{\cone}{cone}

\DeclareMathOperator\leftmutation{\mathbb{L}}
\DeclareMathOperator\rightmutation{\mathbb{R}}

\DeclareMathOperator{\Sym}{Sym}

\DeclareMathOperator{\Aff}{Aff}
\DeclareMathOperator{\Alg}{Alg}
\DeclareMathOperator{\Mod}{Mod}

\DeclareMathOperator{\op}{op}

\DeclareMathOperator{\Spec}{Spec}

\DeclareMathOperator{\Cone}{Cone}
\DeclareMathOperator{\im}{im}

\DeclareMathOperator{\Aut}{Aut}

\DeclareMathOperator\hocolim{hocolim}
\DeclareMathOperator\holim{holim}

\DeclareMathOperator{\Hom}{Hom}

\DeclareMathOperator{\RRHom}{\mathbf{R}Hom}
\DeclareMathOperator{\RRlHom}{\mathbf{R}\kern-0.025em\mathscr{H}\kern-0.3em\textit{om}}
\DeclareMathOperator{\lHom}{\mathscr{H}\kern-0.3em\textit{om}}

\DeclareMathOperator{\End}{End}

\DeclareMathOperator{\pr}{pr}

\renewcommand{\geq}{\geqslant}
\renewcommand{\ge}{\geqslant}
\renewcommand{\leq}{\leqslant}
\renewcommand{\le}{\leqslant}

\declaretheoremstyle[spaceabove = 3pt, spacebelow = 3pt, bodyfont = \itshape]{theorem}
\declaretheoremstyle[spaceabove = 3pt, spacebelow = 3pt]{remark}

\declaretheorem[style=theorem, numberwithin=section]{theorem}

\declaretheorem[style=theorem, sibling=theorem]{corollary}
\declaretheorem[style=theorem, sibling=theorem]{conjecture}
\declaretheorem[style=theorem, sibling=theorem]{lemma}
\declaretheorem[style=theorem, sibling=theorem]{problem}
\declaretheorem[style=theorem, sibling=theorem]{proposition}
\declaretheorem[style=theorem, sibling=theorem]{question}
\declaretheorem[style=theorem, sibling=theorem]{criterion}

\declaretheorem[style=remark, sibling=theorem]{definition}
\declaretheorem[style=remark, sibling=theorem]{example}
\declaretheorem[style=remark, sibling=theorem]{remark}
\declaretheorem[style=remark, sibling=theorem]{situation}
\declaretheorem[style=remark, sibling=theorem]{notation}

\declaretheorem[style=theorem, numberwithin=section, title=Theorem]{alphatheorem}
\declaretheorem[style=theorem, sibling=alphatheorem, title=Conjecture]{alphaconjecture}

\crefname{alphatheorem}{Theorem}{Theorems}
\crefname{alphaconjecture}{Conjecture}{Conjectures}
\crefname{alphacorollary}{Corollary}{Corollaries}
\crefname{alphaproposition}{Proposition}{Propositions}

\usepackage{tikz}
\usepackage{tikz-cd}
\usepackage{rotating}
\newcommand*{\isoarrow}[1]{\arrow[#1,"\rotatebox{90}{\(\sim\)}"
]}
\usetikzlibrary{matrix,shapes,arrows,decorations.pathmorphing}

\tikzset{commutative diagrams/.cd,
mysymbol/.style={start anchor=center,end anchor=center,draw=none}}
\newcommand\MySymb[2][\square]{%
\arrow[mysymbol]{#2}[description]{#1}}
\tikzset{
  shift up/.style={
    to path={([yshift=#1]\tikztostart.east) -- ([yshift=#1]\tikztotarget.west) \tikztonodes}
  }
}

\DeclareMathAlphabet{\mathpzc}{OT1}{pzc}{m}{it}

\newcommand*{\defeq}{\mathrel{\vcenter{\baselineskip0.5ex \lineskiplimit0pt
  \hbox{\scriptsize.}\hbox{\scriptsize.}}}%
=}

\usepackage[
  final
]{showlabels}

\makeatletter
\def\SL@eqntext#1{\rlap{\quad{\showlabelsetlabel{\SL@prlabelname{#1}}}}}
\makeatother

\DeclareMathOperator{\cdec}{\mathcal{DEC}}
\DeclareMathOperator{\dec}{\mathsf{DEC}}
\DeclareMathOperator{\sod}{\mathsf{SOD}}
\DeclareMathOperator\MOR{\mathsf{MOR}}
\DeclareMathOperator{\Sh}{\mathsf{Sh}}
\newcommand{\Auteq}{\operatorname{Auteq}}
\newcommand{\Pic}{\operatorname{Pic}}
\newcommand{\Br}{\operatorname{Br}}

\newcommand{\Schet}[1]{\left( \Sch _{ #1 } \right) _{ \Et } }
\newcommand{\Affet}[1]{\left( \Aff _{ #1 } \right) _{ \Et } }
\newcommand{\sodl}[1]{\mathsf{SOD} ^{ #1 }}
\newcommand{\ntsodl}[1]{\mathsf{ntSOD} ^{ #1 }}
\newcommand{\excl}[1]{\mathsf{Exc}^{#1}}

\newcommand{\am}{\mathop{\mathsf{am}}\nolimits}
\newcommand{\Presheaf}{\mathsf{PSh}}
\newcommand{\id}{\operatorname{id}}
\newcommand{\simto}{\,\widetilde{\to}\,}

\newcommand{\cA}{\ensuremath{\mathcal{A}}}
\newcommand{\cB}{\ensuremath{\mathcal{B}}}
\newcommand{\cC}{\ensuremath{\mathcal{C}}}
\newcommand{\cD}{\ensuremath{\mathcal{D}}}
\newcommand{\cE}{\ensuremath{\mathcal{E}}}
\newcommand{\cF}{\ensuremath{\mathcal{F}}}

\newcommand{\cI}{\ensuremath{\mathcal{I}}}

\newcommand{\cL}{\ensuremath{\mathcal{L}}}

\newcommand{\cO}{\ensuremath{\mathcal{O}}}

\newcommand{\cT}{\ensuremath{\mathcal{T}}}

\newcommand{\cV}{\ensuremath{\mathcal{V}}}

\newcommand{\cX}{\ensuremath{\mathcal{X}}}
\newcommand{\cY}{\ensuremath{\mathcal{Y}}}
\newcommand{\cZ}{\ensuremath{\mathcal{Z}}}

\newcommand{\ttilde}{\tilde{t}}

\newcommand{\bA}{\ensuremath{\mathbb{A}}}

\newcommand{\bC}{\ensuremath{\mathbb{C}}}

\newcommand{\bG}{\ensuremath{\mathbb{G}}}

\newcommand{\bP}{\ensuremath{\mathbb{P}}}

\newcommand{\bZ}{\ensuremath{\mathbb{Z}}}

\newcommand{\bfk}{\mathbf{k}}

\usepackage{todonotes}
\setuptodonotes{inline}

\definecolor{caribbeangreen}{rgb}{0.0, 0.8, 0.6}
\newcommand\pieter[1]{
  \begin{color}{caribbeangreen}\textbf{Pieter}: #1
\end{color}}
\newcommand\andrea[1]{
  \begin{color}{blue}\textbf{Andrea}: #1
\end{color}}

\definecolor{brightpink}{rgb}{1.0, 0.0, 0.5}

\title[Moduli of semiorthogonal decompositions]{Moduli spaces of semiorthogonal decompositions in families}
\author[Pieter Belmans \and Shinnosuke Okawa \and Andrea T.\ Ricolfi]{Pieter Belmans \and Shinnosuke Okawa \and Andrea T.\ Ricolfi}

\thanks{version from: \today}

\begin{document}

\begin{abstract}
  To a smooth and proper morphism~$\cX\to U$ with quasicompact semiseparated target
  we associate a sheaf in the \'etale topology, which takes an affine $U$-scheme $V$ to the set
  of $V$-linear semiorthogonal decompositions (of fixed length)
  of the category $\Perf \cX_{V}$.
  We use Artin's criterion to prove that, when $U$ is excellent, this is in fact an algebraic space
  which is moreover \'etale (though in general non-quasicompact and non-separated) over $U$.
  We moreover generalise the construction of the sheaf to families of geometric noncommutative schemes in the sense of Orlov.
  We also define a subfunctor classifying nontrivial semiorthogonal decompositions,
  and conjecture it is an open and closed subspace.

  Along the way,
  we prove that
  for a smooth and proper family of schemes,
  a semiorthogonal decomposition of the bounded derived category of coherent sheaves of a fibre
  uniquely deforms over an \'etale neighbourhood of the point.
\end{abstract}

\maketitle

{
  \hypersetup{linkcolor=black}
  \setcounter{tocdepth}{2}
  \tableofcontents
}

\section{Introduction}

The aim of this paper is to study the behaviour of semiorthogonal decompositions
in smooth and proper families of varieties.
A survey of recent results in this context is provided by
Kuznetsov's 2022 ICM address \cite{MR4680279}.
The main result of this paper is
the construction of a moduli space for semiorthogonal decompositions,
and a description of its geometric properties.

Semiorthogonal decompositions provide a fundamental method to understand the structure of (enhanced) triangulated categories,
in particular in the setting of algebraic geometry and noncommutative algebra.
In the absolute setting,
this idea goes back to Beilinson's description of the derived category of~$\mathbb{P}^n$
using an exceptional collection \cite{MR0509388}.
It was generalised in \cite{MR0992977,MR1039961,alg-geom/9506012}
to allow semiorthogonal decompositions with more complicated components.

\subsection{Context and motivation}
Before we describe the main results in \cref{subsection:results},
we will first discuss a prototypical example from birational geometry
which illustrates the expected behaviour of semiorthogonal decompositions in smooth families,
and briefly survey their role in algebraic geometry.

\subsubsection{\textbf{The geometry of} \texorpdfstring{$(-1)$}{-1}\textbf{-curves}}
\label{sec:-1curves}
By Castelnuovo's contraction theorem,
a $(-1)$-curve on a surface is the exceptional curve of the blowup of another surface at a smooth point,
and hence can be contracted.
This fact constitutes the foundation of the birational classification of algebraic surfaces.
A classical example in this setting is that of a cubic surface: over an algebraically closed field,
it contains precisely~27~lines, and these are all the~$(-1)$-curves on it.

If we consider a versal \emph{family}~$\mathcal{X}\to U$ of smooth cubic surfaces (over an algebraically closed field),
we can upgrade the set of~$(-1)$-curves to a \emph{moduli space of lines in the fibres} of the family,
or relative Fano scheme of lines,
which will be denoted by
\begin{equation}
  \label{equation:fano-scheme-of-lines}
  \begin{tikzcd}
    \mathcal{F}\arrow{r} & U.
  \end{tikzcd}
\end{equation}
This is a finite \'etale cover of schemes, of degree~27, reflecting the~27~lines on a cubic surface.
A celebrated classical result is that the monodromy group of this covering space is
a subgroup of the Weyl group of type~$\mathrm{E}_6$,
which is responsible for all the different realisations of a cubic surface as blowups of~$\mathbb{P}^2$
in 6 points in general position \cite{MR0552521}.

As the structure sheaf of a $(-1)$-curve is an exceptional object of the bounded derived category of coherent sheaves of a surface,
it induces a nontrivial semiorthogonal decomposition.
The covering space \eqref{equation:fano-scheme-of-lines} will therefore be naturally included as
a connected component of the moduli space of semiorthogonal decompositions associated to the family~$\cX\to U$, after a potential \'etale base change of $U$ killing the monodromy.
We will come back to this point in \Cref{rmk:comments}\ref{item:not-zariski} and more extensively in \Cref{section:del-pezzo-families}.

\subsubsection{\textbf{Semiorthogonal decompositions in algebraic geometry}}
The study of $(-1)$-curves
is a precursor of the minimal model program.
The DK-hypothesis \cite{MR1949787} asserts that there is a close relationship between operations in this program,
and semiorthogonal decompositions of triangulated categories.
It claims that,
given a birational map~$X\dasharrow Y$ between smooth projective varieties~$X$ and~$Y$,
and an equality~$\mathrm{K}_X=\mathrm{K}_Y$
(resp.~inequality~$\mathrm{K}_X>\mathrm{K}_Y$) of canonical divisors,
there is an equivalence of categories~$\derived^\bounded(X)\simeq\derived^\bounded(Y)$
(resp.~a semiorthogonal decomposition of~$\derived^\bounded(X)$ in terms of~$\derived^\bounded(Y)$ and some complement).

A typical example of a birational map in the minimal model program is a divisorial contraction (resp.~a flip, or a flop).
Orlov's paper \cite{MR1208153} studied the case of smooth blowups and projective bundles,
thereby initiating the whole subject.
The behaviour of derived categories for standard flips and flops
is studied in~\cite{alg-geom/9506012}.
See \cite{MR2483950} for an older survey on further developments,
and \cite{MR4524222,MR3918435,MR4362776} for recent developments.

\medskip

There also exist semiorthogonal decompositions which do not originate from the minimal model program.
A typical example is the \emph{Kuznetsov component} for a Fano variety $X$ of index~$\mathrm{i}_X$,
which appears as the semiorthogonal complement of the exceptional collection~$\cO_X, \cO_X(1),\ldots,\cO_X(\mathrm{i}_X-1)$
in~$\derived^\bounded(X)$.
In \cite{alg-geom/9506012} the example of the intersection of two quadrics was studied;
its geometry is controlled by the geometry of the associated hyperelliptic curve
appearing as the Kuznetsov component.
Similarly, the relationship between the geometry of cubic fourfolds and the properties of its associated Kuznetsov category
are a topic of current interest \cite{MR4292740,MR2605171}.

The notion of semiorthogonal decomposition also plays an important role in other areas,
such as matrix factorisations \cite{MR2101296}, or Fukaya-type categories.
They also play a prominent role in mirror symmetry,
by virtue of the generalisation of Dubrovin's conjecture to semiorthogonal decompositions given in \cite{MR4105948}.

\subsection{Results}
\label{subsection:results}
We construct a moduli space for semiorthogonal decompositions (of fixed length) and describe its geometry,
for arbitrary families of smooth proper varieties
(and more generally geometric noncommutative schemes as defined in \cite{MR3545926}).

\begin{alphatheorem}[{\Cref{proposition:description-sod-f}} and \Cref{theorem:sod-f}]
  \label{theorem:main-intro}
  Let~$U$ be a quasicompact and semiseparated scheme,
  and let~$f\colon\cX\to U$ be a smooth and proper morphism of schemes. Fix an integer $\ell \geq 2$. There exists a sheaf~$\sod_f^{\ell}$ on the big \'etale site $\Schet{U}$, such that
  \begin{enumerate}
    \item
      \label{enumerate:main-intro-values}
      for all quasicompact and semiseparated~$U$-schemes~$V$, there is a natural bijection
      \begin{equation}
        \label{equation:values}
        \sod_f^{\ell}(V\to U) \cong
        \Set{
          \begin{array}{c}
            \text{$V$-linear~semiorthogonal} \\
            \text{decompositions~}\Perf \cX_V = \langle \cA^{1},\ldots,\cA^{\ell} \rangle
          \end{array}
        };
      \end{equation}
    \item
      \label{enumerate:main-intro-etale}
      if~$U$ is moreover an excellent scheme,\footnote{
        Excellent rings are defined in
        \cite[\href{https://stacks.math.columbia.edu/tag/07QT}{Tag 07QT}]{stacks-project},
        and a scheme is excellent if there exists an affine open cover by spectra of excellent rings.
        It is a strengthening of the notion of a noetherian ring,
        particularly well-behaved with respect to completion.
        Many rings are excellent,
        cf.~\cite[\href{https://stacks.math.columbia.edu/tag/07QW}{Tag 07QW}]{stacks-project},
        in particular, schemes of finite type over a field are excellent.
      }
      then~$\sod_f^{\ell}$ is an algebraic space
      which is \'etale over~$U$.
  \end{enumerate}
\end{alphatheorem}

The $V$-linearity condition means that the subcategories $\mathcal{A}^{i} \subseteq \Perf \cX_{V}$
are closed under tensor products with perfect complexes pulled back from $V$,
which is one of the essential technical conditions ensuring
the existence of base change for semiorthogonal decompositions (cf.~\Cref{section:base-change}).

The proof of \Cref{theorem:main-intro} consists of checking Artin's axioms
characterising algebraic spaces that are \'etale over the base,
which we take in the form of the following criterion,
corresponding to
the algebraic space version of \cite[Theorem~11.3]{MR3951580},
and originally to \cite[Th\'eor\`eme VII.7.2]{MR0407011}.

\begin{criterion}[Artin's criterion for an \'etale algebraic space]
  \label{criterion:hall-rydh}
  Let $U$ be an excellent scheme.
  Let $\mathsf{P}$ be a presheaf over $\Sch_U$.
  Then $\mathsf{P}$ is an algebraic space which is \'etale over $U$
  if and only if it satisfies the following conditions:
  \begin{enumerate}
    \item
      \label{item:etale-sheaf}
      $\mathsf{P}$ is a sheaf on the big \'etale site $(\Sch_U)_{\Et}$.
    \item
      \label{item:limit-preserving}
      $\mathsf{P}$ is limit-preserving.
    \item
      \label{item:DEF}
      Let $(R,\mathfrak m)$ be a local noetherian ring which is $\mathfrak m$-adically complete,
      with structure morphism $\Spec R \to U$,
      such that the induced morphism $\Spec R/\mathfrak m \to U$ is of finite type.
      Then the canonical map
      \begin{equation}
        \begin{tikzcd}
          \mathsf{P}(\Spec R\to U) \arrow{r} &  \mathsf{P}(\Spec R/\mathfrak m\to U)
        \end{tikzcd}
      \end{equation}
      is bijective.
  \end{enumerate}
\end{criterion}

This criterion gives us a roadmap for the proof of \Cref{theorem:main-intro}.
The base change results we obtain in \Cref{section:base-change}
(generalising Kuznetsov's base change theorem \cite[Proposition~5.1]{MR2801403})
allow us to introduce the \emph{presheaf} of semiorthogonal decompositions in \Cref{definition:presheaf-sod-f}.
We settle \cref{item:etale-sheaf} of \Cref{criterion:hall-rydh}
by invoking \cite[Theorem 1.4]{MR4205113} (cf.~\Cref{theorem:sod-f-fpqc-sheaf}).
The description \eqref{equation:values} is proved in \Cref{proposition:description-sod-f}.
The fact that $\sod_f^{2}$ is limit-preserving (cf.~\Cref{definition:limit-preserving}) is proved in \Cref{theorem:sod-f-lfp},
which confirms \Cref{item:limit-preserving} above, limited to the case $\ell=2$.
The uniqueness of deformations of semiorthogonal decompositions, cf.~\Cref{item:DEF},
is \Cref{theorem:deforming-sods}, again in the case $\ell=2$. An inductive argument concluding the proof of \Cref{theorem:main-intro} for an arbitrary $\ell$ is given, finally, in
\Cref{theorem:sod-f}.
\medskip

Combining \Cref{theorem:sod-f-lfp} and \Cref{theorem:deforming-sods}
with a version of the Artin approximation theorem (see \Cref{theorem:artin-approximation}),
we obtain along the way the following result of independent interest,
which says that semiorthogonal decompositions in a fibre uniquely deform to the total space of a smooth family.

\begin{alphatheorem}[{\Cref{corollary:geometric-deformation}}]
  \label{theorem:geometric-deformation}
  Let~$U$ be a semiseparated scheme,
  which is moreover of finite type over a field~$\bfk$,
  and let~$f\colon\cX\to U$ be a smooth and proper morphism of schemes.
  Fix a $\bfk$-valued point~$0\in U(\bfk)$
  and denote~$X=f^{-1}(0)$.
  Assume that the derived category of $X$ admits a $\bfk$-linear semiorthogonal decomposition
  \begin{equation}
    \label{equation:sod-central-fibre-introduction}
    \derived^\bounded(X) = \braket{\cA,\cB}.
  \end{equation}
  Then, shrinking $U$ to an \'etale neighbourhood of $0$ if necessary,
  there exists a unique $U$\dash linear semiorthogonal decomposition
  \begin{equation}
    \Perf \Xcal = \braket{\cA_U,\cB_U}
  \end{equation}
  whose base change to $\Spec \bfk(0) \to U$ is the initial semiorthogonal decomposition \eqref{equation:sod-central-fibre-introduction}.
\end{alphatheorem}
In the main body of the text we state a more general version of this result, cf.~\Cref{corollary:geometric-deformation},
only requiring the local ring $\mathcal O_{U,0}$ to be a G-ring.

\begin{remark}
  \label{rmk:comments}
  A few comments are in order:
  \begin{enumerate}
    \item \Cref{enumerate:main-intro-etale} in \Cref{theorem:main-intro} does not extend at all to
      the unbounded derived category of quasicoherent sheaves,
      see \Cref{example:topological-sod},
      so the properness condition plays a significant role.
    \item Also, \Cref{theorem:main-intro} fails in general without the smoothness assumption,
      see \Cref{example:cannot-omit-smoothness}.
      When we drop the smoothness assumption on $ f \colon \cX \to U $,
      the structure sheaf of the diagonal~$\cO_{ \Delta_f }$
      is no longer a perfect complex,
      and this is an important property to study
      the moduli space of semiorthogonal decompositions,
      cf.~\cref{section:sod-is-dec}.

      It would be interesting to see what can be said if we instead consider
      semiorthogonal decompositions of~$\derived^\bounded(\coh \cX)$,
      when~$\cX\to U$ is no longer smooth.
    \item It is not clear whether or not the moduli space $\sod_f^{\ell}$ is in fact a scheme in general.
      As pointed out in \Cref{eg:OU} it can happen that
      the structure morphism $\sod_f^{\ell} \to U$ is neither separated nor quasicompact,
      even after restriction to a single connected component.
      Separatedness is a common ingredient in criteria such as \cite[Corollary~II.6.17]{MR0302647}
      to check whether an algebraic space is in fact a scheme.
      This nonseparatedness is a new phenomenon,
      which is not present in the example of~$(-1)$-curves,
      as the relative Fano scheme is proper over the base
      (cf.~\Cref{section:del-pezzo-families}).

      We do, however, prove that $\sod_{f}^{\ell}$ is a \emph{quasi}separated locally noetherian algebraic space,
      with dense open schematic locus, cf.~\Cref{proposition:quasi-separated}.
    \item
      \label{item:valuative} Even though $\sod_f^{\ell}$ is not separated,
      in \Cref{question:valuative} we ask\footnote{
        During the revision of this paper,
        a negative answer to \cref{question:valuative} appeared in the work of Wu, see \cite{Weimufei}.
        We kept \Cref{question:valuative} as originally formulated,
        as it was the motivating question for op.~cit.
      }
      whether the existence part of the valuative criterion for properness is still satisfied for our moduli space.
      This is the second half of the criterion,
      and in fact holds for $(-1)$-curves in the context of \Cref{sec:-1curves},
      as the relative Fano scheme is proper over the base.
    \item
      \label{item:not-zariski}
      One cannot take the neighbourhood in \cref{theorem:geometric-deformation}
      to be Zariski open in general, since there can be monodromy,
      as we observed in the cubic surface example of \Cref{sec:-1curves}.
      Another instance of this phenomenon is given by \Cref{example:zariski-open-not-possible} below.
  \end{enumerate}
\end{remark}

\begin{example}
  \label{example:zariski-open-not-possible}
  Take~$U = \Spec \bfk [ a, a^{ - 1 } ]$,
  and consider the following family of smooth quadric surfaces
  \begin{equation}
    \begin{tikzcd}
      f \colon
      \cX
      =
      (xy + z^2-aw^2 = 0)
      \subset
      \mathbb{P}^3_{x:y:z:w} \times U
      \arrow{r}{\pr_2} &
      U.
    \end{tikzcd}
  \end{equation}
  Consider the fibre~$\cX_1$
  at~$a=1$,
  and the divisor~$D_1 \defeq \left(x = z - w = 0\right) \subset \cX_1$.
  If the exceptional object~$\cO_{\cX_1}(D_1)$
  extends to an object of~$\Perf \cX$
  which is exceptional relative to~$U$,
  then it has to be a line bundle on~$\cX$,
  since every geometric fibre of~$f$ is isomorphic to~$\bP^1\times\bP^1$
  and every exceptional object of rank~$1$ on it is known to be a line bundle \cite[Propositions~2.9 and~2.10]{MR1286839}.

  However, no line bundle on~$\cX$ restricts to~$\cO_{\cX_1}(D_1)$.
  To see this, consider the \emph{irreducible} divisor~$D \defeq \cX \cap \left(x = 0\right) \subset \cX$.
  As the divisor class group of~$\cX \setminus D \cong \bA^2 \times U$
  is trivial,
  it follows that every line bundle on~$\cX$ is proportional to~$\cO_{\cX}(D)$.
  But no power of~$\cO_{\cX}(D)|_{\cX_1}$
  is isomorphic to~$\cO_{\cX_1}(D_1)$.
  Geometrically speaking,
  this comes from the fact that the monodromy around the puncture of~$U$
  exchanges the two rulings of the surface~$\cX_1\cong\bP^1\times\bP^1$.

  On the other hand, base-changing the family by the \'etale double cover given by~$a=b^2$,
  one can kill the monodromy and obtain the line bundle on~$\cX$ given by the divisor
  \begin{equation}
    \left( x = z - b w = 0 \right) \subset \cX,
  \end{equation}
  which does restrict to~$\cO_{\cX_1}(D_1)$.
\end{example}

Alternatively, one can consider a conic bundle which has only smooth fibres but which is not Zariski-locally trivial:
all fibres are isomorphic to~$\mathbb{P}^1$, but the existence of a relative exceptional collection
would force the conic bundle to be Zariski-locally trivial
(see also \cite[Remark~3.8]{1805.04050v4} for the higher-dimensional situation).

\subsection{Relation to earlier works}
The easiest semiorthogonal decompositions are those induced by full exceptional collections.
All the subcategories are then equivalent to~$\derived^\bounded(\bfk)$.
The deformation theory of exceptional objects goes back a long way:
it provided the intuition and inspiration for the results in \cite{MR1230966},
using the expectation that exceptional objects uniquely lift to deformations.
This intuition was made precise for infinitesimal deformations of abelian categories
(and their derived categories) in \cite{MR2175388},
and for formal deformations of abelian categories
(and their derived categories) in \cite{MR3050709}.

In a strictly geometric context these liftability results are explained by
the deformation theory results in \cite{MR1966840,MR2177199}.
More recently, in \cite[Theorem~3.5]{1805.04050v4}
(see \Cref{remark:hu-comparison})
and \cite{galkin-shinder-exceptional}
it was shown that \'etale-locally in a family
one can extend exceptional collections.
\Cref{theorem:geometric-deformation} is a generalisation of this property
to semiorthogonal decompositions.

Another precursor to \Cref{theorem:geometric-deformation},
not just for exceptional collections,
was studied in \cite{MR3764066}.
It takes as the base~$U$ the moduli stack~$\mathcal{M}_g$ of genus~$g$ curves,
and the family~$\mathcal{X}$
is the relative moduli space of rank~2 vector bundles with odd determinant,
and the point~$0\in U$ is any point in the hyperelliptic locus.
The semiorthogonal decomposition of the fibre in~$0$, corresponding to a hyperelliptic curve $C$,
is the decomposition given by the universal vector bundle,
which is shown to give
a fully faithful functor~$\derived^\bounded(C)\hookrightarrow\derived^\bounded(\mathrm{M}_C(2,\mathcal{L}))$
when~$C$ is hyperelliptic using an explicit geometric construction of the moduli space.
\Cref{theorem:geometric-deformation} generalises this example to arbitrary families
and arbitrary semiorthogonal decompositions
where the components are not necessarily derived categories themselves.

\subsection{Amplifications and conjectures}
\label{subsection:amplifications-intro}
Having constructed $\sod_f^{\ell}$,
it is possible to bootstrap, and construct
the moduli space $\sodl{\ell}_{\cB}$ of semiorthogonal decompositions
of a given $U$-linear admissible subcategory $\cB\subseteq\Perf\cX$,
such that~$\sodl{\ell}_f=\sodl{\ell}_{\Perf\cX}$.
This brings us in the setup of Orlov's geometric noncommutative schemes \cite{MR3545926},
and it is worked out in \cref{subsection:relative-to-subcategory}.

The generalisation of \Cref{theorem:main-intro} to $\sodl{\ell}_{\cB}$ given in \Cref{corollary:subcategory-space},
working relative to a subcategory~$\cB\subseteq\Perf\cX$,
suggests the following conjecture
for arbitrary smooth and proper dg categories (and not just geometric dg categories).

\begin{alphaconjecture}
  \label{conjecture:smooth-proper-dg-family}
  Let~$\mathcal{D}$ be a family of smooth and proper dg~categories over a
  quasicompact, semiseparated, excellent scheme~$U$.
  Let $\ell \geq 2$ be an integer.
  Then there exists an \'etale algebraic space $ \sod^{\ell}_{ \cD / U} \to U $ with a functorial bijection
  \begin{equation}
    \sod^{\ell}_{ \cD / U } ( V \to U ) \cong
    \Set{
      \begin{array}{c}
        \text{$V$-linear~semiorthogonal} \\
        \text{decompositions~}\Perf  \cD_V = \langle \cA^{1},\ldots,\cA^{\ell} \rangle
      \end{array}
    }
  \end{equation}
  for quasicompact and semiseparated schemes $V \to U$.
\end{alphaconjecture}
The case where~$\mathcal{D}=\Perf\cX$ (considering the necessarily unique dg enhancement) is \Cref{theorem:main-intro}.
The case where~$\mathcal{D}\subseteq\Perf\cX$ (considering any dg enhancement) is \cref{corollary:subcategory-space}.
Note that the constructions are independent of the choice of enhancement,
because~$\Perf\cD$ is considered as a triangulated category.

Proving \Cref{conjecture:smooth-proper-dg-family} will likely require some techniques from derived algebraic geometry,
as in \cite{MR4205113},
and it is related to the construction of the moduli space of objects in a dg category from \cite{MR2493386}.
More generally one would like to take the base~$U$
in \Cref{conjecture:smooth-proper-dg-family}
not necessarily an excellent scheme,
and merely a sufficiently nice algebraic space, or even an algebraic stack.

\medskip

A second conjecture,
or rather, a pair of related conjectures,
is discussed in \Cref{subsection:nontrivial}.
It concerns a subfunctor of~$\sod_f^\ell$
(or~$\sod_{\cB}^\ell$)
parametrising only \emph{nontrivial} semiorthogonal decompositions,
defined functorially by \Cref{def:ntsod}:
\begin{itemize}
  \item In \Cref{conjecture:ntsod-is-etale}
    we conjecture that this subfunctor is in fact an open and closed algebraic subspace of~$\sod_f^{\ell}$.
    This would confirm the expectation that nontriviality of semiorthogonal decompositions
    is a Zariski-open condition in a smooth and proper family.
  \item In \Cref{conjecture:rouquier-for-local-noetherian}
    we phrase an equivalent conjecture for \Cref{conjecture:ntsod-is-etale},
    still within the context of semiorthogonal decompositions of triangulated categories.
    In \Cref{conjecture:dg-version} we propose an
    (a priori) even stronger dg~version of \Cref{conjecture:rouquier-for-local-noetherian}.
  \item In \Cref{conjecture:referee} we give a third equivalent description
    in terms of the properties of sections of the structure morphism.
\end{itemize}

\medskip

Finally,
the moduli spaces introduced so far are also the natural recipients of
an action by the auto-equivalence group,
resp.~the braid group (by mutation), as discussed in \Cref{subsection:group-actions}.
These group actions reflect many deep questions on properties of semiorthogonal decompositions,
such as the transitivity of the braid group action.

\subsection{Applications}
In \Cref{section:moduli} we also discuss various examples. The following applications will be discussed in this paper:
\begin{enumerate}
  \item In \Cref{proposition:field-of-definition}
    we deduce from \Cref{theorem:main-intro} a nontrivial assertion on
    ``the field of definition'' of semiorthogonal decompositions.
    This was pointed out by Alexander Efimov, and a simplified version can be stated as follows.
    Suppose that $X$ is a smooth proper variety defined over an algebraically closed field $\bfk$,
    and let $K/\bfk$ be a field extension.
    Then every semiorthogonal decomposition of $\derived^\bounded(X_K)$,
    where $X_K$ is the base change of $X$,
    is the base change of a semiorthogonal decomposition of $\derived^\bounded(X)$.
  \item In \Cref{remark:rigidity-of-autoequivalences}
    we deduce from \Cref{theorem:main-intro}
    the fact that semiorthogonal decompositions are rigid under the actions of topologically trivial autoequivalences,
    which is first shown in \cite{1508.00682v2}.
    The proof in \Cref{remark:rigidity-of-autoequivalences} is more conceptual than the original argument.
  \item
    In \Cref{section:del-pezzo-families} we expand on \Cref{sec:-1curves},
    by pinpointing the precise relationship between $\sod_f^\ell$ and the moduli of lines.
\end{enumerate}
In~\cite{2107.03051} the example of~$\mathbb{P}^1\times\mathbb{P}^1$
and its degeneration to the second Hirzebruch surface was discussed in detail, extending upon \Cref{eg:OU}.

\subsection{Relation to deformations of extremal contractions}
Suppose that $f \colon \cX \to U$
is a family of smooth projective varieties,
whose central fibre $\cX_0 = f^{-1}(0)$ admits an extremal contraction $g_0 \colon \cX_0 \to Y$
in the sense of the minimal model program.
The DK-hypothesis suggests that it should induce
a nontrivial semiorthogonal decomposition of $\derived^\bounded(\cX_0)$.
Our main result implies that, possibly after replacing $U$
with an \'etale neighbourhood of $ 0 \in U $,
it should deform to a $U$-linear semiorthogonal decomposition of $\Perf \cX$.
A naive guess here is that the extremal contraction $g_0$
should also deform  to a family of extremal contractions
$g\colon\cX\to\cY$ over $U$,
and that it corresponds to the deformation family of semiorthogonal decompositions via the DK-hypothesis.

Put differently, it is quite likely that one can associate to $f$
a moduli space of extremal contractions \'etale over $U$,
and promote the DK-hypothesis to a functorial morphism from that space to (an appropriate version of) $ \sod_f $.

In fact, for complex spaces it is shown in \cite[(12.3)]{MR1149195}
that extremal contractions \emph{do} deform.
Since the proof is based on a Kodaira-type vanishing theorem,
it is not clear if there are counterexamples in positive or mixed characteristics.
Taking into account that our main results are free of characteristics,
it is conceivable that such counterexamples (if any)
also could be counterexamples to the DK-hypothesis.
For example, note that one needs to use Kodaira vanishing
to prove that the structure sheaf of a Fano manifold in characteristic~$0$ is an exceptional object.

\subsection{Structure of the paper}
In \Cref{section:preliminaries} we collect various preliminaries on
derived categories of sheaves and semiorthogonal decompositions for the reader's convenience.
In \Cref{section:base-change} we discuss base change for semiorthogonal decompositions,
generalising Kuznetsov's results of \cite{MR2801403}.
An alternative discussion can be found in \cite[\S 3.2]{MR4292740}.
The main object of interest, namely the sheaf of semiorthogonal decompositions,
is introduced in \Cref{section:sod-functor-sheaf}.

To prove that this sheaf is an algebraic space,
which is moreover \'etale over the base,
we need to check two more conditions in \cref{criterion:hall-rydh}. The key step is the case of semiorthogonal decompositions of length $2$.
We give an alternative description of~$\sod_f^{2}$
which makes the deformation theory of semiorthogonal decompositions more explicit:
in \Cref{section:sod-is-dec} we introduce the presheaf of decompositions of the structure sheaf of the diagonal,
and show that this presheaf is isomorphic to the presheaf of semiorthogonal decompositions of length 2 (\Cref{theorem:Iso_Functors_SOD_DEC}).

Using this alternative description we show in \Cref{section:local-study} that these functors are limit-preserving.
In \Cref{section:deform_sod} we study the deformation theory of semiorthogonal decompositions,
and show that semiorthogonal decompositions for perfect complexes are unobstructed and have unique lifts.
This is now done using a proof suggested by the referee,
for the original approach (using the deformation theory of morphisms and the translation in \cref{section:sod-is-dec})
one is referred to \cite{former-appendix}.

In \Cref{section:moduli} we can then conclude from Artin's axioms that we indeed have an algebraic space which is \'etale over the base, completing the proof of \Cref{theorem:main-intro}.
We also discuss various examples showcasing how our results can be applied,
and why they cannot be extended to certain other settings.

In \cref{section:amplifications}
we give amplifications of our main construction.
First, we work relative to a fixed (possibly non-geometric) subcategory.
Next, we discuss a (conjectural) generalisation involving only nontrivial semiorthogonal decompositions.
As part of this discussion
we formulate a conjecture on the behaviour of dg categories over integral local noetherian rings
(see \Cref{conjecture:rouquier-for-local-noetherian} and \Cref{conjecture:dg-version}),
which would allow us to conclude as in \Cref{conjecture:ntsod-is-etale}
that there is an open and closed subspace of the moduli space $\sod_{f}^{\ell}$,
parametrising only nontrivial semiorthogonal decompositions.
Finally, we enhance the setup by equipping it with two interesting group actions,
one encoding mutation, and another induced by autoequivalences.

In \cref{section:del-pezzo-families} we revisit the example from \cref{sec:-1curves}.

\subsection*{Acknowledgements}
We want to thank Alexander Kuznetsov for the idea
that deformations of semiorthogonal decompositions can be studied
via the corresponding decomposition of the structure sheaf of the diagonal,
which was the starting point for our work.
We want to thank
Alexander Efimov for pointing out \Cref{proposition:field-of-definition},
Najmuddin Fakhruddin for pointing out \Cref{remark:fakhruddin},
and Isamu Iwanari for suggesting \Cref{example:topological-sod}.
We want to thank
Dmitry Kaledin, Sergey Galkin and Mauro Porta for interesting discussions.
We thank Federico Barbacovi for pointing out a mistake in an earlier version of the manuscript.
We want to thank both referees for the many useful comments and suggestions,
which have greatly improved the paper.
The positive impact one of the referees in particular has had on this paper cannot be overstated:
the alternative argument in \cref{subsection:lfp-using-generators}
and the argument currently used in \cref{subsection:extending-sods}
(with the original argument relegated to \cite{former-appendix})
are due to them.

The first author was partially supported by FWO (Research Foundation---Flanders)
during the initial phase of this project,
and by NWO (Dutch Research Council)
during the latter phase of this project,
as part of the grant \href{https://doi.org/10.61686/RZKLF82806}{doi:10.61686/RZKLF82806},
in the later stages.
The second author was partially supported by Grants-in-Aid for Scientific Research
(16H05994,
  16K13746,
  16H02141,
  16K13743,
  16K13755,
  16H06337,
  19KK0348,
  20H01797,
  20H01794,
  21H04994,
23H01074)
and the Inamori Foundation.
Finally, we want to thank the Max Planck Institute for Mathematics (Bonn) and SISSA for the pleasant working conditions.%

\section{Preliminaries}
\label{section:preliminaries}

First we recall some notation and foundational results.
\subsection{Schemes and derived categories of sheaves}
\label{sec:prelim-on-schemes}

Given a scheme $U$, we shall denote by $\Sch_U$ (resp.~$\Aff_U$) the category of schemes over $U$ (resp.~affine schemes over $U$). We denote by $\Alg_U \subset \Aff_U^{\op}$ the category whose objects are pairs $(A,h)$ where $A$ is a ring and $h \colon \Spec A \to U$ is an \emph{affine} morphism. We call such objects $U$-algebras.

A morphism of schemes $f\colon \cX \to U$ is \emph{semiseparated}
if the diagonal morphism $\iota_{\Delta_f} \colon \cX \to \cX \times_U\cX$ is affine.
A scheme is semiseparated if it is semiseparated over $\Spec \mathbb Z$,
which happens if and only if the intersection of any two open affines is affine.
Throughout we will work with semiseparated schemes;
some of the results we cite or prove
work for the more general notion of quasiseparated schemes,
but others do not.

\medskip

For a scheme $ X $, let $ \Mod \cO_X $ be the category of $ \cO_X $-modules
and $ \Qcoh X\subseteq \Mod \cO_X $ be the subcategory of quasicoherent modules.
We denote by $ \coh X\subseteq \Mod \cO_X $ the subcategory of coherent $\Ocal_X$-modules.
All these categories are abelian, see \cite[\href{https://stacks.math.columbia.edu/tag/01AG}{Tag 01AG},
  \href{https://stacks.math.columbia.edu/tag/077P}{Tag 077P},
\href{https://stacks.math.columbia.edu/tag/01BY}{Tag 01BY}]{stacks-project}.

Let $\derived ( \cO_X ) \defeq \derived ( \Mod \cO_X ) $ be the unbounded derived category of $ \Mod \cO_X
$,
equipped with its natural triangulated structure.
For $\ast \in \set{+,-,\bounded,\emptyset}$, we let $\derived^\ast(\Ocal_X)$ denote the full subcategory of $\derived ( \cO_X )$,
where the decoration $\ast$ indicates the location of the (possibly) nonvanishing cohomology sheaves.
We use the same convention for $\derived^\ast(\Qcoh X)$ and $\derived^\ast(X) \defeq \derived^\ast(\coh X)$.
By $\derived_{\qcoh}^\ast(\Ocal_X)\subseteq \derived^\ast(\Ocal_X)$, resp.~$\derived_{\coh}^\ast(\Ocal_X) \subseteq \derived_{\qcoh}^\ast(\Ocal_X)$,
we denote the full triangulated subcategory of complexes with quasicoherent, resp.~coherent cohomology. Finally, we denote by $\Perf X\subseteq \derived_{\coh}(\Ocal_{\mathcal X})$ the triangulated category of perfect complexes over $X$.

\medskip

A rigorous treatment of derived functors and their compatibilities
in the context of unbounded complexes between the derived categories $\derived(\Ocal_X)$
was carried out by Spaltenstein \cite{MR0932640}, and we refer the reader to op.~cit.~for full details.

Let $f \colon X \to Y$ be a morphism of schemes. Whenever no ambiguity could arise, the functors $f^\ast$, $f_\ast$ will denote their derived version, namely $\mathbf L f^\ast$, $\RR f_\ast$. Similarly, we shall write $-\otimes-$ instead of $-\otimes^\LL-$, and we reserve the notation $\RRlHom_f(-,-)$ for the composite functor $\RR f_\ast \RRlHom_X(-,-)$.

If $f\colon X\to Y$ is a quasicompact and quasiseparated (e.g.~semiseparated) morphism of schemes, then by \cite[\href{https://stacks.math.columbia.edu/tag/08D5}{Tag~08D5}]{stacks-project}, the derived pushforward $f_\ast\colon \derived(\Ocal_X) \to \derived(\Ocal_Y)$ preserves quasicoherent cohomology,
and thus defines an exact functor $f_\ast \colon \derived_{\qcoh}(\Ocal_X)\to \derived_{\qcoh}(\Ocal_Y)$.
The derived pullback $f^\ast\colon \derived(\Ocal_Y) \to \derived(\Ocal_X)$ preserves quasicoherent cohomology (with no assumptions on $f$), therefore it restricts to an exact functor $f^\ast \colon \derived_{\qcoh}(\mathcal O_Y) \to \derived_{\qcoh}(\mathcal O_X)$, and if $f$ is quasicompact and quasiseparated there is an adjunction isomorphism
\begin{equation}
  \label{adjunction}
  \Hom_{\derived_{\qcoh}(\mathcal O_X)}(f^{\ast}A,B) \cong \Hom_{\derived_{\qcoh}(\mathcal O_Y)}(A,f_{\ast}B)
\end{equation}
for objects $A \in \derived_{\qcoh}(\mathcal O_Y)$ and $B \in \derived_{\qcoh}^{+}(\mathcal O_X)$, see \cite[\href{https://stacks.math.columbia.edu/tag/07BD}{Tag~07BD}]{stacks-project}.

If $\cA$ and $\cB$ are subcategories of $\derived_{\qcoh}(\mathcal O_X)$ for a scheme $X$,
we will write $\RRlHom_X(\cB,\cA) = 0$ to mean that $\RRlHom_X(b,a) = 0$ for all objects $a \in \cA$ and $b\in \cB$.

We shall make free use of the \emph{projection formula},
which states that if $f\colon X \to Y$ is a morphism of schemes,
then for every $E \in \derived(\mathcal O_X)$ and $K \in \Perf Y$ one has
\begin{equation}
  f_\ast E \otimes_{\mathcal O_Y} K = f_\ast (E \otimes_{\mathcal O_X} f^\ast K)
\end{equation}
in $\derived(\mathcal O_Y)$, see \cite[\href{https://stacks.math.columbia.edu/tag/0B54}{Tag 0B54}]{stacks-project}.
A more general version for arbitrary complexes with quasicoherent cohomology can be found in \cite[Proposition~3.9.4]{MR2490557},
and will be used in \Cref{lemma:4.2}.

We next recall some known results on the identification of various triangulated categories of sheaves:
\begin{itemize}
  \item
    By \cite[\href{https://stacks.math.columbia.edu/tag/08DB}{Tag 08DB}]{stacks-project},
    if $X$ is quasicompact and semiseparated, the canonical functor
    \begin{equation}
      \begin{tikzcd}
        \derived(\Qcoh X) \arrow{r} & \derived_{\qcoh}(\Ocal_X)
      \end{tikzcd}
    \end{equation}
    is an exact equivalence (see also~\cite[Corollary~5.5]{MR1214458} for the proof in the case where $X$ is quasicompact and separated).
  \item
    By \cite[Corollary~II.2.2.2.1]{SGA6}, if $X$ is noetherian, the canonical functor
    \begin{equation}
      \begin{tikzcd}
        \derived^{\textrm b}(X)  \arrow{r} & \derived^{\textrm b}_{\coh}(\Ocal_X)
      \end{tikzcd}
    \end{equation}
    is an exact equivalence.
  \item
    By \cite[\href{https://stacks.math.columbia.edu/tag/0FDC}{Tag~0FDC}]{stacks-project},
    if $X$ is noetherian and regular, the canonical functor
    \begin{equation}
      \begin{tikzcd}
        \Perf X  \arrow{r} & \derived^{\textrm b}(X)
      \end{tikzcd}
    \end{equation}
    is an exact equivalence.
\end{itemize}

The following result will be important in encoding semiorthogonal decompositions
using morphisms of perfect complexes,
as in \cref{section:sod-is-dec}.
Let $f\colon \cX \to U$ be a separated morphism;
with a slight abuse of notation,
we will refer to $\mathcal O_{\Delta_f} \defeq \iota_{\Delta_f,\ast}\mathcal O_{\cX}$
as the \emph{structure sheaf of the diagonal}.

\begin{lemma}
  \label{lemma:Diagonal_Perfect}
  Let~$U$ be a quasicompact and semiseparated scheme.
  Let $f\colon \Xcal \to U$ be a smooth and separated morphism of schemes.
  Then $\mathcal O_{\Delta_f}$ is a perfect complex on $\Xcal \times_U \Xcal$.
\end{lemma}
We will apply this when~$f$ is in fact smooth and proper.

\begin{proof}
  The diagonal $\iota_{\Delta_f}\colon \Xcal \into \Xcal \times_U \Xcal$ is a closed immersion by the separatedness of $f$,
  and a locally complete intersection morphism by its smoothness,
  so in particular it is a proper perfect morphism.
  Then, by \cite[Example~2.2]{MR2346195},
  the direct image functor $\iota_{\Delta_f \ast}\colon \derived_{\qcoh}(\Ocal_\Xcal) \to \derived_{\qcoh}(\Ocal_{\Xcal \times_U \Xcal})$
  preserves perfect complexes.
  The result now follows since, by definition, $\Ocal_{\Delta_f}=\iota_{\Delta_f,\ast}\Ocal_{\Xcal}$.
\end{proof}

We will also appeal to the following derived versions of the Nakayama--Azumaya--Krull lemma.

\begin{lemma}[Derived Nakayama--Azumaya--Krull, local version]
  \label{lemma:derived-nakayama-local}
  Let $ ( R, \mathfrak m, \bfk ) $ be a local ring. Let
  \begin{equation}
    \begin{tikzcd}
      P^{ \bullet }=
      {[}
        \cdots \arrow{r} & P^{ k }  \arrow{r}{d^{ k}} & P^{ k + 1 }  \arrow{r} & \cdots
      {]}
    \end{tikzcd}
  \end{equation}
  be a bounded complex of finitely generated free $R$-modules.
  Then $P^{ \bullet } = 0$ in $\Perf R$ if and only if $P^{ \bullet } \otimes_R \bfk = 0$ in $\Perf \bfk$.
\end{lemma}

\begin{proof}
  The ``only if'' direction is obvious, so we only prove the contraposition of the other implication.
  Let $ i \in \bZ $ be the maximal index for which $ \mathcal H^i ( P^{ \bullet } ) \neq 0$,
  so that we may and will assume that $ P^{ j } = 0 $ for all $ j > i $.
  Then the map $d^{i-1}\colon P^{i-1}\to P^i$ is not surjective,
  which implies that $d^{i-1}\otimes_R \bfk\colon P^{i-1}\otimes_R\bfk\to P^i\otimes_R\bfk $
  is not surjective either by the usual Nakayama lemma.
  This implies $ P^\bullet \otimes_R \bfk \neq 0\in\Perf \bfk$.
\end{proof}

\begin{lemma}[Derived Nakayama--Azumaya--Krull, global version]
  \label{lemma:derived-nakayama-global}
  Let $ X $ be a scheme and $ F \in \Perf X$.
  Then $F=0$ in $\Perf X$ if and only if $ \iota_x^* F = 0$ in $\Perf \bfk ( x ) $ for all $x \in X$, where $\iota_{x}$ denotes the canonical morphism $\Spec \bfk (x) \to X$.
\end{lemma}

\begin{proof}
  Note that $ \iota_x $ has a factorisation
  \begin{equation}
    \begin{tikzcd}
      \Spec \bfk ( x )  \arrow{r} & \Spec \cO_{ X, x}  \arrow{r}{ j_x } & X.
    \end{tikzcd}
  \end{equation}
  Since $ j_x $ is flat we have that $F=0$ if and only if $j_x^* F = 0$ in $\Perf \cO_{ X, x }$ for all $x \in X$.
  Thus we have reduced the assertion to \Cref{lemma:derived-nakayama-local}.
\end{proof}

\subsection{Semiorthogonal decompositions}
\label{subsection:sod}

We will recall some important definitions from \cite{MR2801403} in this section.

Let $\mathcal T$ be a triangulated category. Given a class of objects $\cA\subseteq \mathcal T$, its \emph{right} and \emph{left orthogonal}
\begin{equation}
  \begin{aligned}
    \cA^\perp
    &= \Set{T \in \mathcal T\mid\Hom_{\mathcal T}(\cA[k],T) = 0\textrm{ for all }k \in \mathbb Z} \\
    \prescript{\perp}{}{\mathcal{A}}
    &= \Set{T \in \mathcal T\mid\Hom_{\mathcal T}(T,\cA[k]) = 0\textrm{ for all }k \in \mathbb Z}
  \end{aligned}
\end{equation}
define triangulated subcategories of $\mathcal T$, closed under taking direct summands.

Two classes of objects $\cA$, $\cB\subseteq \mathcal T$ are called \emph{semiorthogonal} if $\cA\subseteq \cB^\perp$ (which is equivalent to $\cB\subseteq \prescript{\perp}{}{\mathcal{A}}$).
A subcategory $\cA\subset \mathcal T$ is \emph{strictly full} if, whenever an object $a \in \cA$ is isomorphic to an object $t \in \mathcal T$, one has that $t \in \mathcal{A}$  \cite[\href{https://stacks.math.columbia.edu/tag/001D}{Tag 001D}]{stacks-project}.

\begin{definition}\label{def:SOD}
  A \emph{semiorthogonal decomposition of length $\ell$} of a triangulated category $\mathcal T$ is a finite sequence $\cA^1,\ldots,\cA^\ell$ of strictly full triangulated subcategories of $\mathcal T$, such that
  \begin{enumerate}
    \item $\cA^i \subseteq \cA^{j,\perp}$ for $i<j$, and
    \item for every object $T \in \mathcal T$ there exists a sequence of morphisms
      \begin{equation}
        \label{equation:sod-sequence}
        \begin{tikzcd}
          0 = T_\ell \arrow{r} & T_{\ell-1} \arrow{r} &  \cdots \arrow{r} &  T_1 \arrow{r} &  T_0 = T
        \end{tikzcd}
      \end{equation}
      such that $\mathrm{Cone}(T_k \to T_{k-1}) \in \cA^k$ for all $k=1,\ldots,\ell$.
  \end{enumerate}

  When a sequence $\cA^1,\ldots,\cA^\ell \subseteq \mathcal T$ of strictly full triangulated subcategories defines a semiorthogonal decomposition of $\mathcal T$, we write
  \begin{equation}
    \mathcal T = \braket{\cA^1,\ldots,\cA^\ell}.
  \end{equation}
\end{definition}
We use superscripts, rather than the more conventional subscripts, to reserve the latter for the base change notation that will be introduced later.

\begin{lemma}[{\cite[Lemma 2.3]{MR2801403}}]
  \label{lemma:uniqueness-decomposition-triangles}
  Let $ \cT = \langle \cA, \cB \rangle $ be a semiorthogonal decomposition of a triangulated category~$ \cT $ and take an object $ E \in \cT$.
  Let
  \begin{equation}
    \begin{tikzcd}[row sep=tiny]
      b_1 \arrow{r}{s_1} & E \arrow{r}{t_1} & a_1 \arrow{r}{u_1} & b_1[1] \\
      b_2 \arrow{r}{s_2} & E \arrow{r}{t_2} & a_2 \arrow{r}{u_2} & b_2[1]
    \end{tikzcd}
  \end{equation}
  be two distinguished triangles such that $ a_i \in \cA $ and $ b_i \in \cB $ for $ i = 1, 2$.
  Then there is exactly one pair of isomorphisms
  \begin{equation}
    \begin{tikzcd}
      \beta \colon b_1 \arrow{r}{\sim} & b_2, & \alpha \colon a_1 \arrow{r}{\sim} & a_2
    \end{tikzcd}
  \end{equation}
  inducing an isomorphism
  \begin{equation}
    \begin{tikzcd}
      b_1 \arrow{r}{ s_1 } \arrow{d}{\beta} & E \arrow{r}{ t_1 } \arrow[equal]{d} & a_1 \arrow{r}{ u_1 } \arrow{d}{\alpha} & b_1 [ 1 ] \arrow{d}{\beta[1]}\\
      b_2 \arrow{r}{ s_2 } & E \arrow{r}{ t_2 } & a_2 \arrow{r}{ u_2 } & b_2 [ 1 ]
    \end{tikzcd}
  \end{equation}
  of distinguished triangles.
\end{lemma}

\Cref{lemma:uniqueness-decomposition-triangles} implies that,
given a semiorthogonal decomposition $\mathcal T = \braket{\cA^1,\ldots,\cA^\ell}$,
the association $T\mapsto \mathrm{Cone}(T_k\to T_{k-1})$ is functorial.
The corresponding functor $\alpha_k\colon \mathcal T\to\mathcal{A}^k$ is called the $k$th \emph{projection functor} of the given semiorthogonal decomposition.
\Cref{lemma:uniqueness-decomposition-triangles}
will also be used in the proof of \Cref{proposition:zariski-gluing-lemma}.

\begin{definition}
  A full triangulated subcategory $\cA \subseteq \mathcal T$ is called \emph{right admissible} (resp.~\emph{left admissible})
  if the inclusion functor $i\colon \cA \into \mathcal T$ has a right adjoint $i^!\colon \mathcal T\to \cA$
  (resp.~a left adjoint $i^*\colon \mathcal T\to \cA$).
  It is called \emph{admissible} if it is both left and right admissible.
\end{definition}

The following standard lemma will be implicitly used repeatedly.
\begin{lemma}[{\cite[Lemma~3.1]{MR0992977}}]
  \label{lemma:admissibility}
  If $\mathcal T = \braket{\cA,\cB}$ is a semiorthogonal decomposition,
  then $\cA$ is left admissible and $\cB$ is right admissible.
  Conversely, if $\cA\subseteq \mathcal T$ is left admissible,
  then $\mathcal T = \braket{\cA,\prescript{\perp}{}{\mathcal{A}}}$ is a semiorthogonal decomposition,
  and if $\cB\subseteq \mathcal T$ is right admissible,
  then $\mathcal T = \braket{\cB^\perp,\cB}$ is a semiorthogonal decomposition.
\end{lemma}

\begin{definition}
  \label{definition:linear-sod}
  Let $f\colon \Xcal \to U$ be a morphism of schemes.
  \begin{enumerate}
    \item A triangulated subcategory $\mathcal T\subseteq \derived_{\qcoh}(\Ocal_{\mathcal X})$
      is said to be $U$-\emph{linear}
      if $\mathcal T \otimes_{\mathcal O_{\cX}} f^\ast (\Perf U) \subseteq \mathcal T$,
      i.e.,~for every $t\in \mathcal T$ and $p\in \Perf U$ one has $t\otimes_{\Ocal_{\mathcal X}} f^*p\in \mathcal T$.
    \item A semiorthogonal decomposition
      \begin{equation}
        \Perf \Xcal = \braket{\cA^1,\ldots,\cA^\ell}
      \end{equation}
      is said to be \emph{$U$\dash linear} if every $\cA^i$ is $U$-linear.
  \end{enumerate}
\end{definition}

Note that a semiorthogonal decomposition $\Perf \Xcal = \braket{\cA,\cB}$ is  $U$\dash linear
if and only if $ \cA $ is $ U $\dash linear,
if and only if $ \mathcal{ B } $ is $ U $\dash linear.

We record here, for future reference, a generalisation of the orthogonality criterion that was given by Kuznetsov in  \cite[Lemma~2.7]{MR2801403}.

\begin{lemma}
  \label{lemma:semiorthogonality-via-local-Ext}
  Let~$U$ be a quasicompact and semiseparated scheme,
  and let~$f\colon\cX\to U$ be a morphism of schemes.
  A pair of $U$\dash linear subcategories $ \cA, \cB \subseteq \Perf \cX $
  is semiorthogonal
  if and only if $ \RRlHom_f(B, A) = 0$ for all $ A \in \cA $ and $ B \in \cB $.
\end{lemma}

\begin{proof}
  From \cite[Theorem~3.1.1]{MR1996800} it follows that $\derived_{ \qcoh }( \cO_U ) $
  is generated by a perfect complex.
  Then the proof of \cite[Lemma~2.7]{MR2801403} works verbatim.
\end{proof}

Next we recall the action of the braid group on
the set of all semiorthogonal decompositions of a triangulated category~$\mathcal{T}$
(where all subcategories are admissible).
Here we will denote by~$\Br_\ell$ the braid group on~$\ell$ strands.
This will play a role in \Cref{subsection:group-actions}.
We recall the notion of mutation from \cite{MR1039961},
see also \cite[\S2.4]{0904.4330}.

\begin{definition}
  Let~$\mathcal{T}=\langle\mathcal{A}^1,\ldots,\mathcal{A}^\ell\rangle$ be a semiorthogonal decomposition,
  where each~$\mathcal{A}^i$ is an admissible subcategory of~$\mathcal{T}$.
  It will be denoted by~$\mathcal{A}^\bullet$ for the purposes of this definition.
  The \emph{right mutation} (resp.~\emph{left mutation}) at position~$i$ is the semiorthogonal decomposition
  \begin{equation}
    \mathcal{T}
    =
    \langle
    \mathcal{A}^1,\ldots,\mathcal{A}^{i-2},\mathcal{A}^i,\rightmutation_i(\mathcal{A}^\bullet),\mathcal{A}^{i+1},\ldots,\mathcal{A}^\ell
    \rangle
  \end{equation}
  resp.
  \begin{equation}
    \mathcal{T}
    =
    \langle
    \mathcal{A}^1,\ldots,\mathcal{A}^{i-1},\leftmutation_i(\mathcal{A}^\bullet),\mathcal{A}^{i+1},\ldots,\mathcal{A}^\ell
    \rangle,
  \end{equation}
  where
  \begin{equation}\label{equation:right mutation}
    \rightmutation_i(\mathcal{A}^\bullet)\defeq
    {}^\perp\langle\mathcal{A}^1,\ldots,\mathcal{A}^{i-2},\mathcal{A}^i\rangle\cap\langle\mathcal{A}^{i+1},\ldots,\mathcal{A}^\ell\rangle^\perp
  \end{equation}
  resp.
  \begin{equation}\label{equation:left mutation}
    \leftmutation_i(\mathcal{A}^\bullet)\defeq
    {}^\perp\langle\mathcal{A}^1,\ldots,\mathcal{A}^{i-1}\rangle\cap\langle\mathcal{A}^i,\mathcal{A}^{i+2},\ldots,\mathcal{A}^\ell\rangle^\perp.
  \end{equation}
\end{definition}
Mutations are well-defined, i.e.,~they are again semiorthogonal decompositions by admissible subcategories,
by \cite[Proposition~4.9]{MR1039961}.
The left and right mutation operations moreover satisfy the defining relations of the braid group,
so~$\Br_\ell$ acts on the set of semiorthogonal decompositions by admissible subcategories of length $ \ell $ of $\cT$.

The following lemma explains how this additional assumption on the admissibility of the subcategories
is always satisfied in the  setting of this paper (cf.~\Cref{situation:main}).

\begin{lemma}
  \label{lemma:oo-admissible-sod}
  Let~$U$ be a quasicompact and semiseparated scheme,
  and let~$f\colon\cX\to U$ be a smooth and proper morphism of schemes.
  Then for every $U$\dash linear semiorthogonal decomposition $\Perf \Xcal = \braket{\cA^1,\ldots,\cA^\ell}$
  the subcategories~$\cA^i\subset\Perf\cX$ are admissible.
\end{lemma}

\begin{proof}
  Since $f$ is smooth and proper,
  (an enhancement of) the category $\Perf \Xcal$ is smooth and proper over $U$ as a dg category by \cite[Lemma~4.9(6)]{MR3948688}.
  Then the result follows from the proof of \cite[Lemma~4.15(4)]{MR3948688}.
  In op.~cit.~there is a standing separatedness assumption,
  which can be replaced by a semiseparated assumption,
  as explained by \cite[Remarks~3.12 and~3.16]{MR4292740}.
\end{proof}

\section{Base change for semiorthogonal decompositions}
\label{section:base-change}
In this section we will generalise
Kuznetsov's base change theorem for semiorthogonal decompositions \cite[Proposition~5.1]{MR2801403},
proven there in the context of quasiprojective varieties.
For some background the reader is referred to \cite[\S 2.3, \S 5.1]{MR2801403}
and the references therein.
Kuznetsov's base change theorem was also generalised in \cite[\S 3.2]{MR4292740}.

\subsection{Basic setup and definitions}
We fix once and for all the following setup.

\begin{situation}
  \label{situation:main}
  Let~$U$ be a quasicompact and semiseparated scheme,
  and let~$f\colon\cX\to U$ be a smooth and proper morphism of schemes.
\end{situation}
We impose the condition that~$U$ is quasicompact and semiseparated throughout;
derived categories are only well-behaved with these assumptions.
Throughout we fix a cartesian diagram
\begin{equation}
  \label{eq:cart_square}
  \begin{tikzcd}
    \cX_V \MySymb{dr}\arrow{r}{ \phi_\cX } \arrow[swap]{d}{ f_V } & \cX \arrow{d}{f}\\
    V \arrow[swap]{r}{\phi} & U
  \end{tikzcd}
\end{equation}
of schemes, where $f$ is as in \Cref{situation:main}. Note that, in particular, pushforward along (any base change of) $f$ preserves quasicoherent cohomology by \cite[\href{https://stacks.math.columbia.edu/tag/08D5}{Tag~08D5}]{stacks-project}, and we can give the following definition.

\begin{definition}
  The fibre square \eqref{eq:cart_square} is said to be \emph{exact}
  if the canonical natural transformation $ \phi^*\circ f_* \to f_{V \ast }\circ \phi_\cX^* $
  of (derived) functors $\derived_{ \qcoh }(\cO_\cX) \to \derived_{ \qcoh }(\cO_V) $ is an isomorphism.
  We also say that $ \phi \colon V \to U $ is \emph{faithful} with respect to $ f \colon \cX \to U $
  if \eqref{eq:cart_square} is exact.
\end{definition}

The following is a special case of \cite[\href{https://stacks.math.columbia.edu/tag/08IB}{Tag 08IB}]{stacks-project}
and provides us with the setting in which we need faithfulness.
\begin{lemma}
  \label{lemma:flat-implies-faithful}
  If $ f \colon \cX \to U$ is a flat morphism of schemes (e.g.~a morphism as in \Cref{situation:main}),
  then every quasicompact and semiseparated morphism $ \phi \colon V \to U $ is faithful with respect to $f$.
\end{lemma}

For later purposes, we introduce the following notation.

\begin{notation}
  \label{notation:kar-env}
  Consider a cartesian diagram as in \eqref{eq:cart_square}. Let $\cC \subset \Perf \cX$ be a class of objects. We denote by
  \begin{equation}
    \langle \cC \rangle_\phi \subset \Perf \cX_V
  \end{equation}
  the \emph{Karoubian envelope} (also known as \emph{idempotent completion}) of the triangulated hull of the class of objects $\phi_{\cX}^\ast \cC \subset \Perf \cX_V$. This is, by construction, a strictly full triangulated subcategory of $\Perf \cX_V$. When $\phi = \id_{U}$, we simply write $\langle \cC \rangle$ for the Karoubian envelope of the triangulated hull of $\cC$ inside $\Perf \cX$. In particular, we may also write $\langle \cC \rangle_\phi = \langle \phi_{\cX}^\ast \cC \rangle$.
\end{notation}

We can now introduce  the notion of \emph{base change} of a semiorthogonal decomposition,
building on the notion of linearity that was introduced in \Cref{definition:linear-sod}.

\begin{definition}[Base change for semiorthogonal decompositions]
  \label{def:generalised_BC}
  Consider a cartesian diagram as in \eqref{eq:cart_square}. Given a $ U $-linear semiorthogonal decomposition
  \begin{equation}\label{eqn:SOD_on_X}
    \Perf \cX = \braket{ \cA^1,\ldots,\cA^\ell },
  \end{equation}
  we say that a $ V $-linear semiorthogonal decomposition
  \begin{equation}
    \Perf \cX_V = \braket{ \cA^1_\phi,\ldots,\cA^\ell_\phi }
  \end{equation}
  is a \emph{base change} of \eqref{eqn:SOD_on_X} if $\phi_{\cX}^\ast \cA^i\subseteq \cA^i_\phi$ for all $i = 1,\ldots,\ell$.
\end{definition}

Our strategy for the remainder of this section is as follows.
In \Cref{subsection:affine-base-change} we construct a base change as in \Cref{def:generalised_BC} when $V$ is affine
(see~\Cref{proposition:base-change-sod}).
In \Cref{subsection:gluing-lemma} we prove an auxiliary gluing result (see \Cref{proposition:zariski-gluing-lemma}),
which will then be exploited in \Cref{subsection:general-base-change}
to prove a more general base change, cf.~\Cref{corollary:base-change-sod}, which will be used in the proof of \Cref{proposition:description-sod-f}. Finally, in \Cref{subsec:geom-points}, we prove that whether or not two linear semiorthogonal decompositions coincide can be checked on geometric points.

\subsection{Base change to an affine scheme}
\label{subsection:affine-base-change}

The following result is the generalisation of Kuznetsov's base change theorem
\cite[Proposition~5.1]{MR2801403}
in the setup of \Cref{situation:main},
limited to the case where the base change $V \to U$ considered is from an affine scheme $V$.

\begin{proposition}
  \label{proposition:base-change-sod}
  Let $ f \colon \cX \to U $ be a morphism of schemes as in \Cref{situation:main},
  and let $ \phi\colon V\to U $ be a morphism from an affine scheme. Let
  \begin{equation}
    \label{eq:SOD of cX}
    \Perf \cX
    =
    \braket{
      \cA^1, \dots, \cA^\ell
    }
  \end{equation}
  be a $ U $-linear semiorthogonal decomposition.
  Consider, according to \Cref{notation:kar-env}, the strictly full triangulated subcategories
  \begin{equation}
    \label{def:A_phi}
    \cA^i_\phi \defeq \langle \cA^i \rangle_\phi \defeq \langle \phi_{\cX}^{\ast} \cA^{i}\rangle \subseteq \Perf \cX_V,
  \end{equation}
  for $i=1,\ldots,\ell$. Then
  \begin{equation}
    \label{eq:SOD of cX_V}
    \Perf \cX_V
    =
    \braket{
      \cA_\phi^1, \dots, \cA_\phi^\ell
    }
  \end{equation}
  is the unique base change of \eqref{eq:SOD of cX}.
\end{proposition}

\begin{remark}
  \label{remark:properties}
  Before we proceed with the proof, a few remarks are in order.

  \begin{enumerate}
    \item
      \label{enumerate:phi-is-quasicompact}
      Note that, as $V$ is affine and $U$ is semiseparated in \cref{proposition:base-change-sod},
      $\phi$ is quasicompact
      by a combination of \cite[\href{https://stacks.math.columbia.edu/tag/01SP}{Tag 01SP}]{stacks-project}
      and \cite[\href{https://stacks.math.columbia.edu/tag/01SL}{Tag 01SL}]{stacks-project}.
      Moreover, $\phi$ is separated (its source being affine, cf.~\cite[\href{https://stacks.math.columbia.edu/tag/01KN}{Tag 01KN}]{stacks-project}),
      hence a fortiori semiseparated.
      Thus, by \Cref{lemma:flat-implies-faithful}, the morphism $\phi$ is automatically faithful with respect to the smooth (hence flat) morphism $f$.
    \item
      \label{enumerate:proper-over-affine-is-qc-sep}
      The scheme $\cX_V$, being proper over an affine scheme, is separated and quasicompact.
    \item
      \label{enumerate:no-need-for-objects-from-base}
      The definition \eqref{def:A_phi} of the category $\cA^i_\phi$, unlike the analogous definition from \cite{MR2801403}
      (cf.~for instance Lemma 5.2 in loc.~cit.),
      does not involve tensoring the objects in the class $\phi_{\cX}^\ast \cA^i$
      with objects from $f_V^\ast (\Perf V)$.
      This apparent discrepancy is just due to the fact that $\Perf V$ is generated by $\mathcal O_V$, as $V$ is affine.
    \item
      \label{enumerate:more-explicit-base-change}
      Unraveling \Cref{notation:kar-env}, the definition \eqref{def:A_phi} of $\cA^i_\phi$
      can be restated by saying that $\cA^i_\phi$ is
      the minimal strictly full triangulated subcategory of $\Perf \mathcal X_V$ closed under taking direct summands,
      containing the class $\phi_{\cX}^\ast \cA^i$.
  \end{enumerate}
\end{remark}

\begin{proof}[Proof of \Cref{proposition:base-change-sod}]
  To show the uniqueness of the base change, suppose we have another base change $\Perf\cX_V=\langle\cB^1,\dots,\cB^\ell\rangle$ of \eqref{eq:SOD of cX}. Then each $\cB^i$ contains $\phi_\cX^* \cA^i$,
  so that $ \cA^i_\phi \subseteq \cB^i$.
  Thus $ \cA^i_\phi = \cB^i$,
  by \cite[Lemma~3.2]{MR2801403}.

  To show that \eqref{eq:SOD of cX_V} is a base change in the sense of \Cref{def:generalised_BC},
  we point out how the original proof of \cite[Proposition~5.1]{MR2801403} is justified in the greater generality we are working in.
  As in the first part of the proof of \cite[Proposition~5.1]{MR2801403},
  the required semiorthogonality of the subcategories $ \cA^i_\phi, \cA^{ j }_\phi $
  follows from \Cref{lemma:flat-implies-faithful,lemma:semiorthogonality-via-local-Ext}.
  Finally, the assertion that the subcategories $\cA^1_\phi,\ldots,\cA^\ell_\phi$ generate the whole $\Perf \cX_V$
  follows, as explained in the latter half of the proof,
  from the special case of \cite[Proposition~3.15]{MR4292740}
  (which is a generalisation of \cite[Lemma~5.2]{MR2801403})
  where the semiorthogonal decomposition has length~1.
\end{proof}

\subsection{Zariski descent for semiorthogonal decompositions}
\label{subsection:gluing-lemma}
In this subsection we prove a gluing result (in fact, \emph{Zariski descent}) for semiorthogonal decompositions (cf.~\Cref{proposition:zariski-gluing-lemma}). This will used to prove the base change theorem that is required in this paper,
namely \Cref{corollary:base-change-sod}.

\begin{lemma}
  \label{lm:Cech}
  Suppose that $ X = U \cup V $ is an open cover of a scheme $X$ and fix two perfect complexes $ E,F \in \Perf X$.
  Then there is a distinguished triangle
  \begin{multline*}
    \RRHom_X ( E,F )
    \to
    \RRHom_U (E|_U,F|_U )
    \oplus
    \RRHom_V (E|_V, F|_V ) \\
    \to
    \RRHom_{ U  \cap V } (E|_{ U  \cap V },F|_{ U  \cap V } )
    \to
    \RRHom_X ( E, F ) [ 1 ].
  \end{multline*}
\end{lemma}

\begin{proof}
  Let $ j_U, j_V, j_{ U \cap V} $ be the open immersions from $ U, V, U \cap V $ to $ X$.
  Then, by \cite[\href{https://stacks.math.columbia.edu/tag/08GW}{Tag 08GW}]{stacks-project}, there exists a distinguished triangle
  \begin{equation}
    \begin{tikzcd}
      \cO_X
      \arrow{r} &
      j_{ U,\ast } \cO_U
      \oplus
      j_{ V,\ast } \cO_V
      \arrow{r} &
      j_{ U \cap V,\ast } \cO_{ U \cap V }
      \arrow{r} &
      \cO_X [ 1 ].
    \end{tikzcd}
  \end{equation}
  Applying the derived functors $ - \otimes_{ \cO_X }^{\mathbf{L}} \RRlHom_X (E,F) $ and $ \RR \Gamma ( X, - )$,
  we conclude.
\end{proof}

\begin{proposition}[Gluing semiorthogonal decompositions]
  \label{proposition:zariski-gluing-lemma}
  Let $ f \colon \cX \to U $ be a morphism of schemes as in \Cref{situation:main}.
  Let $ U = \bigcup_{ i \in I } U_i $ be a finite affine Zariski open covering
  and set $\cX_i = f^{-1}(U_i)$.
  Suppose that for each $ i \in I $ there exists a $ U_i$-linear semiorthogonal decomposition
  \begin{equation}
    \label{eq:SOD_cX_i}
    \Perf \cX_i = \braket{\cA_i^1,\cA_i^2,\ldots,\cA_i^\ell},
  \end{equation}
  whose base changes to the intersection $ U_{ ij } \defeq U_i \cap U_j $ coincide.
  Then there exists a unique $ U $-linear semiorthogonal decomposition
  \begin{equation}
    \Perf \cX = \braket{ \cA^1, \cA^2,\ldots,\cA^\ell }
  \end{equation}
  whose base change by $ U_i \hookrightarrow U $
  (as in \Cref{proposition:base-change-sod})
  coincides with \eqref{eq:SOD_cX_i}.
\end{proposition}

Note that,
since $U$ is assumed to be semiseparated from \Cref{situation:main},
the intersections $ U_{ij} $ are affine schemes;
in particular,
it makes sense to base change \eqref{eq:SOD_cX_i} to $U_{ij}$ via \Cref{proposition:base-change-sod}.

\begin{proof}
  We start proving the case $\ell=2$, for which we shall use the notation $\cA_i=\cA_i^1$ and $\cB_i = \cA_i^2$. Let us define the subcategories
  \begin{equation}
    \begin{split}
      \cA &\defeq \left\{ E \in \Perf \cX \mid E \vert_{ \cX_i } \in \cA_i \textrm{ for all } i \in I \right\}  \\
      \cB &\defeq \left\{ E \in \Perf \cX \mid E \vert_{ \cX_i } \in \cB_i \textrm{ for all } i \in I \right\}
    \end{split}
  \end{equation}
  of $\Perf \cX$. We will show that $ ( \cA, \cB ) $ is the unique pair of subcategories having the desired properties, by an induction on the number $ N = \# I < \infty $.

  When $ N = 1 $, we have nothing to show.
  Suppose that $ N = 2 $, and write $ I = \{ 1, 2 \} $.
  The semiorthogonality $ \cA \subseteq \cB^{ \perp } $ immediately follows from \Cref{lm:Cech}.
  We next show that $ \Perf \cX $ is generated by $ \cA $ and $ \cB $.
  Take an arbitrary object $ E \in \Perf \cX $.
  For $ i =1,2 $,
  consider the distinguished triangle
  \begin{equation}
    \begin{tikzcd}
      b_i \arrow{r}{s_i} & E|_{\cX_i} \arrow{r}{t_i} & a_i \arrow{r}{u_i} & b_i[1]
    \end{tikzcd}
  \end{equation}
  induced by the semiorthogonal decomposition \eqref{eq:SOD_cX_i},
  which is unique up to unique isomorphism by \Cref{lemma:uniqueness-decomposition-triangles}. Again by \Cref{lemma:uniqueness-decomposition-triangles}, the assumption implies that there exists a unique pair of isomorphisms
  \begin{equation}
    \begin{tikzcd}
      \beta_{ij} \colon b_j|_{\cX_{ij}} \arrow{r}{\sim} & b_i|_{\cX_{ij}}, &
      \alpha_{ij} \colon a_j|_{\cX_{ij}} \arrow{r}{\sim} & a_i|_{\cX_{ij}}
    \end{tikzcd}
  \end{equation}
  which are compatible with the morphisms $ s_i, t_i, u_i $ in the sense that the diagram
  \begin{equation}
    \begin{tikzcd}[row sep=large,column sep=large]
      b_i |_{ \cX_{ 1 2 } } \arrow{r}{ s_i |_{ \cX_{ 1 2 } } } \arrow{d}{\beta_{ j i } }
      & E |_{ \cX_{ 1 2 } } \arrow{r}{ t_i |_{ \cX_{ 1 2 } } } \arrow[equal]{d} & a_i |_{ \cX_{ 1 2 } } \arrow{r}{ u_i |_{ \cX_{ 1 2 } } } \arrow{d}{\alpha_{ j i }} & b_i |_{ \cX_{ 1 2 } } [ 1 ] \arrow{d}{\beta_{ j i }[1]}\\
      b_j |_{ \cX_{ 1 2 } } \arrow{r}{ s_j |_{ \cX_{ 1 2 } } } & E |_{ \cX_{ 1 2 } } \arrow{r}{ t_j |_{ \cX_{ 1 2 } } } & a_j |_{ \cX_{ 1 2 } } \arrow{r}{ u_j |_{ \cX_{ 1 2 } } } & b_j |_{ \cX_{ 1 2 } } [ 1 ]
    \end{tikzcd}
  \end{equation}
  is an isomorphism of distinguished triangles. Thus the object $b_{12} \defeq b_1|_{\cX_{12}}$ is canonically defined. Let $j_i \colon \cX_i \into \cX$ and $j_{12} \colon \cX_{12} = \cX_1 \cap \cX_2 \into \cX$ denote the open immersions into $\cX$. Define
  \begin{equation}
    b \defeq \Cone(j_{1\ast}b_1 \oplus j_{2\ast} b_2 \to j_{12\ast}b_{12})[1],
  \end{equation}
  so that $b \in \cB$. The morphisms $s_1$ and $s_2$, by \cite[\href{https://stacks.math.columbia.edu/tag/08DG}{Tag~08DG}]{stacks-project}, induce a morphism $s\colon b \to E$ and we define $a \defeq \Cone(s)$. Since
  \begin{equation}
    a|_{\cX_i} \cong \Cone(s_i) \cong a_i,
  \end{equation}
  it follows that $a \in \cA$, as required.

  Now consider the general case when $ N \ge 3 $ and suppose that the assertion is true up to $ N - 1$.
  Write $ I = \{ 1, 2, \dots, N \} $ and set $ U ' \defeq \bigcup_{ i = 1}^{ N - 1 } U_i$.
  Applying the induction hypothesis to this covering, we can glue the semiorthogonal decompositions on $ U_1, \dots, U_{ N - 1 } $ uniquely to produce a $ U ' $-linear semiorthogonal decomposition
  \begin{equation}\label{eq:SOD on U'}
    \Perf \cX_{ U ' } = \langle \cA ', \cB ' \rangle.
  \end{equation}
  By using the case $ \# I = 2 $, we can further uniquely glue
  \eqref{eq:SOD on U'}
  with the one on $ U_N $ to obtain the desired $ U $-linear semiorthogonal decomposition.

  The case of a semiorthogonal decomposition of arbitrary length $\ell$ can be proven by induction on~$\ell$. Consider~$U_i$\dash linear semiorthogonal decompositions
  \begin{equation}
    \Perf\cX_i
    =
    \langle
    \cA_i^1,\ldots,\cA^{\ell+1}_i
    \rangle,
  \end{equation}
  of length~$\ell+1$ which coincide over~$U_{ij}$. Then we can reduce it to the statement for semiorthogonal decompositions of length~$\ell$ by considering the semiorthogonal decompositions
  \begin{equation}
    \begin{aligned}
      \Perf\cX_i
      &=
      \langle
      \cA_i^1,\langle\cA_i^2,\cA_i^3\rangle,\ldots,\cA_i^{\ell+1}
      \rangle \\
      &=
      \langle
      \langle\cA_i^1,\cA_i^2\rangle,\cA_i^3,\ldots,\cA_i^{\ell+1}
      \rangle
    \end{aligned}
  \end{equation}
  and glue these two semiorthogonal decompositions to semiorthogonal decompositions of~$\Perf\cX$. It then suffices to consider the intersection of the first components in the two gluings to find the required semiorthogonal decomposition of length~$\ell+1$.
\end{proof}

\subsection{Generalised base change}
\label{subsection:general-base-change}
As a consequence of the gluing result given in the previous section,
we can now present a rather general version of the base change theorem for semiorthogonal decompositions. This will be used in \Cref{proposition:description-sod-f} in order to settle \Cref{theorem:main-intro}\ref{enumerate:main-intro-values}, as well as later in \Cref{lm:sod->dec}.

\begin{corollary}
  \label{corollary:base-change-sod}
  Let $ f \colon \cX \to U $ be a morphism of schemes as in \Cref{situation:main}.
  For every morphism $ \phi \colon V \to U $ from a quasicompact and semiseparated scheme $ V $,
  and for every $U$-linear semiorthogonal decomposition
  \begin{equation}\label{eq:SOD46872634872}
    \Perf \cX = \braket{ \cA^1,\ldots,\cA^\ell },
  \end{equation}
  there exists a unique base change
  \begin{equation}
    \Perf \cX_V = \braket{ \cA^1_\phi, \dots, \cA^\ell_\phi }
  \end{equation}
  of \eqref{eq:SOD46872634872} by $ \phi $.
\end{corollary}

\begin{proof}
  Note that $ V $ admits a finite affine Zariski open covering with affine intersections.
  Over each affine open subset of $ V $ one has the base change of \eqref{eq:SOD46872634872} by \Cref{proposition:base-change-sod}, and they coincide on the intersections by the uniqueness part of \Cref{proposition:base-change-sod}.
  Hence, by \Cref{proposition:zariski-gluing-lemma}, they uniquely glue together to produce a $ V $-linear semiorthogonal decomposition%
  \begin{equation}
    \Perf \cX_V = \braket{ \cA^1_\phi, \dots, \cA^\ell_\phi },
  \end{equation}
  which is a base change of \eqref{eq:SOD46872634872} by construction.

  Now let $ \Perf \cX_V=\langle \cB^1_\phi, \dots, \cB^\ell_\phi \rangle $
  be another base change of \eqref{eq:SOD46872634872}.
  Then again by the uniqueness part in \Cref{proposition:base-change-sod}
  it follows that $ \cA^i_\phi \subseteq \cB^i_\phi $ for all $ i $,
  which immediately implies the equality for all $ i $ by \cite[Lemma~3.2]{MR2801403}.
  Thus we obtain the uniqueness of the base change.
\end{proof}

\subsection{Comparing semiorthogonal decompositions on geometric points}
\label{subsec:geom-points}
We close this section by proving that whether or not two linear semiorthogonal decompositions coincide can be checked on geometric points.

\begin{lemma}\label{lm:coincidence of SODs can be checked at closed points}
  Let $ f \colon \cX \to U $ be a morphism of schemes as in \Cref{situation:main}.
  Suppose that
  \begin{equation}
    \Perf \cX = \langle \cA^1, \cB^1 \rangle, \quad \Perf \cX = \langle \cA^2, \cB^2 \rangle
  \end{equation}
  are $ U $-linear semiorthogonal decompositions whose base changes to all geometric points of $U$ coincide.
  Then they coincide.
  The same holds for semiorthogonal decompositions of length~$\ell\geq 2$.
\end{lemma}

\begin{proof}
  Take an arbitrary object $ E \in \cA^1$,
  and consider the decomposition triangle
  \begin{equation}
    \begin{tikzcd}
      b \arrow{r} & E \arrow{r} & a \arrow{r} & b [ 1 ]
    \end{tikzcd}
  \end{equation}
  with respect to the second semiorthogonal decomposition.
  The assumption implies that, for every geometric point $ u \colon \Spec K \to U$,
  one has that $ b \vert_{ \cX_u } = 0$.
  This implies that, for every geometric point $ x \in \cX$,
  one has $ \iota_x^* b = 0$,
  where $ \iota_x \colon \Spec \bfk ( x ) \to \cX $ is the canonical morphism.
  Since $ b $ is perfect, \Cref{lemma:derived-nakayama-global} implies that $ b = 0$.
  Thus we obtain $ E \cong a \in \cA^2$, so that $ \cA^1 \subseteq \cA^2 $.
  Exchanging the roles of $ \cA^1$ and $ \cA^2$, we obtain the other inclusion.

  The general case follows by reduction to the case of length~2:
  given a semiorthogonal decomposition~$\langle\cA^1,\ldots,\cA^\ell\rangle$
  we obtain a semiorthogonal decomposition~$\langle\cA^1,\langle\cA^2,\ldots,\cA^\ell\rangle\rangle$.
  Given two semiorthogonal decompositions of length~$\ell$
  we can apply this reduction,
  and the lemma for length~2 allows us to conclude that the first subcategories coincide.
  Finally, observe that
  by mutation we can let every subcategory~$\cA^i$ be the first component in
  a semiorthogonal decomposition of length~2,
  and we apply the same sequence of mutations to the second semiorthogonal decomposition of length~$\ell$
  to conclude.
\end{proof}

\section{The sheaf \texorpdfstring{$\sod_f^\ell$}{} of semiorthogonal decompositions}
\label{section:sod-functor-sheaf}

Thanks to the results of \Cref{section:base-change} (and \Cref{proposition:base-change-sod} in particular),
we can now define a functor of semiorthogonal decompositions,
and using \cite{MR4205113}
we will show that it is a sheaf on the big \'etale site of affine schemes,
extending the Zariski descent from \cref{proposition:zariski-gluing-lemma}.
We then define the moduli functor for arbitrary schemes by using the equivalence of toposes from \Cref{lemma:equivalence-of-toposes},
and confirm the natural expectation over quasicompact and semiseparated schemes in~\Cref{proposition:description-sod-f}, which will prove \Cref{theorem:main-intro}\ref{enumerate:main-intro-values}.

\subsection{Definition of the moduli functor and the sheaf property}
The following is the main definition of this paper.
\begin{definition}
  \label{definition:presheaf-sod-f}
  Let $ f \colon \cX \to U $ be a morphism of schemes as in \Cref{situation:main}.
  Fix an integer $\ell \geq 1$.
  Define the functor
  \begin{equation}
    \begin{tikzcd}
      \overline{\sod}^\ell_f \colon \Aff_U^{\op} \arrow{r} & \Sets
    \end{tikzcd}
  \end{equation}
  by sending an affine $U$-scheme $\phi\colon V\to U$ to the set
  \begin{equation}
    \overline{\sod}^\ell_f(\phi) =
    \Set{
      \bigl(\cA^1,\ldots,\cA^\ell\bigr) |
      \begin{array}{c}
        \Perf \Xcal_V  = \braket{\cA^1,\ldots,\cA^\ell} \textrm{ is a }V\textrm{\dash linear} \\
        \textrm{semiorthogonal decomposition}
      \end{array}
    }.
  \end{equation}
  The pullback maps are well-defined by \Cref{proposition:base-change-sod}.
\end{definition}

We quickly recall the following \'etale sites from \cite[\href{https://stacks.math.columbia.edu/tag/0214}{Tag~0214}]{stacks-project}.
Our goal is to extend \cref{definition:presheaf-sod-f}
to the following \'etale site.
\begin{definition}
  \label{definition:big-etale-site}
  The \emph{big \'etale site} $\Schet{ U }$ of a scheme $U$ consists of:
  \begin{enumerate}
    \item the category $ \Sch_U $ of (arbitrary) schemes over $ U $, equipped with
    \item \label{cov-et} the Grothendieck topology in which an \emph{\'etale covering} of an object $ (V \to U) \in \Sch_U $
      is a collection of \'etale morphisms $ \left( \pi_i \colon V_i \to V \right)_{ i \in I } $
      such that $ V = \bigcup_{ i \in I } \pi_i ( V_i )$.
  \end{enumerate}
\end{definition}
We will do this desired extension starting from the following site, denoted $ \Affet { U }$.
\begin{definition}
  \label{definition:affine-etale-site}
  The \emph{big affine \'etale site} $ \Affet{ U } $ of a scheme $U$ consists of:
  \begin{enumerate}
    \item
      the category $ \Aff_U $ of arbitrary \emph{affine} schemes over $ U $, equipped with
    \item
      the Grothendieck topology in which a
      \emph{standard \'etale covering}
      of an object $(V \to U) \in \Aff_{U}$ is an \'etale covering $ \left( \pi_i \colon V_i \to V \right)_{ i \in I } $ in the sense  of \Cref{definition:big-etale-site},
      such that $I$ is a finite set of indices and each $ V_i $ is affine.
  \end{enumerate}
\end{definition}
Given a site $\mathcal S$, we denote by $\Sh \mathcal S$ the associated topos,
i.e.,~the category of set-valued sheaves on $\mathcal S$.
The following equivalence of toposes will be important for the definition of the functors we will be working with.

\begin{lemma}[{\cite[\href{https://stacks.math.columbia.edu/tag/021E}{Tag~021E}]{stacks-project}}]
  \label{lemma:equivalence-of-toposes}
  Let $U$ be a scheme. The natural functor $ \Affet{ U } \to \Schet{ U }$ induces an equivalence of toposes
  \begin{equation}
    \begin{tikzcd}
      \Sh \Affet{ U }
      \arrow{r}{\sim} &
      \Sh \Schet{ U }.
    \end{tikzcd}
  \end{equation}
\end{lemma}

We can now state the main result of this section.

\begin{theorem}
  \label{theorem:sod-f-fpqc-sheaf}
  The functor $\overline{\sod}_f^\ell$ of \Cref{definition:presheaf-sod-f}
  is a sheaf on the big affine \'etale site $\left( \Aff_U \right)_{ \Et }$.
\end{theorem}

\begin{proof}
  This is a special case of the more general result \cite[Theorem 1.4]{MR4205113} by Antieau--Elmanto,\footnote{
    In the first version of this paper, we gave an argument for the sheaf property in characteristic 0,
    but this was surpassed by the characteristic-free argument cited above.
  }
  which in fact implies that $\overline{\sod}_f^\ell$ is a stack in the fppf topology.
  The poset one chooses in their setup is the totally ordered set~$[\ell]=\{1<\cdots<\ell\}$
  as in \cite[Proposition~3.8]{MR4205113}.
  Note that,
  as in the discussion following \cite[Theorem 1.4]{MR4205113},
  the fppf stack is actually discrete in nature,
  and thus an fppf sheaf.
\end{proof}

Note that in \cite{MR4205113} the assumption on the base~$U$
is that it is~1-affine \cite[Definition~1.3.7]{MR3381473},
which is implied
as in \cite[Warning~2.5]{MR4205113}
by taking~$U$ to be quasicompact and quasiseparated.

Thanks to \Cref{theorem:sod-f-fpqc-sheaf} and \cref{lemma:equivalence-of-toposes},
it is possible to give the following definition.

\begin{definition}
  \label{definition:sod-on-all-schemes}
  Let $ f \colon \cX \to U $ be a morphism of schemes as in \Cref{situation:main}.
  Fix an integer $\ell \geq 1$. We define
  \begin{equation}
    \begin{tikzcd}
      \sod_f^\ell \colon \Sch_U^{\op}\arrow{r} & \Sets
    \end{tikzcd}
  \end{equation}
  to be the \'etale sheaf corresponding to $\overline{\sod}_f^\ell$
  under the equivalence of toposes of \Cref{lemma:equivalence-of-toposes}.
  When $\ell=2$, we shall drop the superscript and write $\sod_f$ for $\sod_f^2$.
\end{definition}

\begin{remark}
  \label{remark:MR4205113}
  The main result of \cite{MR4205113}, namely that semiorthogonal decompositions satisfy fppf descent, is proved with the formalism of derived algebraic geometry, and in particular the categories of perfect complexes are considered as enhanced categories. By \cite[Theorem~1.2\,(2)]{MR2669705} and the restriction to perfect complexes for a smooth and proper morphism~$\cX\to U$ as discussed in \cite{MR3730514}, we obtain that this agrees with the approach in this paper.

  In what follows we will further amplify \Cref{theorem:sod-f-fpqc-sheaf} to show that, when $U$ is further assumed to be an excellent scheme,~${\sod}_f^\ell$ is an algebraic space which is \'etale over the base,
  which gives important geometric insight into the structure of this moduli space.
  As explained in \Cref{example:topological-sod} one cannot expect such a strong result in complete generality,
  despite semiorthogonal decompositions satisfying fppf descent in great generality.
\end{remark}

\subsection{Description of \texorpdfstring{$\sod_f^\ell$}{SOD} and \texorpdfstring{proof of \Cref{theorem:main-intro}\ref{enumerate:main-intro-values}}{its values}}
We have the following explicit description of the sheaf $\sod_f^\ell$.
For an arbitrary $ U $-scheme $\phi \colon V \to U$,
we let $\phi_* \colon \Aff_V \to \Aff_U $ be the functor
defined by sending $a \colon V' \to V$ to~$\phi \circ a \colon V' \to U$.
Then the value $\sod_f^\ell (\phi) $ is the limit set
\begin{equation}
  \label{eqn:explicit-sod}
  \sod_f^\ell (\phi)
  =
  \lim_{ \Aff_V^{\op} } \overline{\sod}_f^\ell \circ \phi_* = \lim_{(a\colon V' \to V) \in \Aff_V} \overline{\sod}_f^\ell(\phi \circ a).
\end{equation}

We confirm in \Cref{proposition:description-sod-f} that the value of the sheaf $ \sod_f^\ell \in \Sh  \Schet{ U } $
at a quasicompact and semiseparated $U$-scheme coincides with what we naively expect,
by spelling out the construction of the equivalence between toposes
from \Cref{lemma:equivalence-of-toposes}.

\begin{lemma}\label{lemma:restriction-coincides}
  There is a canonical isomorphism
  \begin{equation}
    \begin{tikzcd}
      \overline{\sod}_f^\ell \arrow{r}{\sim} &
      \sod_f^\ell \big|_{  \Aff_U^{\op} }
    \end{tikzcd}
  \end{equation}
  of functors $\Aff_U^{\op} \to \Sets$.
\end{lemma}

\begin{proof}
  For an object $\phi\colon V \to U$ of $\Aff_U \subset \Sch_{U}$,
  we have an obvious map
  \begin{equation}
    \label{eq:comparison map}
    \begin{tikzcd}
      \overline{\sod}_f^\ell (\phi) \arrow{r} & \displaystyle\lim_{(a\colon V' \to V) \in \Aff_V} \overline{\sod}_f^\ell(\phi\circ a) = \sod_f^\ell (\phi)
    \end{tikzcd}
  \end{equation}
  which yields a map of presheaves $ \overline{\sod}_f^\ell \to
  \sod_f^\ell|_{  \Aff_U^{\op} }$.
  Since the identity $ \id_V\colon V \to V$ is a final object of the category $ \Aff_V$,
  one can easily verify that \eqref{eq:comparison map} is a bijection.
\end{proof}

From this point onwards, in virtue of \Cref{lemma:restriction-coincides}, we will not make a notational distinction between $\overline{\sod}_f^\ell$ and $\sod_f^\ell$.

\smallbreak
We can now complete the proof of \Cref{theorem:main-intro}\ref{enumerate:main-intro-values},
describing the value of~$\sod_f^\ell$ on more general $U$-schemes~$V\to U$.

\begin{proposition}
  \label{proposition:description-sod-f}
  Let $ f \colon \cX \to U $ be a morphism of schemes as in \Cref{situation:main}.
  For every quasicompact and semiseparated $U$-scheme~$V$, there is a natural bijection
  \begin{equation}
    \sod_f^\ell(V\to U)
    \cong
    \Set{
      \begin{array}{c}
        V\textrm{-}\mathrm{linear~semiorthogonal} \\
        \mathrm{decompositions }\Perf \cX_V = \langle \cA^1, \ldots,\cA^\ell \rangle
      \end{array}
    }.
  \end{equation}
\end{proposition}

\begin{proof}
  Combine \Cref{lemma:restriction-coincides} with \Cref{proposition:zariski-gluing-lemma,corollary:base-change-sod}
  and the sheaf property of $\sod_f^\ell$.
\end{proof}

\section{The functor \texorpdfstring{$\dec_{\Delta_f}$}{DEC} of decompositions of the diagonal}
\label{section:sod-is-dec}

The aim of this section is to introduce an alternative description of the functor $ \sod_f \defeq \sod_f^2
$, denoted by $ \dec_{ \Delta_f }$. We start by fixing some notation.

\subsection{Basic construction via Fourier--Mukai functors}
\label{sec:FM}
Let~$U$ be a quasicompact and semiseparated scheme,
and let $f\colon \Xcal \to U$ be a smooth and proper morphism of schemes (i.e.~fix $f$  as in \Cref{situation:main}).
Let $\phi\colon V\to U$ be a morphism; we write $\Xcal_V = \Xcal\times_UV$ for the base change and $f_V\colon \Xcal_V \to V$ for the induced map, just as in \eqref{eq:cart_square}.
Set $\Ycal_V = \mathcal X_V\times_V\mathcal X_V$ and fix an object
\begin{equation}
  K \in \Perf \Ycal_V.
\end{equation}
We denote by $\Phi_K$ the associated \emph{Fourier--Mukai functor}, namely
\begin{equation}
  \label{FM_functor_K}
  \begin{tikzcd}[row sep=tiny]
    \Perf \mathcal X_V \arrow{r}{\Phi_K} & \Perf \mathcal X_V \\
    E \arrow[mapsto]{r} & \pr_{2\ast}(\pr_1^*(E)\otimes K).
  \end{tikzcd}
\end{equation}
Here, $\pr_i$ are the projections $\mathcal Y_V \to \mathcal X_V$. Note that we are implicitly using the fact that $\pr_{2\ast}$ preserves perfectness. To see this, note that $ \pr_2 $ is smooth and hence perfect. Since it is also proper, it follows that $ \pr_2 $ is quasiperfect (see \cite[Example~2.2]{MR2346195} and \cite[page~213]{MR2346195}), hence $\Phi_K$ does indeed land in $\Perf \mathcal X_V$.

\begin{notation}
  \label{notation:KE}
  Let $f$, $\phi$ and $K$ be as above. We denote by
  \begin{equation}
    \mathscr E_K \subseteq \Perf \Xcal_V
  \end{equation}
  the essential image of the Fourier--Mukai functor $\Phi_K$ defined in \eqref{FM_functor_K}. Furthermore, according to \Cref{notation:kar-env}, we denote by
  \begin{equation}
    \langle \mathscr E_K\rangle_{\pr_2} \defeq \langle \pr_{2}^{\ast} \mathscr E_K\rangle \subseteq \Perf \mathcal Y_V
  \end{equation}
  the Karoubian envelope of the triangulated hull of $ \pr_2^\ast \mathscr E_K \subseteq \Perf \mathcal Y_V$, namely the smallest strictly full $\cX_V$-linear triangulated subcategory of $\Perf \mathcal Y_V$ that is closed under taking direct summands and contains the class of objects of the form $\pr_2^\ast u$ for $u \in \mathscr E_K$.

\end{notation}

\subsection{Definition of the functor \texorpdfstring{$\dec_{\Delta_f}$}{DEC}}

In the situation of \Cref{sec:FM}, note that $\Ocal_{\Delta_{f_V}} \in \Perf \Ycal_V$ is a perfect complex because $f_V$ is smooth and separated (cf.~\Cref{lemma:Diagonal_Perfect}). We are now ready to define the functor $\dec_{\Delta_f}$.

\begin{definition}
  \label{def:DEC-functor}
  Let~$f \colon \cX \to U$ be a smooth and proper morphism of schemes.
  Define the functor
  \begin{equation}
    \begin{tikzcd}
      \dec_{\Delta_f}\colon \Aff_U^{\op} \arrow{r} & \Sets
    \end{tikzcd}
  \end{equation}
  by sending an affine $U$-scheme $\phi\colon V \to U$ to the set of isomorphism classes of distinguished triangles
  \begin{equation}
    \label{eq:decomposition zeta}
    \begin{tikzcd}
      \zeta\colon\qquad   K_2\arrow{r} & \Ocal_{\Delta_{f_V}} \arrow{r} & K_1 \arrow{r} & K_2[1]
    \end{tikzcd}
  \end{equation}
  in $\Perf \Ycal_V$, such that
  \begin{equation}
    \RRlHom_{f_V}(\mathscr E_{K_2},\mathscr E_{K_1}) = 0.
  \end{equation}
  Here, $\mathscr E_{\bullet}$ was defined in \Cref{notation:KE}, and a distinguished triangle as in \eqref{eq:decomposition zeta} is considered isomorphic to a second distinguished triangle
  \begin{equation}
    \begin{tikzcd}
      \zeta'\colon\qquad   K'_2\arrow{r} & \Ocal_{\Delta_{f_V}} \arrow{r} & K'_1 \arrow{r} & K'_2[1]
    \end{tikzcd}
  \end{equation}
  whenever there is an isomorphism of distinguished triangles
  \begin{equation}
    \label{eqn:iso-triangles}
    \begin{tikzcd}
      K_2\arrow{r}\isoarrow{d} & \Ocal_{\Delta_{f_V}} \arrow{r}\arrow[equal]{d} & K_1 \arrow{r}\isoarrow{d} & K_2[1]\isoarrow{d} \\
      K'_2\arrow{r} & \Ocal_{\Delta_{f_V}} \arrow{r} & K'_1 \arrow{r} & K'_2[1]
    \end{tikzcd}
  \end{equation}
  where we insist that the middle objects are connected by the identity morphism.
\end{definition}

By \Cref{lemma:uniqueness-decomposition-triangles}, there exists at most one isomorphism $\zeta \simto \zeta'$ and so the functor $\dec_{\Delta_f}$ is indeed set-valued. We call $\dec_{\Delta_f}$ the functor of \emph{decompositions of the diagonal}.

\subsection{\texorpdfstring{The equivalence $\sod_f = \dec_{\Delta_f}$}{The equivalence between SOD and DEC}}
\label{subsec:sod=dec}
Let $ f \colon \cX \to U $ be a morphism of schemes as in \Cref{situation:main}.
The aim of this section is to prove that the functors $\sod_f$ and $\dec_{\Delta_f}$ are naturally isomorphic. To this end, we shall construct natural transformations
\begin{equation}
  \begin{tikzcd}
    \dec_{\Delta_f} \arrow{r}{\delta} & \sod_f, \qquad \sod_f \arrow{r}{\sigma} & \dec_{\Delta_f},
  \end{tikzcd}
\end{equation}
inverse to each other. The former is constructed in the following lemma.

\begin{lemma}
  \label{lm:dec->sod}
  Let $ f \colon \cX \to U $ be a morphism of schemes as in \Cref{situation:main}.
  Let $\phi\colon V \to U$ be a morphism from an affine scheme. Then sending a decomposition $\zeta \in \dec_{ \Delta_f } ( \phi )$, given as \eqref{eq:decomposition zeta}, to the pair
  \begin{equation}
    (\langle\mathscr E_{ K_1 }\rangle, \langle\mathscr E_{ K_2 }\rangle), \qquad \mathscr E_{ K_i } \subseteq \Perf \mathcal X_V,
  \end{equation}
  defines a natural transformation $\delta\colon \dec_{\Delta_f} \to \sod_f$.
\end{lemma}

\begin{proof}
  We need to check that
  \begin{equation}
    ( \langle\mathscr E_{ K_1 }\rangle, \langle\mathscr E_{ K_2 }\rangle )
    \in
    \sod_f ( \phi ).
  \end{equation}
  First of all, we observe that the subcategories $\langle\mathscr E_{K_i}\rangle$, for $i=1,2$, are $V$-linear, simply by definition. The fact that they are semiorthogonal follows combining \Cref{lemma:semiorthogonality-via-local-Ext} with the vanishing $\RRlHom_{f_V}(\mathscr E_{K_2},\mathscr E_{K_1}) = 0$. To check that $\Perf \cX_V$ is generated by $\langle\mathscr E_{K_1}\rangle$ and $\langle\mathscr E_{K_2}\rangle$, take an object $F \in \Perf \cX_V$ and apply the composite functor $\pr_{2\ast}(\pr_1^\ast F \otimes -)$ to the distinguished triangle \eqref{eq:decomposition zeta} defining $\zeta$. This yields a distinguished triangle
  \begin{equation}
    \begin{tikzcd}
      \Phi_{K_2}(F) \arrow{r} & F \arrow{r} & \Phi_{K_1}(F) \arrow{r} & \Phi_{K_2}(F)[1],
    \end{tikzcd}
  \end{equation}
  with $\Phi_{K_i}(F) \in \mathscr E_{K_i}$ for $i = 1,2$.
  But then $\langle\mathscr E_{K_1}\rangle$ and $\langle\mathscr E_{K_2}\rangle$ generate $\Perf \cX_V$.
  Thus we conclude that~$(\langle\mathscr E_{K_1}\rangle, \langle\mathscr E_{K_2}\rangle)
  \in
  \sod_f ( \phi )$, as required. The compatibility of this assignment with base change is straightforward to check.
\end{proof}

Next, we construct a natural transformation $\sigma\colon \sod_f \to \dec_{ \Delta_f}$ in the opposite direction, following \cite[Theorem 7.1]{MR2801403}. This will be the inverse to the natural transformation $\delta$ obtained in \Cref{lm:dec->sod}.

\begin{lemma}
  \label{lm:sod->dec}
  Let $ f \colon \cX \to U $ be a morphism of schemes as in \Cref{situation:main}.
  There exists a natural transformation $\sigma\colon \sod_f \to \dec_{ \Delta_f}$.
\end{lemma}

\begin{proof}
  Let $\phi\colon V \to U$ be a morphism from an affine scheme. Fix a semiorthogonal decomposition
  \begin{equation}
    (\cA^1,\cA^2) \in \sod_f(\phi).
  \end{equation}
  Recall that we have set  $\Ycal_V = \Xcal_V \times_V \Xcal_V$. Using the faithfulness of $f_V\colon \cX_V \to V$ with respect to itself, we can form the $\Xcal_V$\dash linear pullback semiorthogonal decomposition
  \begin{equation}
    \label{SOD-Y_V}
    \Perf \Ycal_V = \braket{\cA^1_{\pr_2},\cA^2_{\pr_2}},
  \end{equation}
  where $\cA^i_{\pr_2} \defeq \langle \cA^i \rangle_{\pr_2} \defeq \langle \pr_2^\ast \cA^i \rangle \subseteq \Perf \Ycal_V$ denotes, just as in \Cref{proposition:base-change-sod}, the Karoubian envelope of the triangulated hull of $\pr_2^\ast \cA^i \subseteq \Perf \cX_V$ (cf.~\Cref{notation:kar-env}). Note that to form this (unique) base change, we are applying \Cref{corollary:base-change-sod} with $f$ replaced by $f_V$ and $\phi$ replaced by $f_V$ again, and the assumptions
  (namely that $\cX_V$ be quasicompact and semiseparated)
  are met by \Cref{enumerate:proper-over-affine-is-qc-sep} in \Cref{remark:properties}.

  Now, given the semiorthogonal decomposition \eqref{SOD-Y_V}, we can uniquely decompose the structure sheaf of the diagonal $\Ocal_{\Delta_{f_V}} \in \Perf \Ycal_V$
  by means of a distinguished triangle
  \begin{equation}
    \label{triangle_candidate}
    \begin{tikzcd}
      K_{2} \arrow{r} &  \Ocal_{\Delta_{f_V}} \arrow{r} &  K_{1} \arrow{r} &  K_{2} [1],
    \end{tikzcd}
  \end{equation}
  where $K_i \in \cA^i_{\pr_2} \subseteq \Perf \Ycal_V$.
  Sending $(\cA^1,\cA^2) \in \sod_f(\phi)$ to the distinguished triangle \eqref{triangle_candidate} is compatible with base change and thus defines the sought after natural transformation $\sigma\colon \sod_f \to \dec_{\Delta_f}$.
\end{proof}

\begin{lemma}\label{lm:sod->dec->sod is identity}
  Let $ f \colon \cX \to U $ be a morphism of schemes as in \Cref{situation:main}.
  Then the composite natural transformation $\delta\circ\sigma\colon \sod_f \to \dec_{\Delta_f} \to \sod_f$ is the identity.
\end{lemma}

\begin{proof}
  Let $\phi\colon V \to U$ be a morphism from an affine scheme. Fix a semiorthogonal decomposition
  \begin{equation}
    (\cA^1,\cA^2) \in \sod_f(\phi),
  \end{equation}
  so that we can form the distinguished triangle \eqref{triangle_candidate}.
  We have to prove that
  \begin{equation}
    ( \langle \mathscr E_{ K_{1} } \rangle, \langle \mathscr E_{ K_{2} } \rangle)
    =
    ( \cA^1, \cA^2 )
    \in
    \sod_f ( \phi ).
  \end{equation}
  By \cite[Lemma 3.2]{MR2801403}, it is enough to show that $\cA^i \subseteq \langle \mathscr E_{ K_{i} } \rangle$ for $i=1,2$. We do the case $i=1$, the other inclusion being proved similarly.

  Let~$a\in\cA^1$.
  The natural morphism~$a\to\Phi_{K_{1}}(a)$
  induced by the second morphism in \eqref{triangle_candidate}
  is an isomorphism by construction,
  and assumption on~$a$,
  showing that~$a\in \langle \mathscr E_{ K_{1} } \rangle$ as required.
\end{proof}

\begin{lemma}\label{lm:dec->sod->dec is identity}
  Let $ f \colon \cX \to U $ be a morphism of schemes as in \Cref{situation:main}.
  Then the composite natural transformation $\sigma\circ\delta\colon \dec_{\Delta_f} \to \sod_f \to \dec_{\Delta_f}$ is the identity.
\end{lemma}

\begin{proof}
  Let $\phi\colon V \to U$ be a morphism from an affine scheme. Fix a decomposition $ \zeta \in \dec_{ \Delta_f } (\phi) $ as in \eqref{eq:decomposition zeta}, and consider the associated semiorthogonal decomposition
  \begin{equation}
    \label{sod8749}
    ( \langle\mathscr E_{ K_1 }\rangle, \langle\mathscr E_{ K_2 }\rangle )
    \in
    \sod_f ( \phi )
  \end{equation}
  of $ \Perf \cX_V$.
  Note that we have equalities
  \begin{equation}
    \label{eqn:462}
    \langle\mathscr E_{K_1}\rangle_{\pr_2} = (\langle\mathscr E_{K_2}\rangle_{\pr_2})^\perp, \qquad \langle\mathscr E_{K_2}\rangle_{\pr_2} = {}^{ \perp }(\langle\mathscr E_{K_1}\rangle_{\pr_2}),
  \end{equation}
  as $\cX_V$-linear triangulated subcategories of $\Perf \cY_V$ (cf.~\Cref{notation:KE}).

  We need to show that the decomposition induced by \eqref{sod8749} is nothing but the decomposition $ \zeta $ we started with. Using \eqref{eqn:462}, we are reduced to proving the following assertions:
  \begin{equation}
    K_1 \in (\langle\mathscr E_{K_2}\rangle_{\pr_2})^\perp, \qquad K_2 \in {}^{ \perp }(\langle\mathscr E_{K_1}\rangle_{\pr_2}).
  \end{equation}
  By the definition of the category $\langle\mathscr E_{K_2}\rangle_{\pr_2}$ and the additivity and exactness of $\RRHom$, it is enough to show the vanishing
  \begin{equation}
    \label{eq:the vanishing to be confirmed}
    \RRHom_{\cY_V}(\pr_2^* F, K_1 )=0,
  \end{equation}
  for all $F \in \mathscr E_{K_2} \subseteq \Perf \cX_V$.%

  By definition of $\mathscr E_{K_2}$, for each such $F \in \mathscr E_{K_2}$ there exists a perfect complex $H \in \Perf \cX_V$ such that $F \cong \pr_{ 2 \ast } ( \pr_1^\ast H \otimes K_2 )$. Thus
  \begin{equation}
    \label{rhom-vanishing}
    \begin{split}
      \RRHom_{\cY_V}(\pr_2^\ast F, K_1 )
      &\cong\RRHom_{\cY_V}(\pr_2^\ast \pr_{ 2 \ast } ( \pr_1^\ast H \otimes K_2),K_1) \\
      &\cong\RRHom_{\cX_V}(\pr_{ 2 \ast } ( \pr_1^\ast H \otimes K_2),\pr_{ 2 \ast }K_1) \\
      &\cong\RRHom_{\cX_V}(F,\pr_{ 2 \ast }K_1),
    \end{split}
  \end{equation}
  where we have used the adjunction isomorphism \eqref{adjunction} for the second isomorphism.
  On the other hand, $\pr_{ 2 \ast }K_1 = \pr_{ 2 \ast }(\pr_1^\ast \mathcal O_{\cX_V} \otimes K_1) = \Phi_{K_1}(\mathcal O_{\cX_V}) \in \mathscr E_{K_1}$ and by definition of decomposition we have $\RRHom_{\cX_V}(\mathscr E_{K_2},\mathscr E_{K_1}) = 0$, therefore the last term in \eqref{rhom-vanishing} vanishes and \eqref{eq:the vanishing to be confirmed} is proved.
\end{proof}

\begin{theorem}\label{theorem:Iso_Functors_SOD_DEC} %
  Let $ f \colon \cX \to U $ be a morphism of schemes as in \Cref{situation:main}.
  Then $\dec_{\Delta_f}$ and $\sod_f$, as presheaves on $\Aff_U$, are isomorphic to each other.

  In particular, by \Cref{theorem:sod-f-fpqc-sheaf}, $ \dec_{\Delta_f} $ is a sheaf on the big affine \'etale site $ \Affet{ U }
  $.
\end{theorem}

\begin{proof}
  Combine \Cref{lm:sod->dec->sod is identity} and \Cref{lm:dec->sod->dec is identity} with one another.
\end{proof}

\subsection{Explicit description of \texorpdfstring{$\dec_{ \Delta_f}$}{DEC} via classical generators}
We next rephrase the functor $ \dec_{ \Delta_f} $ in a convenient way, using the notion of generator for triangulated categories.

\begin{definition}(\cite[\S2.1]{MR1996800})
  A collection of objects $ \mathcal G $ \emph{classically generates} a triangulated category $ \cT $ if $ \cT $ is the minimal strictly full triangulated subcategory of $ \cT $ which contains $ \mathcal G $ and is closed under taking direct summands.
  An object $ G $ \emph{classically generates} the category if $ \mathcal G = \{ G \} $ does.
\end{definition}

\begin{lemma}
  \label{lemma:bvdb-ko-p}
  Let $ \cX $ be a quasicompact and semiseparated scheme.
  Then
  \begin{enumerate}
    \item
      \label{item:cg-i}
      The category $ \Perf \cX $ admits a classical generator.
    \item
      \label{item:cg-ii}
      Every right or left admissible triangulated subcategory $ \cC \subseteq \Perf \cX $
      is classically generated by the image of the classical generator under the projection functor.
    \item
      \label{item:cg-iii}
      If $\psi\colon \cX' \to \cX$ is an affine morphism and $G \in \Perf \cX$ is a classical generator,
      then $\psi^\ast G \in \Perf \cX'$ is a classical generator.
  \end{enumerate}
\end{lemma}

\begin{proof}
  See \cite[Theorem~3.1.1]{MR1996800} for \cref{item:cg-i},
  \cite[Lemma~3.11]{1508.00682v2} for \cref{item:cg-ii} (the proof works verbatim for the semiseparated case),
  and \cite[Lemma 2.2]{MR4582884} for \cref{item:cg-iii}.
\end{proof}

This provides the following alternative description of the functor~$\dec_{\Delta_f}$, as it suffices to check vanishing of Ext's between generators to confirm semiorthogonality. This description will be used in the proof of \Cref{theorem:sod-f-lfp} given in \Cref{subsection:lfp-using-decompositions-of-diagonal}.

\begin{proposition}\label{pr:semiorthogonality via generator}
  Let $ f \colon \cX \to U $ be a morphism of schemes as in \Cref{situation:main}.
  Let $ \phi \colon V \to U $ be a morphism from an affine scheme, and fix a classical generator $ G \in \Perf \cX_V$.\footnote{Note that $\Perf \cX_{V}$ does indeed admit a classical generator: to see this, combine \Cref{lemma:bvdb-ko-p}\ref{item:cg-i} with \Cref{remark:properties}\ref{enumerate:proper-over-affine-is-qc-sep}.}
  Then
  \begin{equation}
    \dec_{\Delta_f}(\phi) =
    \Set{
      \begin{array}{c}
        \mathrm{distinguished~triangles~}K_2 \to \cO_{ \Delta_{f_V} } \to K_1 \to  K_2 [ 1 ]  \\
        \mathrm{~in~} \Perf \cY_V \mathrm{~such~that~}
        \RRlHom_{f_V}
        ( \Phi_{ K_2} ( G ), \Phi_{ K_1} ( G ) )
        = 0
      \end{array}
    } \Bigg{/} \emph{isomorphism},
  \end{equation}
  where the notion of isomorphism is the one spelled out in \eqref{eqn:iso-triangles}.
\end{proposition}

\section{The functor \texorpdfstring{$\sod_f$}{SOD} is limit-preserving}
\label{section:local-study}
The goal of this section is to prove that the functor $\sod_f=\sod_f^2$ of \Cref{definition:sod-on-all-schemes}
is limit-preserving.
This corresponds to \cref{item:limit-preserving} in \Cref{criterion:hall-rydh},
which appears in the proof of \Cref{theorem:sod-f}.

Recall, for any scheme $U$, the category $\Alg_U \subset \Sch_U^{\op}$ whose objects are pairs $(A,h)$ where $A$ is a ring and $h \colon \Spec A \to U$ is an affine morphism. We shall slightly abuse notation and write simply $A$ for an object of $\Alg_{U}$ throughout.
Next, recall the following definition, that we take from \cite[Definition~1.5]{MR0268188}.

\begin{definition}\label{definition:limit-preserving}
  Let~$U$ be a scheme.
  A functor~$\mathsf F \colon \Alg_U \to \Sets$
  is said to be \emph{limit-preserving}
  if for every functor
  \begin{equation}
    \label{functor-I}
    \begin{tikzcd}
      I \arrow{r} & \mathrm{Alg}_U, \quad i \mapsto A_{i}
    \end{tikzcd}
  \end{equation}
  from a directed set~$I$,
  the canonical map
  \begin{equation}
    \begin{tikzcd}
      \varinjlim \mathsf F(A_i)  \arrow{r} & \mathsf F( A )
    \end{tikzcd}
  \end{equation}
  is bijective,
  where we write~$A = \varinjlim_{ i \in I } A_i$.
  More generally, a functor~$\mathsf F \colon \Sch_U^{ \op } \to \Sets$
  is said to be limit-preserving
  if its composition with~$\Spec \colon \Alg_U \to \Sch_U^{ \op }$ is.
\end{definition}
This notion is called being \emph{locally of finite presentation} in \cite{MR0268188},
but \cref{definition:limit-preserving} is in line with \cite[\href{https://stacks.math.columbia.edu/tag/05LX}{Tag 05LX}]{stacks-project}.
The statement we will prove is the following.

\begin{theorem}
  \label{theorem:sod-f-lfp}
  Let $ f \colon \cX \to U $ be a morphism of schemes as in \Cref{situation:main}.
  Then the functor $\sod_f$ of \Cref{definition:presheaf-sod-f} is limit-preserving.
\end{theorem}
That the functor~$\sod_f^\ell$ is limit-preserving for all~$\ell$
will be obtained as part of the proof of \cref{theorem:sod-f}.

We present two proofs of \Cref{theorem:sod-f-lfp}. For both, we need the following preliminary lemma.

\begin{lemma}\label{lemma:0 homotopic at finite level}
  Let $(A_{i})_{i \in I}$ be a directed system of rings, and set $A =  \varinjlim_{ i \in I } A_i$. If a perfect complex $ P \in \Perf A_i$ has the property that
  $ P \otimes_{ A_i} A $ is acyclic,
  then there exists some~$j\geq i$ such that~$P\otimes_{A_i} A_j$ is acyclic.
\end{lemma}

\begin{proof}
  One has the following natural isomorphism
  for all finitely generated projective $ A_i $-modules $ M $ and $ N$,
  since both of them are direct summands of finitely generated free $ A_i $-modules:
  \begin{equation}
    \begin{tikzcd}
      \displaystyle\varinjlim_{ j \in I }
      \Hom_{ A_j } ( M \otimes_{ A_i } A_j, N \otimes_{ A_i } A_j )
      \arrow{r}{\sim} &
      \Hom_A ( M \otimes_{ A_i } A, N \otimes_{ A_i } A).
    \end{tikzcd}
  \end{equation}
  We can realize~$P$ using a bounded complex of projective~$A_i$-modules,
  hence the assumption on~$P$ is equivalent to this representative being contractible,
  and we obtain the existence of $ H \in \End^{ - 1 }_A ( P \otimes_{ A_i } A ) $
  such that $ d H + H d = \id_{ P \otimes_{ A_i } A }$.
  Since $P$ can be represented using a bounded complex of finitely generated projective $ A_i $-modules,
  it immediately follows from the observation above that $ H $
  lifts to some $ H_j \in \End^{ - 1 }_{ A_j } ( P \otimes_{ A_i } A_j ) $ for some $ j \ge i $,
  and moreover that the equality $ d H_j + H_j d = \id_{ P \otimes_{ A_i } A_j } $ holds possibly after replacing $ j $.
\end{proof}

\begin{corollary}\label{corollary:a priori isomorphy}
  Let
  \(
    f \colon \cX \to U
  \)
  be a flat and quasicompact morphism of schemes. Consider a directed system
  \(
    ( A_{i} )_{ i \in I }
  \)
  of
  \(
    U
  \)-rings, and set $A \defeq  \varinjlim_{ i \in I } A_i$ and $\cX_{i} \defeq \cX \times_{U}\Spec A_{i}$.
  Let
  \(
    \theta _{ i }
  \)
  be a morphism in the category
  \(
    \Perf \cX _{ i }
  \)
  whose pullback to
  \(
    \Perf \cX _{ A }
  \)
  is an isomorphism.
  Then there exists some~$j\geq i$ such that the pullback of~$\theta_i$ to~$\Perf \cX_j$ is an isomorphism.
\end{corollary}

\begin{proof}
  Being an isomorphism is equivalent to its cone being zero.
  The claim then follows from \Cref{lemma:0 homotopic at finite level}, since a perfect complex being acyclic is Zariski local and
  \(
    \cX _{ i }
  \)
  admits a finite affine open cover as it is quasicompact over the affine scheme
  \(
     \Spec A _{ i }
  \)
  and hence is itself quasicompact.
\end{proof}

\subsection{Proof of \texorpdfstring{\Cref{theorem:sod-f-lfp}}{limit-preserving} using classical generators}
\label{subsection:lfp-using-generators}

In this subsection we present a first proof of \Cref{theorem:sod-f-lfp}, using classical generators. This approach was suggested by the anonymous referee. Our original proof of \Cref{theorem:sod-f-lfp}, exploiting the functor $\dec_{\Delta_f}$, will be given in \Cref{subsection:lfp-using-decompositions-of-diagonal}.

\begin{proof}[Proof of \Cref{theorem:sod-f-lfp}]
  Because $U$ is quasicompact and semiseparated, and $f$ is proper, $\cX$ is also quasicompact and semiseparated.
  Therefore, by \Cref{lemma:bvdb-ko-p}\ref{item:cg-i},
  we can fix a classical generator~$G\in\Perf\cX$.
  Fix a functor as in \eqref{functor-I}.
  We need to prove that the natural map
  \begin{equation}
    \label{equation:natural-map-lfp}
    \begin{tikzcd}
      \varinjlim \sod_{f}(A_{i}) \arrow{r} & \sod_{f}(A)
    \end{tikzcd}
  \end{equation}
  is bijective.
  To fix notation, consider the cartesian diagrams
  \begin{equation}
    \label{eqn:2pullbacks}
    \begin{tikzcd}
      \Xcal_A\MySymb{dr}\arrow[swap]{d}{f_A}\arrow{r}{\rho_i} &
      \Xcal_i\MySymb{dr}\arrow{d}{f_i}\arrow{r} &
      \Xcal\arrow{d}{f} \\
      \Spec A\arrow{r}{\eta_i} &
      \Spec A_i\arrow{r} &
      U
    \end{tikzcd}
  \end{equation}
  defining the schemes $\cX_i \defeq \cX \times_U\Spec A_i$ and $\cX_A\defeq \cX \times_U \Spec A$.
  Let~$G_i$, resp.~$G_A$ denote the pullbacks of $G$ to~$\cX_i$, resp.~$\cX_A$.
  By \Cref{lemma:bvdb-ko-p}\ref{item:cg-ii},
  they are again classical generators,
  because~$\Spec A_i\to U$, resp.~$\Spec A\to U$ (and hence their pullbacks $\cX_i \to \cX$, resp.~$\cX_A \to \cX$) are affine.
  Consider the decomposition triangle
  \begin{equation}
    \begin{tikzcd}
      b\arrow{r} & G_A\arrow{r} & a\arrow{r} & b[1]
    \end{tikzcd}
  \end{equation}
  induced by a semiorthogonal decomposition~$\Perf\cX_A=\langle\cA_A,\cB_A\rangle$.
  By \Cref{proposition:3.3}
  we can first approximate~$b$ by some~$b_i$,
  and by \cref{proposition:3.2}
  we can lift the morphism~$b\to G_A$ to a morphism~$b_i\to G_i$
  for an appropriate~$i\in I$.
  Let~$a_i\defeq \cone(b_i\to G_i)$.
  Using the notation in \eqref{eqn:2pullbacks}, we have that
  \begin{equation}
    \rho_i^*(a_i)
    =\rho_i^*(\cone(b_i\to G_i))
    \cong\cone(\rho_i^*(b_i)\to\rho_i^*(G_i))
    \cong\cone(b\to G_A)
    =a,
  \end{equation}
  and thus by base change along \eqref{eqn:2pullbacks} we have
  \begin{equation}
    \begin{aligned}
      \eta_i^*\RRlHom_{f_i}(b_i,a_i)
      &\cong \RRlHom_{f_A}(\rho_i^*(b_i),\rho_i^*(a_i)) \\
      &\cong \RRlHom_{f_A}(b,a) \\
      &=0.
    \end{aligned}
  \end{equation}
  By \Cref{lemma:0 homotopic at finite level}
  applied to the perfect complex $\RRlHom_{f_i}(b_i,a_i)$ there exists a~$j\in I$
  such that $j \geq i$ and the pullbacks~$a_j\defeq a_{i}|_{\cX_{j}}$ and~$b_j\defeq b_{i}|_{\cX_{j}}$ of~$a_i$ and~$b_i$
  are semiorthogonal in~$\Perf\cX_j$.

  On the other hand,
  pulling back to $\cX_{j}$ the triangle~$b_i\to G_i\to a_i\to b_i[1]$
  to $\Perf \cX_{j}$ we obtain the triangle
  \begin{equation}
    \begin{tikzcd}
      b_j\arrow{r} & G_j\arrow{r} & a_j\arrow{r} & b_j[1].
    \end{tikzcd}
  \end{equation}
  Because~$G_j$ classically generates~$\Perf\cX_j$, we obtain that~$\cA_j\defeq\langle a_j\rangle$
  and~$\cB_j\defeq\langle b_j\rangle$
  give a semiorthogonal decomposition~$\Perf\cX_j=\langle\cA_j,\cB_j\rangle$.
  Finally,
  since we have that~$\rho_j^*(a_j)\cong\rho_i^*(a_i)\cong a\in\cA_A$,
  and similarly~$\rho_j^*(b_j)\cong\rho_i^*(b_i)\cong b\in\cB_A$,
  we obtain that
  $\cA_A=\langle\cA_j\rangle_{\eta_j}$
  and~$\cB_A=\langle\cB_j\rangle_{\eta_j}$,
  as required.
  This proves surjectivity of \eqref{equation:natural-map-lfp}.

  To prove injectivity,
  we will consider two semiorthogonal decompositions~$\langle\cA_i,\cB_i\rangle$
  and~$\langle\cA_i',\cB_i'\rangle$
  in~$\sod_f(A_i)$,
  for some~$i\in I$,
  and we assume that they are identified in~$\sod_f(A)$.

  For an index $k\geq i$, we denote by $\cB_{k}$, resp.~$\cB'_{k}$, the base change of $\cB_{i}$, resp.~$\cB'_{i}$, along $\Spec A_{k} \to \Spec A_{i}$.

  Let~$G_i\in\Perf\cX_i$ be a classical generator,
  and let~$b_i\in\cB_i$ be the projection,
  which is a classical generator by \cref{lemma:bvdb-ko-p}\ref{item:cg-ii}.
  We will show that there exists some~$k\geq i$
  for which~$b_i|_{\cX_k}\in\cB_k'$,
  so that the argument where we exchange the semiorthogonal decompositions
  shows that~$\cB_k=\cB_k'$.

  By the construction of the base change of the semiorthogonal decomposition~$\langle\cA_i',\cB_i'\rangle$ to~$\Spec A$ (see \eqref{def:A_phi})
  we have that~$\rho_i^*(b_i)$
  can be realized (up to taking direct summands) by taking iterated cones of morphisms between
  objects of the form~$\rho_i^*(b_i')$
  where~$b_i'\in\cB_i'$.
  These morphisms (in~$\Perf\cX_A$)
  can be lifted to morphisms in~$\Perf\cX_{A_j}$ for some~$j\geq i$
  by \Cref{proposition:3.2}.
  Thus we obtain an object $b _{ j } ' \in \cB_j'$
  such that $\rho_i^*(b_i)$ is a direct summand of
  \(
    \rho _{ j } ^{ \ast } ( b _{ j } ' )
  \).
  Let
  \(
    p
  \)
  be the endomorphism of
  \(
    \rho _{ j } ^{ \ast } ( b _{ j } ' )
  \)
  such that
  \(
    \rho _{ i } ^{ \ast } ( b _{ i } )
    =
    \im ( p )
  \).
  Again by \Cref{proposition:3.2}, replacing $j$ by a larger index if necessary, there exists a lift
  \(
    p _{ j } \in \End ( b _{ j } ' )
  \)
  of
  \(
    p
  \)
  which by \Cref{corollary:bijective} also satisfies the idempotency
  \(
    p _{ j } ^{ 2 }
    =
    p _{ j }
  \).
  Since
  \(
    \Perf \cX _{ j }
  \)
  is Karoubian, there exists a direct sum decomposition
  \(
    b _{ j } '
    =
    \im ( p _{ j } )
    \oplus
    \im ( 1 - p _{ j } )
  \)
  and hence
  \(
    \rho _{ j } ^{ \ast } \im ( p _{ j } )
    \simeq
    \im ( p )
    =
    \rho _{ i } ^{ \ast } ( b _{ i } )
  \).
  Note that
  \(
    \cB _{ j } '
  \)
  is closed under direct summands, hence
  \(
    \im ( p _{ j } )
    \in
    \cB _{ j } '
  \).
  Finally, again by \Cref{corollary:bijective} and \Cref{corollary:a priori isomorphy} there exists some~$k\geq j$ such that
  \(
    \im ( p _{ j } ) \vert _{ \cX _{ k } }
    \simeq
    b _{ i } \vert _{ \cX _{ k } }
  \)
  and hence the conclusion.
\end{proof}

\subsection{Proof of \texorpdfstring{\Cref{theorem:sod-f-lfp}}{limit-preserving} using decompositions of the diagonal}
\label{subsection:lfp-using-decompositions-of-diagonal}
In this section we give another proof of \Cref{theorem:sod-f-lfp}, exploiting \Cref{theorem:Iso_Functors_SOD_DEC}. To this end, we now introduce an auxiliary functor.

\begin{definition}
  Let $g\colon \cY \to U$ be a smooth and proper morphism of schemes.
  Let~$F\in\Perf\cY$.
  Define the functor
  \begin{equation}
    \label{functor_MOR}
    \begin{tikzcd}
      \MOR_{F}\colon \Sch_U^{\op} \arrow{r} & \Sets
    \end{tikzcd}
  \end{equation}
  by sending a $U$-scheme $V\to U$ to the set
  \begin{equation}
    \MOR_{F}(V \to U) \defeq
    \Set{\textrm{morphisms }K \xrightarrow{s} F_V \textrm{ in }\Perf \Ycal_V}\Big{/}\simeq
  \end{equation}
  where $F_V$ is the pullback of $F$ along the projection $\Ycal_V \to \Ycal$ and an isomorphism between
  $s\colon K \to F_V$ and $s'\colon K' \to F_V$ is an isomorphism $\theta\colon K\to K'$ such that $s'\circ\theta = s$.
\end{definition}

\begin{theorem}
  \label{theorem:Mor-limit-preserving}
  The functor $\MOR_{F}$ is limit-preserving.
\end{theorem}

In the proof, we shall use the following notation. For a $U$-algebra $B$, we set $\cY_{B} \defeq \cY \times_{U}\Spec B$, $F_{B}\defeq F|_{\cY_{B}}$ and $\MOR_{F}(B) \defeq \MOR_{F}(\Spec B \to U)$.

\begin{proof}
  Let $(A_{i})_{i \in I}$ be a directed system of $U$-algebras, and set $A=\varinjlim A_{i}$. We need to prove that the natural map
  \begin{equation}
    \label{mor-limit}
    \begin{tikzcd}
      \displaystyle \varinjlim \MOR_{F}(A_{i}) \arrow{r} & \MOR_{F}(A)
    \end{tikzcd}
  \end{equation}
  is bijective.

  To prove surjectivity of \eqref{mor-limit}, we fix an equivalence class
  \begin{equation}
    \begin{tikzcd}
      {[} K \arrow{r}{s} & F_{A} {]}\,\in\,\MOR_{F}(A)
    \end{tikzcd}
  \end{equation}
  of a morphism $s$ in $\Perf \cY_{A}$.
  We need to confirm that there exists an index $i \in I$ such that $s$ is, up to isomorphism, the restriction of a morphism
  \begin{equation}
    \begin{tikzcd}
      K_{i} \arrow{r}{s_{i}} & F_{A_{i}} \,\,\,\textrm{in }\Perf \cY_{A_{i}}
    \end{tikzcd}
  \end{equation}
  along the natural map $\cY_{A} \to \cY_{A_{i}}$. By \Cref{proposition:3.3}, we can construct a lift $K_{j}$ of $K$ over $\cY_{j}$ for some $j \in I$. In particular, we have $K \cong  K_{j}|_{\cY_{A}}$. By means of \Cref{proposition:3.2}, we can now lift $s$, viewed as a morphism $K_{j}|_{\cY_{A}} \to (F_{A_{j}} )|_{\cY_{A}}$, to a morphism $K_{j}|_{\cY_{A_{i}}} \to F_{A_{i}}$ for some $i \geq j$.

  To prove injectivity of \eqref{mor-limit}, suppose there is an index $i \in I$ and two isomorphism classes
  \begin{equation}
    \begin{tikzcd}
      {[}K_{i} \arrow{r}{s_{i}} & F_{A_{i}} {]}, \,\,\,   {[}K'_{i} \arrow{r}{s'_{i}} & F_{A_{i}} {]}\in\,\displaystyle \varinjlim \MOR_{F}(A_{i})
    \end{tikzcd}
  \end{equation}
  such that the restrictions of $s_{i}$ and $s'_{i}$ to $\cY_{A}$ define the same element in $\MOR_{F}(A)$. This means that, if we set $K = K_{i}|_{\cY_{A}}$ and $K' = K'_{i}|_{\cY_{A}}$, there exists an isomorphism $\theta \colon K \simto K'$ in $\Perf \cY_{A}$ such that $s' \circ \theta = s$, where $s = s_{i}|_{\cY_{A}}$ and $s' = s'_{i}|_{\cY_{A}}$. We need to prove that there is an index $j \geq i$ and an isomorphism
  \begin{equation}
    \label{wanted-iso-j}
    \begin{tikzcd}
      \theta_{j}\colon K_{j} \arrow{r}{\sim} & K'_{j}
    \end{tikzcd}
  \end{equation}
  in $\Perf \cY_{A_{j}}$ such that
  \begin{equation}
    \label{wanted-diag-j}
    \begin{tikzcd}
      K_{j} \arrow{rr}{\theta_{j}}\arrow[swap]{dr}{s_{j}} && K'_{j}\arrow{dl}{s'_{j}}  \\
      & F_{A_{j}}&
    \end{tikzcd}
  \end{equation}
  commutes, where we have set $K_{j} \defeq K_{i}|_{\cY_{A_{j}}}$ and $K'_{j} \defeq K'_{i}|_{\cY_{A_{j}}}$ and similarly $s_{j}\defeq s_{i}|_{\cY_{A_{j}}}$ and $s'_{j}\defeq s'_{i}|_{\cY_{A_{j}}}$. The existence of the lift $\theta_{j}$ of $\theta$ in \eqref{wanted-iso-j} follows either from \cite[Proposition 2.2.1]{MR2177199} or from the combination of \cref{proposition:3.3} and \Cref{corollary:a priori isomorphy}. The sought after compatibility \eqref{wanted-diag-j}, after possibly replacing $j$ with some $j'\geq j$, follows from the bijectivity in \Cref{corollary:bijective}.
  \end{proof}

Consider now our usual setup, namely a smooth and proper morphism of schemes $f \colon \cX \to U$. Set $\cY = \cX \times_U \cX$. For a direct limit of $U$-algebras $A = \varinjlim A_i$ we consider the fibre diagram
\begin{equation}
  \label{big-diag}
  \begin{tikzcd}
    \Ycal_A\MySymb{dr}\arrow{d}\arrow{r} &
    \Ycal_i\MySymb{dr}\arrow{d}\arrow{r} &
    \Ycal\arrow{d} \\
    \Xcal_A\MySymb{dr}\arrow[swap]{d}{f_A}\arrow{r}{\rho_i} &
    \Xcal_i\MySymb{dr}\arrow{d}{f_i}\arrow{r} &
    \Xcal\arrow{d}{f} \\
    \Spec A\arrow{r}{\eta_i} &
    \Spec A_i\arrow{r} &
    U
  \end{tikzcd}
\end{equation}
along with the perfect complex
\begin{equation}
  \Ocal_{\Delta_f} \in \Perf \Ycal,
\end{equation}
so that the functor $\MOR_{\Ocal_{\Delta_f}}$ from \eqref{functor_MOR} is well-defined.

\begin{proof}[Proof of \Cref{theorem:sod-f-lfp}]
  By \Cref{theorem:Iso_Functors_SOD_DEC}, we are reduced to proving that $\dec_{\Delta_{f}}$ is limit-preserving.
  The functor $\dec_{\Delta_f}$ is a subfunctor of the functor $\MOR_{\Ocal_{\Delta_f}}$,
  which is limit-preserving
  by \Cref{theorem:Mor-limit-preserving}. Then we have the following diagram of sets
  \begin{equation}
    \label{diag:DEC_F}
    \begin{tikzcd}
      \varinjlim \dec_{\Delta_f}(A_i) \arrow[hook]{d}\arrow{r}{\tau} &
      \dec_{\Delta_f}(A)\arrow[hook]{d} \\
      \varinjlim \MOR_{\Ocal_{\Delta_f}}(A_i) \arrow{r}{\sim} &
      \MOR_{\Ocal_{\Delta_f}}(A)
    \end{tikzcd}
  \end{equation}
  implying the injectivity of the map $\tau$. We need to show its surjectivity.
  Pick an element
  \begin{equation}
    \begin{tikzcd}
      \zeta\colon
      K \arrow{r}{s} & \Ocal_{\Delta_{f_A}} \arrow{r} &  L \arrow{r} & K [ 1 ]
    \end{tikzcd}
  \end{equation}
  of~$\dec_{\Delta_f}(A)$.
  We know, by definition of~$\dec_{\Delta_f}$, that
  \begin{equation}
    \RRlHom_{f_A}(\mathscr E_K,\mathscr E_L) = 0.
  \end{equation}
  Now use Diagram \eqref{diag:DEC_F} to lift $\zeta$ to an element
  \begin{equation}
    \begin{tikzcd}
      \zeta_i
      \colon
      K_i \arrow{r}{s_{i}} & \Ocal_{\Delta_{f_i}} \arrow{r} & L_i \arrow{r} &  K_i [ 1 ]
    \end{tikzcd}
  \end{equation}
  of~$\MOR_{\Ocal_{\Delta_f}}(A_i)$,
  for some $i$.
  To show that $\zeta_i$ can be lifted to an element of $\dec_{\Delta_f}(A_j)$ for some $j \geq i$, fix a classical generator $G_{A_i}$ of $ \Perf \cY_i $ (note that $ \cY_i $ is quasicompact and semiseparated so $G_{A_{i}}$ exists by \Cref{lemma:bvdb-ko-p}\ref{item:cg-i}), and define the two objects
  \begin{equation}
    k_i \defeq \Phi_{ K_i } ( G_{ A_i} )\in\mathscr E_{K_i} \subseteq \Perf \Xcal_i,
    \quad
    l_i \defeq \Phi_{ L_i } ( G_{ A_i} )
    \in \mathscr E_{L_i}\subseteq \Perf \Xcal_i.
  \end{equation}
  Consider their (derived) pullbacks
  \begin{equation}
    \rho_i^*k_i \in \mathscr E_K \subseteq \Perf \Xcal_A,\quad  \rho_i^*l_i \in \mathscr E_L \subseteq \Perf \Xcal_A.
  \end{equation}
  Then we know by \Cref{pr:semiorthogonality via generator}, and base change along \eqref{big-diag}, that
  \begin{equation}
    0 = \RRlHom_{f_A}(\rho_i^*k_i,\rho_i^*l_i)
    =  f_{A,\ast}\rho_i^*\RRlHom_{\Xcal_i}(k_i,l_i)
    \cong  \eta_i^\ast  \RRlHom_{f_i}(k_i,l_i).
  \end{equation}
  Note that $ P \defeq \RRlHom_{f_i}(k_i,l_i)\in \Perf A_i$,
  since $ f_{ i,\ast } $ preserves perfect complexes (see the proof of \Cref{lemma:Diagonal_Perfect}).
  Hence by applying \Cref{lemma:0 homotopic at finite level} to the perfect complex $ P $ above, it follows that
  \begin{equation}
    \label{eq:0 at the level j}
    0
    =
    \RRlHom_{f_i}(k_i,l_i) \otimes_{ A_i} A_j
  \end{equation}
  for some $ j \ge i$.

  Let $ \rho_{ i,j } \colon \cX_j \to \cX_i $ be the base change of $ \eta_{ i,j } \colon \Spec A_j \to \Spec A_i$.
  Then the right-hand side of \eqref{eq:0 at the level j} can be computed as follows, by repeatedly using the base change isomorphisms
  \begin{equation}
    \RRlHom_{f_j}(\rho_{ i,j }^* \Phi_{ K_i} ( G_{ A_i}), \rho_{ i,j }^{\ast} \Phi_{ L_i} ( G_{ A_i})
    )
    \cong
    \RRlHom_{f_j}(\Phi_{ K_j} ( G_{ A_j}),\Phi_{ L_j} ( G_{ A_j})).
  \end{equation}
  Here $ K_j = K_i \otimes_{ A_i} A_j $ and $ L_j = L_i \otimes_{ A_i} A_j$,
  respectively.
  By \Cref{lemma:bvdb-ko-p}\ref{item:cg-iii}, this, in turn, implies that $ \zeta_i \otimes_{ A_i} A_j\in\dec_{ \Delta_f } ( A_j )$.
\end{proof}

\section{Deformations of semiorthogonal decompositions}
\label{section:deform_sod}

The goal of this section
is to prove \cref{theorem:deforming-sods}
on the existence and uniqueness of (formal) deformations of semiorthogonal decompositions.

The approach taken in \cref{subsection:extending-sods}
using filtrations
was suggested by the referee.
It provides an alternative to the original and more involved argument,
which used the perspective of Fourier--Mukai transforms,
the comparison from \cref{subsec:sod=dec},
and the deformation theory of morphisms with fixed lift of the argument
relative to deformations of abelian categories.
This alternative approach,
and the necessary deformation theory
which is of independent interest,
is included as an application of these
deformation-theoretical results in \cite{former-appendix}.

\subsection{Extending semiorthogonal decompositions to infinitesimal neighbourhoods}
\label{subsection:extending-sods}
Let~$(R,\mathfrak{m},\bfk)$ be a local noetherian ring.
Let~$\cX\to\Spec R$ be a smooth and proper morphism.
For~$n\geq 0$
we write~$R_n\defeq R/\mathfrak{m}^{n+1}$,
so that~$R_0=\bfk$.
Set~$\cX_n\defeq\cX\times_{\Spec R}\Spec R_n$,
and let
\begin{equation}
  \begin{tikzcd}
    i\colon\cX_0\arrow[hook]{r} & \cX_n
  \end{tikzcd}
\end{equation}
be the natural embedding.

First,
we study~$\derived^\bounded(\cX_n)$
for a fixed~$n\geq 0$.
\begin{lemma}
  \label{lemma:4.1}
  Let~$n\geq 0$.
  For any~$F\in\derived^\bounded(\cX_n)$
  there exists
  a chain of maps
  \begin{equation}
    \begin{tikzcd}
      0=F_{-1}\arrow{r} & F_0\arrow{r} &\cdots\arrow{r} & F_{n-1}\arrow{r} & F_n=F
    \end{tikzcd}
  \end{equation}
  in~$\derived^\bounded(\cX_n)$,
  such that for all~$k=0,\ldots,n$
  there exists an isomorphism~$\cone(F_{k-1}\to F_k)\cong i_* i^*F$.
\end{lemma}

\begin{proof}
  Consider the filtration of the structure sheaf of the diagonal
  \begin{equation}
    \begin{tikzcd}
      0
      \arrow[hook]{r} & \cO_{\Delta_{\cX_0}}
      \arrow[hook]{r} & \cO_{\Delta_{\cX_1}}
      \arrow[hook]{r} & \cdots
      \arrow[hook]{r} & \cO_{\Delta_{\cX_{n-1}}}
      \arrow[hook]{r} & \cO_{\Delta_{\cX_n}}
    \end{tikzcd}
  \end{equation}
  with all factors isomorphic to~$\cO_{\Delta_{\cX_0}}$.
  Applying the corresponding Fourier--Mukai functors to~$F$,
  we obtain the required filtration.
\end{proof}

\begin{lemma}
  \label{lemma:4.2}
  There exists a sequence of functors~$\Phi_k\colon\Perf\cX_0\to\Perf\cX_0$,
  for~$k\geq 0$,
  such that
  for all~$F\in\Perf\cX_0$
  we have that
  \begin{equation}
    \label{equation:4.3}
    i^* i_*F
    \cong
    \hocolim\Phi_k(F).
  \end{equation}
  Moreover,
  $\Phi_k$ preserves any triangulated subcategory of~$\Perf\cX_0$.
\end{lemma}

\begin{proof}
  The functor~$i^* i_*\colon\derived(\cX_0)\to\derived(\cX_0)$ is a Fourier--Mukai functor,
  whose kernel is given by
  \begin{equation}
    \eta^*\eta_*\cO_{\Delta_{\cX_0}}
    \in
    \derived_{\mathrm{coh}}^-(\mathcal O_{\cX_0\times_{\bfk}\cX_0}),
  \end{equation}
  where~$\eta\colon\cX_0\times_{\bfk}\cX_0\hookrightarrow\cX_n\times_{\Spec R_n}\cX_n$ is the natural embedding.
  For~$k\geq 0$ we define
  \begin{equation}
    K_k\defeq\tau^{\geq-k}(\eta^*\eta_*\cO_{\Delta_{\cX_0}})
    \in
    \derived^\bounded(\cX_0\times_{\bfk}\cX_0),
  \end{equation}
  the canonical truncation of~$\eta^*\eta_*\cO_{\Delta_{\cX_0}}$.
  Then
  \begin{equation}
    \eta^*\eta_*\cO_{\Delta_{\cX_0}}\cong\hocolim K_k,
  \end{equation}
  hence \eqref{equation:4.3} holds,
  if we define~$\Phi_k$
  as the Fourier--Mukai functor with kernel~$K_k$.

  It remains to show that~$\Phi_k$ preserves any triangulated subcategory of~$\Perf\cX_0$
  (and in particular, preserves~$\Perf\cX_0$ itself).
  For this,
  note that there are distinguished triangles
  \begin{equation}
    \begin{tikzcd}
      \mathrm{L}_k\eta^*(\eta_*\cO_{\Delta_{\cX_0}})[k]
      \arrow{r} & K_k
      \arrow{r} & K_{k-1}
      \arrow{r} &\mathrm{L}_k\eta^*(\eta_*\cO_{\Delta_{\cX_0}})[k+1]
    \end{tikzcd}
  \end{equation}
  for~$k\geq 0$, with~$K_{-1}=0$.
  Hence, by induction on~$k$,
  using the triangle of Fourier--Mukai functors
  induced by a triangle of kernels,
  it is enough to check that the Fourier--Mukai functor
  with kernel~$\mathrm{L}_k\eta^*(\eta_*\cO_{\Delta_{\cX_0}})$
  preserves any triangulated subcategory of~$\Perf\cX_0$.
  By the projection formula \cite[Proposition~3.9.4]{MR2490557}
  we have that
  \begin{equation}
    \eta_*\left( \mathbf{L}\eta^*\eta_*\cO_{\Delta_{\cX_0}} \right)
    \cong
    \eta_*\left( \mathbf{L}\eta^*\eta_*\cO_{\Delta_{\cX_0}}\otimes\cO_{\cX_0\times_{\bfk}\cX_0} \right)
    \cong
    \eta_*\cO_{\Delta_{\cX_0}}\otimes^{\mathbf{L}}\eta_*\cO_{\cX_0\times_{\bfk}\cX_0}.
  \end{equation}
  Taking the~$k$th cohomology object on the right,
  we obtain
  \begin{equation}
    \eta_*\mathrm{L}_k\eta^*(\eta_*\cO_{\Delta_{\cX_0}})
    \cong
    \operatorname{Tor}_k^{\cO_{\cX_n\times_{\Spec R_n}\cX_n}}(\eta_*\cO_{\Delta_{\cX_0}},\eta_*\cO_{\cX_0\times_{\bfk}\cX_0}).
  \end{equation}
  Moreover,
  we have that~$\eta_*\cO_{\cX_0\times_{\bfk}\cX_0}\cong f_n^*R_0$,
  where we write~$f_n\colon\cX_n\times_{\Spec R_n}\cX_n\to\Spec R_n$ for the structure morphism,
  and~$R_0$ is considered as an object in~$\derived^\bounded(\Spec R_n)$.
  Hence,
  we can use a minimal free resolution
  \begin{equation}
    \begin{tikzcd}
      \cdots
      \arrow{r} & R_n^{\oplus m_2}
      \arrow{r} & R_n^{\oplus m_1}
      \arrow{r} & R_n^{\oplus m_0}
      \arrow{r} & R_0\arrow{r} & 0
    \end{tikzcd}
  \end{equation}
  (here~$m_0=1$ and~$m_i<+\infty$ because~$R$ is noetherian)
  to get a locally free resolution
  \begin{equation}
    \begin{tikzcd}
      \cdots
      \arrow{r} & \cO_{\cX_n\times_{\Spec R_n}\cX_n}^{\oplus m_2}
      \arrow{r} & \cO_{\cX_n\times_{\Spec R_n}\cX_n}^{\oplus m_1}
      \arrow{r} & \cO_{\cX_n\times_{\Spec R_n}\cX_n}
      \arrow{r} & R_0\arrow{r} & 0.
    \end{tikzcd}
  \end{equation}
  This implies that~$ \eta_*\mathrm{L}_k\eta^*(\eta_*\cO_{\Delta_{\cX_0}})\cong\eta_*\cO_{\Delta_{\cX_0}}^{\oplus m_k}$.
  Thus,
  $\mathrm{L}_k\eta^*(\eta_*\cO_{\Delta_{\cX_0}})\cong\cO_{\Delta_{\cX_0}}^{\oplus m_k}$,
  and we see that the corresponding Fourier--Mukai functor
  takes any~$F\in\Perf\cX_0$ to~$F^{\oplus m_k}$,
  hence preserves any triangulated subcategory of~$\Perf\cX_0$,
  as required.
\end{proof}

We can now conclude how
a semiorthogonal decomposition
\begin{equation}
  \label{equation:4.4}
  \derived^\bounded(\cX_0)
  =
  \Perf\cX_0
  =
  \langle\cA_0,\cB_0\rangle
\end{equation}
(recall that~$\cX_0$ is smooth over~$\bfk$)
induces one for~$\derived^\bounded(\cX_n)$.
Indeed,
consider the triangulated hulls~$\langle i_*\cA_0\rangle$ and~$\langle i_*\cB_0\rangle$
in~$\derived^\bounded(\cX_n)$.
Then we have the following.

\begin{proposition}
  \label{proposition:4.5}
  There is a semiorthogonal decomposition
  \begin{equation}
    \derived^\bounded(\cX_n)
    =
    \left\langle
    \langle i_*\cA_0\rangle,
    \langle i_*\cB_0\rangle
    \right\rangle.
  \end{equation}
\end{proposition}

\begin{proof}
  First we check semiorthogonality.
  It is enough to check that~$i_*\cA_0$ and~$i_*\cB_0$
  are semiorthogonal.
  So we take~$a_0\in\cA_0$,
  $b_0\in\cB_0$,
  and note that
  \begin{equation}
    \begin{aligned}
      \Hom(i_*b_0,i_*a_0)
      &\cong\Hom(i^* i_*b_0,a_0) \\
      &\cong\Hom(\hocolim\Phi_k(b_0),a_0) \\
      &\cong\holim\Hom(\Phi_k(b_0),a_0) \\
      &=0,
    \end{aligned}
  \end{equation}
  where the first isomorphism is the adjunction \eqref{adjunction},
  the second is \eqref{equation:4.3},
  the third is the standard property of the homotopy colimit,
  and the last follows from \eqref{equation:4.4}
  because~$\Phi_k$ preserves~$\cB_0$ by \cref{lemma:4.2}.

  On the other hand,
  the fact that~$i_*\cA_0$ and~$i_*\cB_0$ generate~$\derived^\bounded(\cX_n)$
  follows immediately from \cref{lemma:4.1}
  and the semiorthogonal decomposition \eqref{equation:4.4}.
\end{proof}

Next,
we explain how to induce a semiorthogonal decomposition of~$\Perf\cX_n$
from \eqref{equation:4.4}.
Given a subcategory~$\cC_0\subset\Perf\cX_0$
we denote by~$\widehat{\cC_0}\subset\derived_{\mathrm{qcoh}}(\mathcal O_{\cX_0})$
the localizing
(i.e., triangulated and closed under arbitrary coproducts)
subcategory generated by~$\cC_0$.
By \cite[Proposition~4.2]{MR2801403}
we have the induced semiorthogonal decomposition
\begin{equation}
  \label{equation:4.7}
  \derived_{\mathrm{qcoh}}(\mathcal O_{\cX_0})
  =
  \langle
  \widehat{\cA_0},
  \widehat{\cB_0}
  \rangle.
\end{equation}
We also have the following observation.
\begin{lemma}
  \label{lemma:4.8}
  We have
  \begin{equation}
    i^*\langle i_*\cA_0\rangle\subset\widehat{\cA_0}
  \end{equation}
  and
  \begin{equation}
    i^*\langle i_*\cB_0\rangle\subset\widehat{\cB_0}.
  \end{equation}
\end{lemma}

\begin{proof}
  Indeed, if~$a_0\in\cA_0$
  then~$i^* i_*a_0\cong\hocolim\Phi_k(a_0)$ by \eqref{equation:4.3}
  and~$\Phi_k(a_0)\in\cA_0$ by \cref{lemma:4.2},
  hence~$i^* i_*a_0\in\widehat{\cA_0}$.
  The same argument works for~$\cB_0$.
\end{proof}

Now we define the subcategories~$\cA_n$ and~$\cB_n$ of~$\Perf\cX_n$ by
\begin{equation}
  \label{equation:4.9}
  \begin{aligned}
    \cA_n&\defeq\{F\in\Perf\cX_n\mid i^*F\in\cA_0\} \\
    \cB_n&\defeq\{F\in\Perf\cX_n\mid i^*F\in\cB_0\}
  \end{aligned}.
\end{equation}

\begin{lemma}
  \label{lemma:4.10}
  We have~$\cA_n\subset\langle i_*\cA_0\rangle$,
  and~$\cB_n\subset\langle i_*\cB_0\rangle$.
\end{lemma}

\begin{proof}
  If~$F\in\cA_n$
  then~$i^*F\in\cA_0$ by the construction in \eqref{equation:4.9},
  and thus~$F\in\langle i_*\cA_0\rangle$
  by \cref{lemma:4.1}.
  The same argument works for~$\cB_n$.
\end{proof}

We now come to the main result of this section.
\begin{theorem}
  \label{proposition:4.11}
  Consider the semiorthogonal decomposition \eqref{equation:4.4}.
  There is an induced semiorthogonal decomposition
  \begin{equation}
    \Perf\cX_n
    =
    \langle\cA_n,\cB_n\rangle.
  \end{equation}
\end{theorem}

\begin{proof}
  First we check semiorthogonality.
  By \cref{lemma:4.10}
  it is enough to check that~$i_*\cA_0$
  and~$\cB_n$ are semiorthogonal.
  Take~$a_0\in\cA_0$ and~$b\in\cB_n$,
  and note that~$\Hom(b,i_*a_0)\cong\Hom(i^*b,a_0)$
  by the adjunction isomorphism \eqref{adjunction}.
  Since~$i^*b\in\cB_0$ by \eqref{equation:4.9},
  we have that~$\Hom(b,i_*a_0)=0$
  because \eqref{equation:4.4} is a semiorthogonal decomposition of~$\Perf\cX_0$.

  Now let~$E\in\Perf\cX_n\subset\derived^\bounded(\cX_n)$
  (we have this inclusion of categories by \cite[\href{https://stacks.math.columbia.edu/tag/0FXU}{Tag 0FXU}]{stacks-project},
  as~$\cX_n$ is smooth and proper over~$\Spec R_n$, where~$R_n$ is noetherian).
  By \cref{proposition:4.5}
  there is a distinguished triangle
  \begin{equation}
    \begin{tikzcd}
      b\arrow{r} & E\arrow{r} & a\arrow{r} & b[1],
    \end{tikzcd}
  \end{equation}
  where~$i^*a\in i^*\langle i_*\cA_0\rangle\subset\widehat{\cA_0}$
  and~$i^*b\in i^*\langle i_*\cB_0\rangle\subset\widehat{\cB_0}$
  by \cref{lemma:4.8}.

  On the other hand,
  we have~$i^*E\in\Perf\cX_0$.
  Using the semiorthogonal decomposition~\eqref{equation:4.4},
  the obvious embeddings~$\cA_0\subset\widehat{\cA_0}$ and~$\cB_0\subset\widehat{\cB_0}$,
  and the uniqueness of the decomposition triangle for \eqref{equation:4.7},
  we conclude that~$i^*a\in\cA_0$
  and~$i^*b\in\cB_0$.
  In particular,
  we have~$i^*a,i^*b\in\Perf\cX_0$,
  hence~$a,b\in\Perf\cX_n$
  by \cite[Lemma~2.11]{MR4395103},
  and therefore~$a\in\cA_n$ and~$b\in\cB_n$
  by the definition in \eqref{equation:4.9}.
  This proves that~$\cA_n$ and~$\cB_n$
  generate~$\Perf\cX_n$.
\end{proof}

\subsection{Formal deformations and \'etale-local behaviour of semiorthogonal decompositions}
The next result corresponds to \Cref{item:DEF} of \Cref{criterion:hall-rydh},
and thus provides the final point we would like to verify, at least for semiorthogonal decompositions of length $\ell=2$.

\begin{theorem}
  \label{theorem:deforming-sods}
  Let $ f \colon \cX \to U $ be a morphism of schemes as in \Cref{situation:main}.
  Let $(R,\mathfrak m,\bfk)$ be a local noetherian ring that is $\mathfrak m$-adically complete.
  For every morphism $\Spec R \to U$, the natural map
  \begin{equation}
    \begin{tikzcd}
      \vartheta\colon \sod_f(\Spec R \to U) \arrow{r} & \sod_f(\Spec \bfk \to U)
    \end{tikzcd}
  \end{equation}
  is bijective.
\end{theorem}

\begin{proof}
  Denote~$R_n\defeq R/\mathfrak{m}^{n+1}$ as in \cref{subsection:extending-sods}.
  Because~$R$ is complete noetherian
  and~$\sod_f$ is limit-preserving (cf.~\Cref{theorem:sod-f-lfp}),
  we have that
  \begin{equation}
    \sod_f(\Spec R \to U)\cong\displaystyle\varprojlim\sod_f(\Spec R_n \to U),
  \end{equation}
  and it suffices to show that the natural maps
  \begin{equation}
    \begin{tikzcd}
      \label{equation:infinitesimal-restriction-map}
      \sod_f(\Spec R_n \to U)\arrow{r} & \sod_f(\mathbf{\Spec k} \to U)
    \end{tikzcd}
  \end{equation}
  are all isomorphisms.
  The surjectivity follows from \cref{proposition:4.11}.

  To prove that the restriction map \eqref{equation:infinitesimal-restriction-map}
  is also injective,
  let
  \begin{equation}
    \label{equation:another-deformation}
    \Perf\cX_n=\langle\cA_n',\cB_n'\rangle
  \end{equation}
  be a semiorthogonal decomposition
  whose restriction along~$i\colon\cX_0\hookrightarrow\cX_n$
  coincides with \eqref{equation:4.4}.
  We wish to show that an object~$F\in\Perf\cX_n$
  belongs to~$\cA_n'$
  if and only if~$i^*F\in\cA_0$,
  and thus that~$\cA_n=\cA_n'$
  by the construction of the lift in \eqref{equation:4.9}.

  If~$F\in\cA_n'$,
  then~$i^*F\in\cA_0$
  by definition of the base change.
  Conversely,
  assume that~$i^*F\in\cA_0$.
  Consider the decomposition triangle
  \begin{equation}
    \begin{tikzcd}
      b_n'\arrow{r} & F\arrow{r} & a_n'\arrow{r} & b_n'[1]
    \end{tikzcd}
  \end{equation}
  for the semiorthogonal decomposition~\eqref{equation:another-deformation}.
  Restricting along~$i^*$ we obtain the distinguished triangle
  \begin{equation}
    \begin{tikzcd}
      i^*b_n'\arrow{r} & i^*F\arrow{r} & i^*a_n'\arrow{r} & i^*b_n'[1],
    \end{tikzcd}
  \end{equation}
  and because~$i^*F,i^*a_n'\in\cA_0$
  we obtain that~$i^*b_n'=0$.
  But then by \cref{lemma:derived-nakayama-global}
  we have that~$b_n'=0$,
  and thus~$F\in\cA_n'$ as required.
\end{proof}

\subsection{Proof of \Cref{theorem:geometric-deformation}}
We will now prove \Cref{corollary:geometric-deformation}
(which is a generalisation of \Cref{theorem:geometric-deformation})
which allows us to deform a semiorthogonal
decomposition on $X=f^{-1}(0)$ to an \'etale neighbourhood of $0\in U$.
For this we recall a version of Artin approximation,
originally stated (in a more restricted version) in \cite[Corollary~2.2]{MR0268188}.

\begin{theorem}[Artin approximation]
  \label{theorem:artin-approximation}
  Let $f\colon\cX\to U$ be a morphism of schemes as in \Cref{situation:main}.
  Let $\mathsf F\colon \Sch_U^{\op}\to \Sets$ be a functor which is limit-preserving.
  Let $0\in U$ be a point such that~$\mathcal{O}_{U,0}$ is a G-ring,\footnote{
    See \cite[\href{https://stacks.math.columbia.edu/tag/07GG}{Tag 07GG}]{stacks-project} for the definition.
  This holds for all points if~$U$ is excellent.}
  and take $\overline{\xi} \in \mathsf F(\Spec \widehat{\Ocal}_{U,0})$.
  Then there exists an \'etale neighbourhood $U'\to U$ of $0$ and an element $\xi'\in \mathsf F(U')$ such that
  \begin{equation}
    \label{eq:Artin approximation}
    \xi' \equiv \overline{\xi}\,\, (\bmod\,\, \mathfrak m_{U,0}).
  \end{equation}
\end{theorem}

\begin{proof}
  This is the case~$N=0$ of \cite[Theorem~B.5.18]{alper-stacks-and-moduli}.
\end{proof}

Let us explain the statement of \Cref{theorem:artin-approximation}.
For a ring $A$, set $\mathsf F(A)=\mathsf F(\Spec A)$, for brevity.
Then (using that $\mathsf F$ is contravariant) an \'etale neighbourhood $\Spec \bfk (0) \to U' \to U$ induces a commutative diagram of sets
\begin{equation}
  \begin{tikzcd}
    \mathsf F(\widehat{\Ocal}_{U,0}/\mathfrak m_{U,0}) & \mathsf F(\widehat{\Ocal}_{U,0})\arrow{l}\isoarrow{d} & \mathsf F(\Ocal_{U,0})\arrow{l}\arrow{d} & \mathsf F(U)\arrow{l}\arrow{d} \\
    & \mathsf F(\widehat{\Ocal}_{U',u'}) & \mathsf F(\Ocal_{U',u'})\arrow{l} & \mathsf F(U')\arrow{l}
  \end{tikzcd}
\end{equation}
where $u'\in U'$ is the image of $\Spec \bfk(0) \to U'$.
The condition~\eqref{eq:Artin approximation}
means that $\overline{\xi} \in \mathsf F(\widehat{\Ocal}_{U,0})$ and $\xi'\in\mathsf F(U')$
go to the same element in $\mathsf F(\widehat{\Ocal}_{U,0}/\mathfrak m_{U,0})$ under the given maps.

We can now prove that a semiorthogonal decomposition always uniquely deforms to the nearby fibres in the following sense.

\begin{corollary}
  \label{corollary:geometric-deformation}
  Let $ f \colon \cX \to U $ be a morphism of schemes as in \Cref{situation:main}.
  Let $0\in U$ be a point such that~$\mathcal{O}_{U,0}$ is a G-ring,
  and let $X = f^{-1}(0) \subset \cX$ denote the fibre of $f$ over $0$.
  Assume that the derived category $\derived^\bounded(X)\defeq \derived^\bounded(\coh X)$ of $X$ admits a $\bfk(0)$-linear semiorthogonal decomposition
  \begin{equation}
    \label{equation:sod-central-fibre}
    \derived^\bounded(X) = \braket{\cA,\cB}.
  \end{equation}
  Then, shrinking $U$ to an \'etale neighbourhood of $0$ if necessary, one finds a unique $U$\dash linear semiorthogonal decomposition
  \begin{equation}
    \Perf \Xcal = \braket{\cA_U,\cB_U}
  \end{equation}
  whose base change to $\Spec \bfk(0) \to U$ is the initial semiorthogonal decomposition \eqref{equation:sod-central-fibre}.
\end{corollary}

\begin{proof}
  Let $(R,\mathfrak m,\bfk)$ be the completion of the local ring $\mathcal O_{U,0}$, and consider the closed immersion
  \begin{equation}
    \begin{tikzcd}
      \Spec R/\mathfrak m \arrow[hook]{r} &  \Spec R.
    \end{tikzcd}
  \end{equation}
  By \cref{theorem:deforming-sods} we know that there is a bijection
  \begin{equation}
    \begin{tikzcd}
      \sod_f(\Spec R \to U) \arrow{r}{\sim} & \sod_f(\Spec R/\mathfrak m \to U).
    \end{tikzcd}
  \end{equation}
  Let
  \begin{equation}
    \xi_{\mathfrak m} \in \sod_f(\Spec R/\mathfrak m\to U)
  \end{equation}
  be the element corresponding to \eqref{equation:sod-central-fibre}.
  Now, $\sod_f$ is limit-preserving by \Cref{theorem:sod-f-lfp}.
  Therefore, by \Cref{theorem:artin-approximation}, there exists an \'etale neighbourhood $U'\to U$
  of the point $0\in U(\bfk)$ and a $U'$-linear semiorthogonal decomposition
  \begin{equation}
    \xi' \in \sod_f(U'\to U)
  \end{equation}
  deforming the semiorthogonal decomposition \eqref{equation:sod-central-fibre} on the central fibre.
\end{proof}

As a special case of \cref{corollary:geometric-deformation}
we obtain \cref{theorem:geometric-deformation}.

\begin{remark}
  \label{remark:hu-comparison}
  In \cite[Theorem~3.5]{1805.04050v4} Hu obtained a version of \Cref{corollary:geometric-deformation}
  valid in the special case of semiorthogonal decompositions consisting of exceptional objects.
\end{remark}

\section{Moduli spaces of semiorthogonal decompositions}
\label{section:moduli}
In this section we combine the results on the deformation theory of semiorthogonal decompositions
and the functor~$\sod_f$ to complete the proof of \Cref{theorem:main-intro},
and give some examples of its geometric properties.

\subsection{Proof of \texorpdfstring{\Cref{theorem:main-intro}}{main theorem}: checking Artin's axioms}
The following theorem is the main result of this paper.

\begin{theorem}
  \label{theorem:sod-f}
  Let~$U$ be a quasicompact and semiseparated scheme,
  and let~$f\colon\cX\to U$ be a smooth and proper morphism of schemes. Fix an integer $\ell \geq 2$.
  If we assume moreover that~$U$ is excellent,
  then~$\sod_f^{\ell}$ is a nonempty algebraic space
  which is \'etale over~$U$.
\end{theorem}

The proof of \Cref{theorem:sod-f} consists in checking the conditions in \Cref{criterion:hall-rydh},
first for~$\ell=2$ (writing~$\sod_f=\sod_f^2$, as ever),
and then by induction for all~$\ell$.

\begin{proof}[Proof of \Cref{theorem:sod-f}]
  By construction (see also \Cref{theorem:sod-f-fpqc-sheaf} for the main step),
  $\sod_f$ satisfies \cref{item:etale-sheaf} of \Cref{criterion:hall-rydh},
  and \Cref{theorem:sod-f-lfp} proves \Cref{item:limit-preserving} in \cref{criterion:hall-rydh}.

  To check \Cref{item:DEF} in \cref{criterion:hall-rydh},
  we choose a complete local noetherian $U$-ring $R$
  such that the induced morphism~$\Spec R/\mathfrak{m}\to U$ is of finite type,
  and let
  \begin{equation}
    \begin{tikzcd}
      \vartheta\colon \mathsf{SOD}_f(\Spec R \to U) \arrow{r} & \mathsf{SOD}_f(\Spec R/\mathfrak m \to U)
    \end{tikzcd}
  \end{equation}
  be the canonical map.
  This map is bijective by \Cref{theorem:deforming-sods}
  (note that injectivity can also be proved by means of \Cref{lm:coincidence of SODs can be checked at closed points}).

  Finally, the algebraic space~$\sod_f$ is nonempty
  because we always have the trivial semiorthogonal decompositions $\langle 0,\Perf \cX_{V}\rangle = \Perf \cX_{V} = \langle \Perf \cX_{V},0\rangle$,
  which are automatically $V$-linear. This settles the case $\ell = 2$.

  \medskip

  For the induction step we will consider the morphism of sheaves
  \begin{equation}
    \begin{tikzcd}
      \sod_f^\ell\arrow{r}{\vartheta} & \sod_f^{\ell-1}\times\sod_f
    \end{tikzcd}
  \end{equation}
  defined (on~$\Aff_U$, and subsequently extended using \cref{lemma:equivalence-of-toposes}, as earlier)
  by sending (for an affine~$U$-scheme~$\phi\colon V\to U$)
  the~$V$-linear semiorthogonal decomposition~$(\cA^1,\cA^2,\ldots,\cA^\ell) \in \sod_{f}^{\ell}(\phi)$
  of~$\Perf\cX_V$
  to the pair of~$V$-linear semiorthogonal decompositions
  \begin{equation}
    \left(
      (\langle \cA^1, \cA^2\rangle, \cA^3, \ldots,\cA^\ell),
      (\cA^1, \langle \cA^2, \cA^{ 3 }, \ldots,\cA^\ell \rangle)
    \right).
  \end{equation}
  Observe that we have an equality
  \begin{equation}
    \cA^{ 2 }
    =
    \langle \cA^1, \cA^2 \rangle
    \cap
    \langle \cA^2, \cA^{ 3 }, \ldots,\cA^\ell \rangle.
  \end{equation}
  This shows that~$\vartheta(\phi)$ is injective,
  and that the image consists of those pairs of~$V$-linear semiorthogonal decompositions
  \begin{equation}\label{eq:a pair of SOD}
    \left(
      (\cB^1, \dots, \cB^{ \ell - 1 }),
      (\cC^1, \cC^{ 2 })
    \right)
  \end{equation}
  satisfying the condition
  \begin{equation}
    \label{eq:C1 is contained in B1}
    \cC^1 \subset \cB^1,
  \end{equation}
  or, equivalently, the condition
  \begin{equation}
    \label{eq:C1 is contained in B1-equiv}
    \cC^1 \subset ( \cB^i )^{ \perp }
    \quad
    \textrm{for all } i = 2, \dots, \ell - 1.
  \end{equation}
  By the induction hypothesis, we know that the codomain of the morphism~$\vartheta$ is an algebraic space \'etale over~$U$. On the other hand, the base change (in the category of sheaves) of
  \(
    \vartheta
  \)
  by the morphism
  \(
    V \to \sod_f^{\ell-1}\times\sod_f
  \)
  of algebraic spaces corresponding to a pair of~$V$-linear semiorthogonal decompositions as in \eqref{eq:a pair of SOD} is represented by the complement in
  \(
    V
  \)
  of the support of the object
  \(
    \RRlHom _{ f _{ V }}
    \left(
      \bigoplus
      _{ 2 \le i \le \ell - 1 } b ^{ i },
      c ^{ 1 }
    \right)
    \in
    \Perf V
  \),
  where
  \(
    b ^{ i }
  \)
  and
  \(
    c ^{ 1 }
  \)
  are classical generators of
  \(
    \cB ^{ i }
  \)
  and
  \(
    \cC ^{ 1 }
  \),
  respectively, by~\Cref{lemma:semiorthogonality-via-local-Ext}. This implies that
  \(
    \vartheta
  \)
  is an open immersion of algebraic spaces, concluding the proof.
\end{proof}

The proof of \Cref{theorem:main-intro} is complete (the first part had already been settled in \Cref{proposition:description-sod-f}).

\subsection{First examples, questions and properties}
We now elaborate on the geometric properties of~$\sod_f^{\ell}\to U$
by exhibiting some examples of its interesting and potentially complicated behaviour.

\subsubsection{\texorpdfstring{$\sod_{f}^{\ell}$}{} \textbf{over fields}}
The first thing to observe is that, although~$\sod_f^{\ell}$ is shown to be an algebraic space over $U$ in an abstract way,
it turns out to be something tame over the generic points of $U$, or put differently, when $U$ is the spectrum of a field.

\begin{proposition}
  \label{proposition:sod-f-over-field}
  If $f\colon X \to U = \Spec \bfk$ is a smooth and proper morphism, where $\bfk$ is a field, then $ \sod_f^{\ell}$ is a separated $\bfk$-scheme of the form
  \begin{equation}
    \coprod_{ i \in I } \Spec L_i
  \end{equation}
  for a (possibly infinite) collection of finite separable field extensions $ L_i/\bfk $.
\end{proposition}

\begin{proof}
  The assertion that $\sod_f^{\ell}$ is a scheme under the assumption $U = \Spec \bfk$ follows from a general result about algebraic spaces \'etale over a field,
  cf.~\cite[\href{https://stacks.math.columbia.edu/tag/03KX}{Tag 03KX}]{stacks-project}.
  On the other hand, the restriction of $\sod_f^{\ell} \to \Spec \bfk$ to each connected component is affine, and this confirms that $\sod_f^{\ell} \to \Spec \bfk$ is separated.
  The realisation $\sod_{f}^{\ell} = \coprod_{ i \in I } \Spec L_i$ follows from
  \cite[\href{https://stacks.math.columbia.edu/tag/02GL}{Tag 02GL}]{stacks-project}.
\end{proof}

We also have the following result on the \emph{field of definition}
of a semiorthogonal decomposition.

\begin{proposition}
  \label{proposition:field-of-definition}
  Let $\bfk \subset K$ be an extension of fields,
  and $X$ be a smooth projective variety over $\bfk$.
  Let
  \begin{equation}
    \label{equation:sod-of-base-change}
    \derived^\bounded ( X_K )
    =
    \langle \cA^1_K, \dots, \cA^\ell_K \rangle
  \end{equation}
  be a $K$-linear semiorthogonal decomposition.
  Then there is a finite intermediate field extension~$\bfk \subset L \subset K$
  and a unique~$L$-linear semiorthogonal decomposition $\derived^\bounded (X_L)=\langle \cA^1_L, \dots, \cA^\ell_L \rangle$
  whose base change by $L \subset K$ coincides with \eqref{equation:sod-of-base-change}.
\end{proposition}

\begin{proof}
  By definition, the $K$-linear semiorthogonal decomposition as in
  \eqref{equation:sod-of-base-change} bijectively corresponds to
  a morphism of $\bfk$-algebraic spaces $s \colon \Spec K \to \sodl{\ell}_{X/\bfk}$.
  By \Cref{proposition:sod-f-over-field},
  $\sodl{\ell}_{X/\bfk}$ is a scheme,
  and there exists an intermediate field $L$ as in the assertion
  such that $s$ is the composition of
  the obvious morphism $\Spec L\to\Spec K$
  and an embedding as a connected component $\iota \colon \Spec L \hookrightarrow \sodl{\ell}_{X/\bfk}$.
  The $L$-linear semiorthogonal decomposition of $\derived^\bounded ( X_L ) $
  corresponding to $\iota$ has the desired properties.

  The uniqueness follows from the obvious equality $ \cA^i_L=\{F \in \derived^\bounded ( X_L ) \mid F_K \in \cA^i_K\}$,
  where $F_K$ denotes the pullback of $F$ along the morphism $X_K\to X_L$.
\end{proof}

\begin{example}
  \label{example:trivial-sods}
  The next thing to point out is that there are always $\ell$ sections $U \to \sod_f^{\ell}$ of the structure morphism $ \sod_f^{\ell} \to U$,
  corresponding to the obvious trivial semiorthogonal decompositions
  \begin{equation}
    \Perf  \cX  = \langle 0,\ldots,0,\Perf  \cX, 0, \ldots,0 \rangle,
  \end{equation}
  where the whole category $\Perf \cX$ sits in $i$th position for $i=1,\ldots,\ell$.
  The images of the sections are disjoint and give $\ell$ connected components $U \subset \sod_f$.
  In \Cref{subsection:nontrivial} we discuss how to remove these trivial components from the algebraic space~$\sod_f^{\ell}$ by defining a subfunctor which takes into account only nontrivial semiorthogonal decompositions.
\end{example}

\begin{example}
  Let~$X$ be a smooth projective variety over an algebraically closed field~$\bfk$,
  whose derived category admits no nontrivial semiorthogonal decompositions.
  Typical examples are Calabi--Yau varieties (by Serre duality),
  or curves of genus~$g\geq 2$ \cite{MR2838062}.
  Writing the structure morphism as~$f\colon X\to\Spec \bfk$,
  this non-existence can be rephrased as the identification
  \begin{equation}
    \sod_f^{\ell}\cong \coprod_{i=1}^{\ell}\Spec \bfk,
  \end{equation}
  where the right-hand side are the trivial components we just discussed.
  More generally, for a family~$f\colon \mathcal{X}\to U$ of such varieties, we have
  \begin{equation}
    \sod_f^{\ell}\cong \coprod_{i=1}^{\ell} U.
  \end{equation}
\end{example}

Needless to say, things become more interesting when $X$ admits nontrivial semiorthogonal decompositions. Below is the first nontrivial example.
\begin{example}
  Let~$f\colon\mathbb{P}_\bfk^1\to\Spec \bfk$.
  It is standard that all nontrivial semiorthogonal decompositions of~$\derived^\bounded(\mathbb{P}_\bfk^1)$
  are given as exceptional collections
  of the form~$\langle\mathcal{O}_{\mathbb{P}_\bfk^1}(i),\mathcal{O}_{\mathbb{P}_\bfk^1}(i+1)\rangle$ for~$i\in\mathbb{Z}$
  (up to shift, which is invisible when considering the subcategory generated by the exceptional object).

  This implies that
  \begin{equation}
    \sod_f^{2}\cong
    \overbrace{\left( \Spec \bfk \sqcup \Spec \bfk\right)}^{\text{trivial components}}
    \sqcup
    \coprod_{i \in \mathbb Z} \Spec\bfk.
  \end{equation}
  Thus we see that already in the simplest of interesting cases there is no chance of the space~$\sod_f$ being quasicompact.

  In \Cref{subsection:group-actions} we suggest how~$\sod_f^{\ell}$ can be equipped with an action of the braid group,
  encoding mutation of semiorthogonal decompositions.
  By transitivity of the braid group action on the set of exceptional collections of~$\derived^\bounded(\mathbb{P}_\bfk^1)$
  this will imply that the quotient reduces to a single point (after we remove the trivial components).
\end{example}

As the following example\footnote{This example is developed in great detail in \cite{2107.03051}.} shows,
the structure morphism $\sod_f \to U$ is in general not separated (nor of finite type), let alone quasifinite, even when restricted to connected components of~$\sod_f$.

\begin{example}\label{eg:OU}
  Let~$\bfk$ be an algebraically closed field.
  Set $U = \bA^1$. There exists a smooth projective morphism $ f \colon \cX \to  U $ of schemes such that
  \begin{enumerate}
    \item $ f^{ - 1 } ( 0 ) \cong \Sigma_2 $ and
    \item $ f^{ - 1 } ( \bG_{ \mathrm{m} } ) \cong \Sigma_0 \times \bG_{ \mathrm{m} } $
      over $ \bG_{ \mathrm{m} } = \bA^1 \setminus \{ 0 \} \subseteq \bA^1$,
  \end{enumerate}
  where $ \Sigma_d\defeq\bP_{ \bP^1 }\left( \cO_{ \bP^1 } \oplus \cO_{ \bP^1 } ( - d ) \right) $
  is the Hirzebruch surface of degree $ d $.
  By~\cite[Proposition~2.10]{MR1286839} an exceptional object on the surface $ \Sigma_0 $
  is an exceptional vector bundle up to shifts,
  and it is uniquely determined by its class in $ \mathrm{K}_0 ( \Sigma_0 ) $ up to even shifts.

  On the other hand, by~\cite[Claim~3.1]{MR3431636},
  there are infinitely many exceptional sheaves on $ \Sigma_2 $,
  all of which have the same class in $ \mathrm{K}_0 ( \Sigma_2 )$.
  By the deformation theory of exceptional objects
  they extend uniquely to $ f $-exceptional objects on $ \cX $; in general this is the case only after passing to an \'etale neighbourhood of $ 0 \in U $,
  but in this example it is not necessary since
  \(
    \bA ^{ 1 }
  \)
  is simply connected and the morphism $ f $ is trivial over $ \bG_{ \mathrm{m} } $.

  The restrictions of these $f$-exceptional objects
  on a non-central fiber
  (which are all isomorphic to~$\bP^1\times\bP^1$)
  are all isomorphic to a single exceptional vector bundle.
  In this way one can find $ f $-linear semiorthogonal decompositions yielding
  an irreducible component of $ \sod_f $ which is neither separated nor quasicompact over $ U $,
  because it is the affine line with infinitely many origins.

  In particular, this means that $\sod_f $ is not necessarily universally closed over $ U $,
  by \cite[\href{https://stacks.math.columbia.edu/tag/04XU}{Tag 04XU}]{stacks-project}.
\end{example}

As we already pointed out in the introduction, this example makes it difficult to answer the question whether $ \sod_f^{\ell}$ is always a scheme or not.

However, we observe the following,
showing that not all good properties are lost.
\begin{proposition}
  \label{proposition:quasi-separated}
  Let $ f \colon \cX \to U $ be a morphism of schemes as in \Cref{situation:main},
  and assume moreover that~$U$ is excellent.
  Then~$\sod_f^{\ell}$ is a locally noetherian and quasiseparated algebraic space,
  and has a dense open schematic locus.
\end{proposition}

\begin{proof}
  By \cref{theorem:sod-f} the structure morphism~$\sod_f^{\ell}\to U$ is \'etale and thus locally of finite presentation, hence a fortiori locally of finite type.
  On the other hand, because~$U$ is assumed excellent,
  it is locally noetherian.
  Therefore applying \cite[\href{https://stacks.math.columbia.edu/tag/04ZK}{Tag 04ZK}]{stacks-project}
  to $\sod_f^{\ell}\to U$
  shows that~$\sod_f^{\ell}$ is also locally noetherian.

  On the other hand, $\sod_{f}^{\ell}$ is a \emph{decent} algebraic space (in the sense of \cite[\href{https://stacks.math.columbia.edu/tag/03I7}{Tag 03I7}]{stacks-project}) because $U$ is (as any scheme) decent and decency is preserved by \'etale maps by \cite[\href{https://stacks.math.columbia.edu/tag/0ABU}{Tag 0ABU}]{stacks-project}. Therefore $\sod_{f}^{\ell}$ is quasiseparated by \cite[\href{https://stacks.math.columbia.edu/tag/0BB6}{Tag 0BB6}]{stacks-project}, and in particular, by \cite[\href{https://stacks.math.columbia.edu/tag/06NH}{Tag 06NH}]{stacks-project}, it has a dense open schematic locus.
\end{proof}

\subsubsection{\textbf{Existence part of valuative criterion for properness}}
We can also ask the following question, motivated by the fact that
a morphism is universally closed if and only if it is quasicompact and satisfies the existence part of the valuative criterion (\cite[\href{https://stacks.math.columbia.edu/tag/01KF}{Tag 01KF} and \href{https://stacks.math.columbia.edu/tag/04XU}{Tag 04XU}]{stacks-project}).
If answered affirmatively, then it would imply that for every connected component $ Z \subset \sod_f^{\ell}$,
the restriction of the structure morphism $ Z \to U $ is surjective.

This question has recently received a negative answer by work of Wu \cite{Weimufei},
but we kept it in as in the first version of this paper.

\begin{question}
  \label{question:valuative}
  Let $f \colon \cX \to U$ be a morphism of schemes as in \Cref{situation:main}, where $U$ is excellent and purely of characteristic~$0$. Let $\ell \geq 2$ be an integer.
  Does $ \sod_f^{\ell} \to U $ satisfy the existence part of the valuative criterion
  (cf.~\cite[\href{https://stacks.math.columbia.edu/tag/01KD}{Tag 01KD}]{stacks-project})?
  Namely, let $R$ be a valuation ring and $Q$ its field of fractions.
  For every commutative diagram
  \begin{equation}
    \begin{tikzcd}[row sep=large,column sep=large]
      \Spec Q \arrow{r} \arrow{d} & \sod_f^{\ell} \arrow{d}\\
      \Spec R \arrow{r} \arrow[dotted]{ru} & U,
    \end{tikzcd}
  \end{equation}
  is there at least one dotted arrow which makes the two induced triangles commute?
\end{question}
This question is related to the conjectural existence of a specialisation morphism
for Grothendieck rings of categories \`a la Bondal--Larsen--Lunts \cite{MR2051435},
as discussed in \cite{1810.03220v1}.

The next example, due to Fakhruddin,
shows that there are counterexamples to \Cref{question:valuative}
in characteristic $2$ and in mixed characteristic $(0,2)$.
So far we do not know of any counterexample in characteristic $ p $ or $ ( 0, p ) $ for $ p > 2$,
but it seems likely that one can obtain counterexamples in those cases too.
In fact, the deformation invariance of genera is known to fail in those characteristics as well
(see for example \cite[Examples 8.7 and 8.8]{MR799664}).
Unfortunately for the general fibres of these families no line bundle is an exceptional object,
so the argument we will outline does not work in these cases.

\begin{example}
  \label{remark:fakhruddin}
  Let $\bfk$ be an algebraically closed field of characteristic $2$.
  Let $f \colon \cX \to U = \Spec R$,
  where $ R = \bfk \llbracket t \rrbracket$,
  be a smooth projective family of Enriques surfaces in characteristic $2$,
  where the generic fibre is ordinary and the special fibre is supersingular (see \cite{MR3393362}).
  The structure sheaf of the generic fibre is an exceptional object,
  so we get a section $ \Spec Q \to \sod_f $ over the generic point.
  On the other hand, there is no semiorthogonal decomposition on the central fibre,
  since its canonical bundle is trivial.
  In particular, the section cannot be extended over $ U $.

  By the lifting result \cite[Theorem 4.10]{MR3393362}, for every (super)singular Enriques surface $X$
  over~$\bfk$,
  one can find a smooth projective morphism $ f \colon \cX \to U = \Spec S$
  where $S$ is a complete discrete valuation ring of mixed characteristic,
  whose residue field is $\bfk$,
  the special fibre is isomorphic to $X$, and the generic fibre is an Enriques surface.
  By the same arguments as above, this example shows that \Cref{question:valuative} also fails in mixed characteristics.
\end{example}

\subsubsection{\textbf{Topological semiorthogonal decompositions}}
One could also consider semiorthogonal decompositions of $\derived(\Qcoh X)$
which are of topological nature.
We see that the uniqueness of deformations of semiorthogonal decompositions of such categories is satisfied
if the parameter space is the spectrum of an artinian local ring, but completely fails over a general parameter space.
\begin{example}
  \label{example:topological-sod}
  Let $ X $ be a variety of dimension at least~1 over a field $ \bfk $,
  with a (strict) closed subset $ Z \subset X $,
  and let $ i \colon V = X \setminus Z \hookrightarrow X $
  be the open immersion of the complement of $ Z$.
  Then there is the semiorthogonal decomposition
  \begin{equation}\label{equation:wrong-way-recollement-as-sod}
    \derived(\Qcoh X)=
    \langle
    \derived(\Qcoh V),
    \derived_Z(\Qcoh X)
    \rangle,
  \end{equation}
  where $\derived_Z(\Qcoh X) $ is the subcategory of $\derived(\Qcoh X) $ consisting of those complexes whose cohomology is supported on $Z$
  (see for instance \cite[\S 6.2.1]{MR2434186}).
  One can recover the closed subset~$ Z $ from the subcategory~$\derived_Z(\Qcoh X) \into \derived(\Qcoh X)$
  as the subset of those closed points whose structure sheaves are contained in the subcategory.

  However,
  one can easily work over a base
  where there are infinitely many distinct deformations of
  the semiorthogonal decomposition~\eqref{equation:wrong-way-recollement-as-sod}.
  For example, let $ X = \bA^1 = \Spec \mathbb{C} [ t ] $, $ Z = \{ 0 \} $,
  and $ U = \Spec \mathbb{C} \llbracket t \rrbracket $ the neighbourhood of the origin of $ X$.
  Take a closed irreducible curve $ C \subset X \times X $ satisfying
  \begin{equation}
    C\cap\left(X \times \{ 0 \}\right)= \{ ( 0, 0 )\},
  \end{equation}
  and let $ \mathcal{Z} \hookrightarrow X \times U $ be its inverse image under the obvious map $ \cX \defeq X \times U \to X \times X$.
  Then the $ U $-linear semiorthogonal decomposition
  \begin{equation}
    \derived ( \Qcoh \cX ) = \langle \derived ( \Qcoh \cV ),\derived_{ \cZ } ( \Qcoh \cX ) \rangle,
  \end{equation}
  where $ \cV = \cX \setminus \cZ$,
  deforms the semiorthogonal decomposition~\eqref{equation:wrong-way-recollement-as-sod} of the central fibre.
  Now note that we have infinitely many choices for the curve $ C$,
  and hence for the subcategories.
  Note also that their base changes along $ \Spec \bC \llbracket t \rrbracket/ ( t^{ n + 1 } )\hookrightarrow U$
  are all the same for all $ n \ge 0 $.

  One possible way of interpreting this phenomenon
  is in terms of the non-uniqueness of the algebraization of the projection kernel.
  Note that this interpretation assumes some version of \cref{section:sod-is-dec} in this setting,
  and a generalization of \cite{former-appendix} for the decomposition triangle in this context.
  Provided that this works,
  the map similar to $ \vartheta $ in \Cref{theorem:deforming-sods},
  which in turn is concerned with \cref{item:DEF} in \Cref{criterion:hall-rydh}, is not injective.
  This is outside the scope of the current paper.
\end{example}

\section{Amplifications}
\label{section:amplifications}
In this section we will bootstrap from the construction of~$\sod_f^{\ell}$ to define closely related moduli spaces.
These amplifications, as anticipated in \Cref{subsection:amplifications-intro}, happen in two steps,
the latter being part of a conjecture:
\begin{enumerate}
  \item we explain how to work relatively to a fixed subcategory~$\mathcal{B} \subseteq \Perf \mathcal{X}$ in a~$U$-linear semiorthogonal decomposition $\Perf \mathcal{X} = \langle\mathcal{A},\mathcal{B}\rangle$. The output (cf.~\Cref{corollary:subcategory-space}) is an \'etale algebraic space $\sod_{\cB}^{\ell} \to U$ parametrising semiorthogonal decompositions of $\cB$, thereby making the theory valid for geometric subcategories, in the spirit of \cite{MR3545926};
  \item we conjecture that a naturally defined subsheaf $\ntsodl{\ell}_{ \cB } \subseteq \sodl{\ell}_{\cB}$, parematrising only nontrivial semiorthogonal decompositions (i.e.,~the zero subcategory is not allowed as a component of the decomposition, cf.~\Cref{def:ntSOD affine}), defines an open and closed subfunctor.
\end{enumerate}
Moreover, in \Cref{subsection:group-actions}, we will define natural group actions (by mutations and autoequivalences) defined on our moduli spaces.

\subsection{Relative to a fixed subcategory}
\label{subsection:relative-to-subcategory}
We will now bootstrap to consider not just semiorthogonal decompositions of perfect complexes on schemes,
but to all categories of geometric nature.
This is to be interpreted in the sense that they are components of a semiorthogonal decomposition of perfect complexes on a scheme,
similar to the notion of a geometric dg category in \cite{MR3545926}.
Note that semiorthogonal decompositions are by definition independent of dg enhancements,
as they are defined for the homotopy category of perfect complexes over a dg category,
thus we need not worry about enhancements of geometric dg categories.

\begin{definition}
  Let $ f \colon \cX \to U $ be a morphism of schemes as in \Cref{situation:main}.
  Fix $ \ell \ge 2 $ and a $ U $-linear semiorthogonal decomposition
  \begin{equation}
    \label{eq:fixed_sod}
    \Perf \cX
    =
    \langle
    \cA, \cB
    \rangle,
  \end{equation}
  and consider the presheaf $ \overline{\sod}^{\ell}_{ \cB } \in \Presheaf ( \Aff_U ) $
  which sends an affine $U$-scheme $ \phi \colon V \to U $ to the set
  \begin{equation}
    \overline{\sod}^{\ell}_{ \cB }(\phi) =
    \Set{
      (\cA^1, \dots, \cA^\ell)
      |
      \begin{array}{c}
        \cB_\phi=\langle \cA^1, \dots, \cA^\ell\rangle \textrm{ is a }V\textrm{-linear}  \\
        \textrm{semiorthogonal decomposition}
      \end{array}
    },
  \end{equation}
  where~$\cB_\phi$ is the base change of~$\mathcal{B}$ (cf.~\Cref{proposition:base-change-sod}).
\end{definition}

Note that the semiorthogonal decomposition \eqref{eq:fixed_sod} we fixed corresponds to a section $ s \colon U \to \sod_f
$.
This choice allows us to prove the following proposition.

\begin{proposition}
  \label{proposition:eq:sod of an admissible subcategory as fibre product}
  Let $ f \colon \cX \to U $ be a morphism of schemes as in \Cref{situation:main}. Fix $ \ell \ge 2 $ and a $ U $-linear semiorthogonal decomposition as in \eqref{eq:fixed_sod}.
  Let
  \begin{equation}
    \label{morphism-presh}
    \begin{tikzcd}
      \overline{\sod}_f^{\ell+1}\arrow{r} & \overline{\sod}_f
    \end{tikzcd}
  \end{equation}
  be the morphism of presheaves defined by sending, for an affine $U$-scheme $V$, a~$V$-linear semiorthogonal decomposition
  $(\cA^1,\ldots,\cA^{\ell+1})$ to the~$V$-linear semiorthogonal decomposition
  $(\cA^1,\langle\cA^2,\ldots,\cA^{\ell+1}\rangle)$.
  There exists a natural isomorphism of presheaves
  \begin{equation}\label{eq:sod of an admissible subcategory as fibre product}
    \begin{tikzcd}
      \overline{\sod}^{\ell}_{ \cB }
      \arrow{r}{\sim} &
      \overline{\sod}^{\ell + 1 }_f \times_{\overline{\sod}_f} U
    \end{tikzcd}
  \end{equation}
  on $ \Aff_U$.
\end{proposition}

\begin{proof}
  This follows from the observation that for each $ U $-scheme $ \phi \colon V \to U $,
  a semiorthogonal decomposition of $ \cB_\phi $ of length $ \ell $ is nothing but
  a semiorthogonal decomposition of $ \Perf \cX_V $ of length $ \ell + 1 $ whose first component is $ \cA_\phi $.
\end{proof}

The description of $\overline{\sod}^{\ell}_{ \cB }$ in~\eqref{eq:sod of an admissible subcategory as fibre product} as a fibre product implies that it also is a sheaf on the big affine \'etale site $ \Affet{U}$, as the right-hand side is a fibre product of presheaves all of which are known to be sheaves. In particular, it makes sense to give the following definition.

\begin{definition}
  \label{def:sodl}
  Let $ f \colon \cX \to U $ be a morphism of schemes as in \Cref{situation:main}.
  We let $ \sodl\ell_{ \cB } $ denote the sheaf on the big \'etale site $ \Schet{ U } $
  corresponding to the sheaf $\overline{\sod}^{\ell}_{ \cB }$ of \Cref{proposition:eq:sod of an admissible subcategory as fibre product}
  under the equivalence of toposes of \Cref{lemma:equivalence-of-toposes}.
\end{definition}

\begin{corollary}
  \label{corollary:subcategory-space}
  Let $ f \colon \cX \to U $ be a morphism of schemes as in \Cref{situation:main}. Further assume that $U$ is excellent.
  The sheaf $\sodl{\ell}_{ \cB }$, as defined in \Cref{def:sodl},
  is a nonempty \'etale algebraic space over~$ U $.
\end{corollary}

\begin{proof}
  Under the equivalence of toposes of \Cref{lemma:equivalence-of-toposes},
  the sheaf $\sodl{\ell}_{ \cB }$ still satisfies the isomorphism of \Cref{eq:sod of an admissible subcategory as fibre product}.
  Now \Cref{theorem:sod-f} immediately implies that $ \sodl{\ell}_{ \cB } $
  is an \'etale algebraic space over $ U $,
  as the right-hand side of \Cref{eq:sod of an admissible subcategory as fibre product}
  is a fibre prouct of \'etale algebraic spaces over $ U $.
\end{proof}

\subsection{Restricting to nontrivial semiorthogonal decompositions}
\label{subsection:nontrivial}
We next introduce the subfunctor of $\sod_{\cB}^\ell$
which only parametrises \emph{nontrivial} semiorthogonal decompositions.
We expect it to be an open subspace (see \Cref{conjecture:rouquier-for-local-noetherian}).

Keeping the setup of the previous subsection, consider a $U$-linear semiorthogonal decomposition
\begin{equation}
  \Perf \cX
  =
  \langle
  \cA,\cB
  \rangle.
\end{equation}
If~$\cB$ is taken to be all of $\Perf \cX$ one can modify the notation in the next definition accordingly to reflect the dependence on the morphism $f\colon\cX\to U$ only.
Recall that a geometric point of a scheme $V$ is a morphism of the form $ x \colon \Spec \Omega \to V$,
where $ \Omega $ is an algebraically closed field.

\begin{definition}\label{def:ntSOD affine}
  Let $ f \colon \cX \to U $ be a morphism of schemes as in \Cref{situation:main}. The presheaf of nontrivial semiorthogonal decompositions of $\mathcal{B}$ on $ \Aff_U$,
  denoted $\ntsodl{\ell}_{ \cB } \subseteq \sodl{\ell}_{ \cB }$, is defined by sending an affine $U$-scheme $ \phi \colon V \to U $ to the set
  \begin{equation}\label{eq:ntsod}
    \ntsodl{\ell}_{ \cB } ( V\xrightarrow{\phi} U )
    \defeq
    \Set{
      (
        \cA_\phi^1, \dots, \cA_\phi^\ell
      ) \in \sodl{\ell}_{ \cB } ( \phi )
      |
      \begin{array}{c}
        \cA_\phi^i \big\vert_x \neq 0 \textrm{ for all } i \textrm{ and for every}\\
        \textrm{geometric point } x\textrm{ of } V
      \end{array}
    },
  \end{equation}
  where~$\cA_\phi^i \big\vert_x$ denotes the base change to the geometric point $x \colon \Spec \Omega \to V$.
\end{definition}

\begin{lemma}\label{lm:ntsodl is a sheaf on the big affine etale site.}
  The presheaf $ \ntsodl{\ell}_{ \cB } $ defined in \Cref{def:ntSOD affine} is a sheaf on the big affine \'etale site $ \Affet { U }$.
\end{lemma}

\begin{proof}
  As $ \ntsodl{\ell}_{ \cB } $ is a subpresheaf of the sheaf $ \sodl{\ell}_{ \cB }$,
  the first sheaf condition (the locality, or separatedness) is automatically satisfied.
  The second sheaf condition (the gluing) is also satisfied,
  as the condition in \eqref{eq:ntsod} is easily seen to be \'etale local on the base.
\end{proof}

Similarly to the proof of \Cref{proposition:description-sod-f},
and taking into account that the defining condition of \eqref{eq:ntsod}
is \'etale local on the base, we have the following description.

\begin{proposition}
  Let $ f \colon \cX \to U $ be a morphism of schemes as in \Cref{situation:main}.
  For each quasicompact and semiseparated $U$-scheme $ \phi \colon V \to U$,
  there is a natural bijection
  \begin{equation}
    \ntsodl{\ell}_{ \cB } ( \phi )
    \cong
    \Set{
      (
        \cA_\phi^1, \dots, \cA_\phi^\ell
      ) \in \sodl{\ell}_{ \cB } ( \phi )
      |
      \begin{array}{c}
        \cA_\phi^i \big\vert_x \neq 0 \textrm{ for all } i \textrm{ and for every}\\
        \textrm{geometric point } x\textrm{ of } V
      \end{array}
    }.
  \end{equation}
\end{proposition}

\begin{definition}\label{def:ntsod}
  Let $ f \colon \cX \to U $ be a morphism of schemes as in \Cref{situation:main}.
  Let $ \ntsodl\ell_{ \cB } $ be the sheaf on the big \'etale site $ \Schet{ U } $ which corresponds, via \Cref{lemma:equivalence-of-toposes}, to the sheaf on the big affine \'etale site denoted by the same symbol in \Cref{lm:ntsodl is a sheaf on the big affine etale site.}.
\end{definition}

We expect the following to hold,
parallel to \cref{theorem:main-intro}\ref{enumerate:main-intro-etale}.

\begin{conjecture}
  \label{conjecture:ntsod-is-etale}
  Let $ f \colon \cX \to U $ be a morphism of schemes as in \Cref{situation:main}. Further assume that $U$ is excellent.
  The sheaf $ \ntsodl\ell_{ \cB } $ is an \'etale algebraic space over $ U$.
  Moreover, the natural morphism
  \begin{equation}
    \label{equation:inclusion-ntsod-sod}
    \begin{tikzcd}
      \ntsodl\ell_{ \cB } \arrow[hook]{r} & \sodl\ell_{ \cB }
    \end{tikzcd}
  \end{equation}
  is an open and closed immersion.
\end{conjecture}

To the best of the authors' knowledge, \Cref{conjecture:ntsod-is-etale} is still open.
We now describe two equivalent characterizations,
which phrase it in terms of a vanishing criterion for admissible subcategories
(see \cref{conjecture:rouquier-for-local-noetherian})
and a condition on the trivial sections described in \cref{example:trivial-sods}
(see \cref{conjecture:referee}).
In \cref{proposition:conjectures-equivalent} we will prove these two conditions are equivalent,
and in \cref{proposition:conjecture-implies-conjecture-and-vice-versa}
we prove that they are equivalent to \cref{conjecture:ntsod-is-etale}.

\begin{conjecture}
  \label{conjecture:rouquier-for-local-noetherian}
  Let $(A,\mathfrak{m})$ be an integral local noetherian ring, and $U=\Spec A$.
  Suppose that $f\colon\cX\to U$ is a smooth proper morphism of schemes,
  and let $\Perf\cX = \langle\cA,\cB\rangle $ be
  a $U$-linear semiorthogonal decomposition.
  If $\cA \otimes_A Q = 0$,
  where $Q$ is the field of fractions of $A$,
  then $\cA = 0$.
\end{conjecture}

Instead of \cref{conjecture:rouquier-for-local-noetherian}
one can also ask the following question.
\begin{question}
  \label{conjecture:dg-version}
  Let~$R$ be a complete discrete valuation ring,
  with~$Q$ its fraction field.
  Let~$A$ be a smooth and proper dg algebra over~$R$.
  Does~$A\otimes_RQ$ being quasi-isomorphic to zero
  imply that~$A$ is quasi-isomorphic to zero?
\end{question}
A positive answer would imply the previous conjecture.
However,
because it is not known whether every smooth and proper dg category
is geometric,
i.e.,
an admissible subcategory in the derived category of a smooth and projective variety,
this question is a priori stronger than the geometric conjectures stated above.

Observe that
\Cref{conjecture:rouquier-for-local-noetherian} is in general false
if we drop the smoothness assumption.
\begin{remark}
  \label{example:cannot-omit-smoothness}
  Let $C$ be a smooth projective curve over $\mathbb{C}$,
  let $X$ be the blowup of $C\times_{\mathbb{C}}\Spec \mathbb{C} \llbracket t \rrbracket $ in a closed point,
  and $E\hookrightarrow X$ the exceptional curve.
  Then the semiorthogonal decomposition
  \begin{equation}
    \derived^\bounded (X)
    =
    \langle
    \cO_E ( - 1 ),
    {}^{ \perp } \cO_E ( - 1 )
    \rangle
  \end{equation}
  is nontrivial on the central fibre but it is trivial on the generic fibre.
  For more on this phenomenon in the context of categorical absorptions,
  see \cite{MR4698898}.
\end{remark}

Before stating the next conjecture,
we discuss geometric properties of sections of the structure morphism.
\begin{lemma}
  \label{lemma:sections-are-open}
  A section of the structure morphism
  \(
    \sodl{\ell} _{ \cB } \to U
  \)
  is an open immersion.
\end{lemma}

\begin{proof}
  A section is necessarily a monomorphism,
  and it is representable by \cite[\href{https://stacks.math.columbia.edu/tag/03HK}{Tag 03HK}]{stacks-project}
  as
  \(
    U
  \)
  is a scheme.
  \Cref{theorem:sod-f}
  and the cancellation property for \'etale morphisms from \cite[\href{https://stacks.math.columbia.edu/tag/03FV}{Tag 03FV}]{stacks-project}
  implies that it also is \'etale,
  hence it is an open immersion by \cite[\href{https://stacks.math.columbia.edu/tag/025G}{Tag 025G}]{stacks-project}.
\end{proof}

The following conjecture, together with its equivalence to \Cref{conjecture:rouquier-for-local-noetherian}, was suggested by a referee.
\begin{conjecture}
  \label{conjecture:referee}
  The sections of the projection
  \(
    \sodl{\ell} _{ \cB } \to U
  \)
  representing the
  \(
    \ell
  \)
  trivial semiorthogonal decompositions discussed in \Cref{example:trivial-sods} are closed immersions.
\end{conjecture}

\begin{proposition}
  \label{proposition:conjectures-equivalent}
  \Cref{conjecture:rouquier-for-local-noetherian} is equivalent to \Cref{conjecture:referee}.
  Moreover, if we assume these conjectures, then for every connected noetherian affine scheme
  \(
    \phi \colon V \to U
  \)
  we have that
  \begin{equation}
    \ntsodl{\ell} _{ \cB } ( \phi )
    =
    \left\{
      ( \cA ^{ 1 } _{ \phi }, \dots, \cA ^{ \ell } _{ \phi } )
      \mid
      \cA ^{ i } _{ \phi }\big\vert _{ x } \neq 0
      \textrm{ for all } i \textrm{ and for some }
      \textrm{geometric point } x\textrm{ of } V
    \right\}.
  \end{equation}
\end{proposition}

\begin{proof}
  We explain this for the case
  \(
    \ell = 2
  \)
  and
  \(
    \cB
    =
    \Perf \cX
  \), the general case being similar.

  We will first assume \cref{conjecture:referee}.
  Let
  \(
    U = \Spec A
  \)
  be as in \Cref{conjecture:rouquier-for-local-noetherian}
  and let
  \(
    s
    \colon
    U
    \to
    \sod _{ f }
  \)
  be a section corresponding to a semiorthogonal decomposition
  \(
    \Perf \cX
    =
    \langle
    \cA,
    \cB
    \rangle
  \)
  for which
  \(
    \cA \otimes _{ A } Q
    =
    0
  \).
  Let
  \(
    s _{ 1 } \colon U \to \sod _{ f }
  \)
  be the trivial section corresponding to the trivial semiorthogonal decomposition
  \(
    \Perf \cX
    =
    \langle
    0,
    \Perf \cX
    \rangle
  \),
  so that
  \(
    s \vert _{ \Spec Q }
    =
    s _{ 1 } \vert _{ \Spec Q }
  \) using the condition $\cA \otimes_{A}Q=0$.
  Let us consider (without changing the notation) the topological spaces associated to the algebraic spaces
  \cite[\href{https://stacks.math.columbia.edu/tag/03BU}{Tag 03BU}]{stacks-project}.
  Then for any point
  \(
    u \in \Spec A
  \)
  we have
  \begin{equation}
    s ( u )
    \in
    s ( \Spec A )
    =
    s ( \overline{\Spec Q} )
    \subseteq
    \overline{s( \Spec Q )}
    =
    \overline{s _{ 1 }( \Spec Q )}
    \subseteq
    \overline{s _{ 1 }( \Spec A )}
    =
    s _{ 1 }( \Spec A ),
  \end{equation}
  where the last equality follows from the assumption that
  \(
    s _{ 1 }
  \)
  is a closed immersion
  by \cref{conjecture:referee}.
  Hence
  \begin{equation}
    s ( u )
    \in
    s _{ 1 } ( \Spec A )
    \cap
    \sod _{ f, u }
    =
    \left\{
      s _{ 1 } ( u )
    \right\},
  \end{equation}
  so that
  \(
    s ( u ) = s _{ 1 } ( u )
  \). This shows that the category
  \(
    \cA
  \)
  is trivial over every point
  \( u \in \Spec A
  \), hence \Cref{conjecture:rouquier-for-local-noetherian} holds.

  Conversely, assume \Cref{conjecture:rouquier-for-local-noetherian}. Let
  \(
    s _{ 1 } \colon U \to \sod _{ f }
  \)
  be a trivial section as in the previous paragraph.
  Since it is an open immersion of locally noetherian algebraic spaces
  by \cref{lemma:sections-are-open},
  it is quasicompact by
  \cite[\href{https://stacks.math.columbia.edu/tag/0CPM}{Tag 0CPM}]{stacks-project}.
  Hence,
  by ~\cite[\href{https://stacks.math.columbia.edu/tag/03KA}{Tag 03KA}]{stacks-project}
  it is enough to check the existence part of the valuative criterion for properness
  \cite[\href{https://stacks.math.columbia.edu/tag/01KD}{Tag 01KD}]{stacks-project}
  for discrete\footnote{Strictly speaking, the existence part of the valuative criterion for properness, in the form of \cite[\href{https://stacks.math.columbia.edu/tag/01KD}{Tag 01KD}]{stacks-project}, requires the use of general valuation rings; however, since $U$ and $\sod_{f}$ are locally noetherian (cf.~\Cref{proposition:quasi-separated}), the formulation is equivalent to the one involving only \emph{discrete} valuation rings.} valuation rings.
  Let
  \(
    A
  \)
  be a discrete valuation ring and consider the following commutative diagram of solid arrows.
  \begin{equation}
    \label{existence-part-diagram}
    \begin{tikzcd}[row sep=large,column sep=large]
      \Spec Q \arrow{r} \arrow{d}
      &
      U
      \arrow[
        d, "s _{ 1 }"
      ] \\
      \Spec A
      \arrow
      [
        "s"'
      ]{r}
      \arrow
      [
        dotted
      ]
      {ur}
      &
      \sod _{ f }
    \end{tikzcd}
  \end{equation}
  It suffices to show that there exists a unique dashed arrow making the diagram commute. However, the commutativity of \eqref{existence-part-diagram} implies that
  \(
    s \vert _{ Q }
  \)
  represents a trivial semiorthogonal decomposition. Hence by \Cref{conjecture:rouquier-for-local-noetherian} the semiorthogonal decomposition represented by
  \( s
  \) over
  \( \Spec A
  \)
  is also trivial, meaning that it factors through
  \(
    s _{ 1 }
  \).

  For the assertion in the moreover part, it is enough to show that for any pair of points
  \(
    v _{ 1 }, v _{ 2 } \in V
  \)
  we have that
  \begin{equation}
    \cA ^{ i } _{ \phi }\big\vert _{ v _{ 1 } } = 0
    \iff
    \cA ^{ i } _{ \phi }\big\vert _{ v _{ 2 } } = 0.
  \end{equation}
  Since
  \(
    V
  \)
  is noetherian and connected,
  any pair of points is connected by a finite chain of specializations and generalizations. Hence we may and will assume that
  \(
    v _{ 2 }
    \in
    \overline{\left\{
        v _{ 1 }
    \right\}}
  \).
  Put
  \(
    V = \Spec A '
  \)
  and
  \(
    \mathfrak{p}_{ 1 } \subseteq \mathfrak{p}_{ 2 } \triangleleft A '
  \)
  be the prime ideals corresponding to
  \(
    v _{ 1 }
  \)
  and
  \(
    v _{ 2 }
  \), respectively. Now put
  \(
    A =
    \left(
      A ' / \mathfrak{p}_{ 1 }
    \right)
    _{
      \mathfrak{p} _{ 2 }
    }
  \),
  take the base change by the natural morphism
  \(
    \Spec A \to V
  \),
  and then apply \Cref{conjecture:rouquier-for-local-noetherian} to the integral local noetherian ring
  \( A
  \).
\end{proof}

Finally, we relate the two equivalent conjectures from \cref{proposition:conjectures-equivalent}
to \cref{conjecture:ntsod-is-etale}.
\begin{proposition}
  \label{proposition:conjecture-implies-conjecture-and-vice-versa}
  \cref{conjecture:rouquier-for-local-noetherian}
  and \Cref{conjecture:ntsod-is-etale} are equivalent.
\end{proposition}

\begin{proof}
  We first show that \cref{conjecture:rouquier-for-local-noetherian}
  implies \cref{conjecture:ntsod-is-etale}.
  We know by \Cref{lm:ntsodl is a sheaf on the big affine etale site.} that~$\ntsodl{\ell}_\cB$ is a sheaf.
  We now show that
  \eqref{equation:inclusion-ntsod-sod}
  is an open and closed immersion of algebraic spaces over $ U $.
  For this take a noetherian affine $ U $-scheme $ \phi \colon V \to U $ and a section
  \(
    s \colon V \to \sodl{ \ell } _{ \cB }
  \).
  If we assume \Cref{conjecture:rouquier-for-local-noetherian},
  then \Cref{proposition:conjectures-equivalent} implies that the base change of
  \eqref{equation:inclusion-ntsod-sod} by the section
  \(
    s
  \)
  (in the category of sheaves)
  is represented by the disjoint union of connected components of
  \(
    V
  \)
  over which the semiorthogonal decomposition
  corresponding to~$s$ is nontrivial.
  This is an open and closed subscheme of
  \(
    V
  \),
  and this concludes of one direction.

  Finally,
  observe that \cref{conjecture:ntsod-is-etale} implies that the locus of trivial semiorthogonal decompositions
  is the closed complement.
  Hence, if it contains the generic point,
  it must be everything,
  which is the conclusion in \cref{conjecture:rouquier-for-local-noetherian}.
\end{proof}
Summing up, the conjectures discussed above are related by the chain of implications
\begin{equation}
  \begin{tikzcd}[row sep=large]
    \mbox{Answer to \Cref{conjecture:dg-version} is yes}
    \arrow[Rightarrow]{d} \\
    \mbox{\Cref{conjecture:rouquier-for-local-noetherian}}
    \arrow[Leftrightarrow]{rr}{\mbox{\footnotesize\Cref{proposition:conjectures-equivalent}}}
    \arrow[Leftrightarrow, swap]{dr}{\mbox{\footnotesize\Cref{proposition:conjecture-implies-conjecture-and-vice-versa}}}
    & & \mbox{\Cref{conjecture:referee}}
    \arrow[Leftrightarrow]{dl} \\
    & \mbox{\Cref{conjecture:ntsod-is-etale}}&
  \end{tikzcd}
\end{equation}
They should be considered in the context of
the (natural) expectation that having a nontrivial semiorthogonal decomposition is a \emph{Zariski open condition} in (smooth proper) families. As explained in the introduction, this result has a rich history in (noncommutative) algebraic geometry and deformation theory of abelian categories, mostly from the point of view of exceptional objects. \Cref{conjecture:ntsod-is-etale} extends this to semiorthogonal decompositions with more complicated components, and assumes nothing about the properties of the components.

In the companion paper \cite{MR4939473} we explain how \Cref{conjecture:ntsod-is-etale}
can be used to study the \emph{indecomposability} of derived categories in families.
Let us quote an important example from op.~cit.,
to illustrate how these conjectures yield indecomposability results.

\begin{example}
  \label{example:indecomposability}
  Let~$C$ be a smooth projective curve of genus~$g$, and let~$n=1,\ldots,\lfloor\frac{g+3}{2}\rfloor-1$. Then assuming \Cref{conjecture:ntsod-is-etale},~$\derived^\bounded(\Sym^nC)$ is indecomposable. The starting point for this is the indecomposability results of \cite{1508.00682v2}, which are used in \cite{MR4276298} to show that for a \emph{generic} curve the derived category of~$\Sym^nC$ is indecomposable. One can then extend this to all curves,
  subject to a positive answer to \Cref{conjecture:ntsod-is-etale}.

  However,
  this indecomposability is now also known using different methods,
  unconditionally,
  and for all~$n=1,\ldots,g-1$,
  by \cite{2107.09564}.
\end{example}

\subsection{Group actions}
\label{subsection:group-actions}
There are two interesting groups acting on the moduli space~$\sodl{\ell}_f$ (and~$\ntsodl{\ell}_f$):
\begin{enumerate}
  \item the braid group, encoding mutation of semiorthogonal decompositions;
  \item the auto-equivalence group.
\end{enumerate}
We will now define and study these two group actions,
as we expect them to play an important role in the study of the geometry of~$\sodl{\ell}_f$ (and~$\ntsodl{\ell}_f$).
Throughout,
$U$ is a quasicompact, semiseparated and excellent scheme,
and~$f\colon\cX\to U$ is a smooth and proper morphism of schemes.

\subsubsection{\textbf{Acting by mutations}}
The braid group action for a given semiorthogonal decomposition
was recalled in \Cref{subsection:sod}.
We will now upgrade it to an action on~$\sodl{\ell}_f$ (and~$\ntsodl{\ell}_f$).

\begin{proposition}
  \label{proposition:braid-group-on-sod-f}
  The action of the braid group on the set of semiorthogonal decompositions
  lifts to an action on~$\sodl{\ell}_f$ (and~$\ntsodl{\ell}_f$).
\end{proposition}

\begin{proof}
  We check the assertions for the sheaf
  \(
    \overline{\sod}^\ell_f
  \)
  defined on affine schemes over~$U$ (see~\Cref{definition:presheaf-sod-f}),
  as the case of~$\sodl{\ell}_f$ follows from the equivalence of toposes of \Cref{lemma:equivalence-of-toposes}.

  \cref{lemma:oo-admissible-sod} implies that mutations preserve admissibility of the components of the semiorthogonal decomposition. To see that the mutations commute with base change, recall that the base change of~$V'$\dash linear semiorthogonal decompositions
  for~$\Perf\cX _{V '}$,
  where
  \(
    V ' \to U
  \)
  an affine scheme over~$U$, under a morphism
  \(
    \phi \colon V \to V '
  \)
  of affine schemes over~$U$, was defined in \Cref{proposition:base-change-sod} as
  \begin{equation}
    \mathcal{A}_\phi^i\defeq\langle \phi_{\mathcal{X}}^*E\mid E\in\mathcal{A}^i\rangle.
  \end{equation}
  Let us check how base change is compatible with mutation, in the case of a right mutation,
  the case of a left mutation being similar.
  For this it suffices to check that
  \begin{equation}\label{equation:base change and mutation}
    \rightmutation_i(\mathcal{A}^\bullet)_\phi
    =
    \rightmutation_i(\mathcal{A}_\phi^\bullet)
    \subseteq
    \Perf
    \mathcal{X}_V,
  \end{equation}
  where~$\mathcal{A}_\phi^\bullet$ denotes the base change along~$\phi$ of the semiorthogonal decomposition~$\mathcal{A}^\bullet$. Thanks to the description of right mutation as the semiorthogonal complement~\eqref{equation:right mutation}, both sides of~\eqref{equation:base change and mutation} can be described as the following subcategory of~$\Perf\mathcal{X}_V$:
  \begin{equation}
    {}^\perp\langle\mathcal{A}^1,\ldots,\mathcal{A}^{i-2},\mathcal{A}^{i}\rangle
    \cap
    \langle\mathcal{A}^{i+1},\ldots,\mathcal{A}^\ell\rangle^\perp
  \end{equation}
  The action preserves the subfunctor~$\ntsodl{\ell}_f$ as mutation induces an equivalence of the components.
\end{proof}

An interesting and important question is to understand when this action is \emph{transitive}.
For full exceptional collections, this is a famous question \cite{MR1230966},
which is answered affirmatively for del Pezzo surfaces \cite{MR1286839} and the Hirzebruch surface of degree 2 \cite{2107.03051}.
For more general semiorthogonal decompositions there are counterexamples \cite{MR3177299,1304.0903,MR4701883},
and recently counterexamples of non-transitivity for full exceptional collections have been constructed in \cite{MR4897441}.

It is also possible to improve the somewhat pathological geometric properties of~$\sodl{\ell}_{f}$,
in particular its nonseparatedness,
by quotienting out the braid group action, as the following example shows.
\begin{example}
  Let $ f $ be the morphism as in \Cref{eg:OU},
  i.e.,
  the degeneration of~$\Sigma_0=\mathbb{P}^1\times\mathbb{P}^1$ to~$\Sigma_2$ over~$U=\mathbb{A}^1$.
  Let $ \excl{4}_f \hookrightarrow \sodl{ 4 }_f $ be the open subspace of semiorthogonal decompositions given by full $f$-exceptional collections.
  It follows from \cite[Theorem~6.1]{2107.03051} that $ \excl{4 }_f / \Br_4 \cong U = \bA^1$.
  Therefore, by taking the quotient, the annoying features pointed out in \Cref{eg:OU} go away,
  at least for full~$f$-exceptional collections.
\end{example}

\subsubsection{\textbf{Acting by autoequivalences}}
For the second group action, let us recall the following definitions. Throughout, as ever,
$U$ is a quasicompact, semiseparated and excellent scheme,
and~$f\colon\cX\to U$ is a smooth and proper morphism of schemes.
\begin{definition}
  An
  \(
    f
  \)-linear autoequivalence of Fourier--Mukai type is an autoequivalence $ \Phi $ of $ \Perf  \cX  $ of the form
  \begin{equation}
    \Phi( - )
    =
    \mathbf{R} p_{ 2 * } \left( p_1^* ( - ) \otimes^{ \mathbf{L} } \cE \right)
  \end{equation}
  for some object $ \cE \in \Perf ( \cX \times_U \cX ) $,
  where $ p_1, p_2 \colon \cX \times_U \cX \to \cX $ are the projections.
  In this paper we will only consider autoequivalences of Fourier--Mukai type, and simply refer to them as autoequivalences.
  The group of $ f $-linear autoequivalences will be denoted by $ \Auteq ( f )$.
\end{definition}

It easily follows from the projection formula that
\begin{equation}
  \left( f^* \cF \otimes - \right) \circ \Phi
  \cong
  \Phi \circ \left( f^* \cF \otimes - \right)
\end{equation}
for all
\(
  \Phi \in \Auteq ( f )
\)
and all
$ \cF \in \Perf U $. Also, note that an
\(
  f
\)-linear autoequivalence
\(
  \Phi
\)
(of Fourier--Mukai type) admits the base change along a morphism
\(
  \phi \colon V \to U
\):
namely, an
\(
  f _{ V }
\)-linear autoequivalence of Fourier--Mukai type of
\(
  \Perf \cX _{ V }
\)
defined by the pullback of the kernel~$\cE$ along the induced morphism.
\begin{remark}
  \label{remark:toen-vaquie-group}
  One would like to upgrade this to a sheaf of groups,
  by considering~$f_V$\dash linear autoequivalences, for~$V\to U$ \'etale.
  In \cite[Corollary~3.24]{MR2493386} this is shown to be a group algebraic space, locally of finite type,
  when~$U$ is affine and assuming~$f$ cannot be written as a disjoint union of morphisms.
  To check condition~(1) from op.~cit.,
  one uses that Hochschild cohomology of schemes vanishes in negative degrees,
  by the Hochschild--Kostant--Rosenberg decomposition for smooth morphisms from \cite[\S0]{MR2357327}.

  In light of \Cref{conjecture:smooth-proper-dg-family},
  it would be interesting to show that over an arbitrary base~$U$,
  this sheaf of groups is a locally algebraic group space,
  and that it acts on the moduli space of semiorthogonal decompositions.
  Observe that the vanishing of negative Hochschild cohomology
  does not hold for smooth and proper dg~categories in general.
  To construct an example,
  it suffices to consider the dg category obtained as the gluing of~$\bfk$ and~$\bfk$ along the dg~bimodule~$\bfk\oplus\bfk[-n]$.
\end{remark}

For now we will restrict ourselves to the action of the global~$f$\dash linear autoequivalences, and prove the following easy assertion.

\begin{proposition}
  The group~$\Auteq(f)$ naturally acts on~$\sodl{\ell}_f$, and preserves the subfunctor~$\ntsodl{\ell}_f$.
\end{proposition}
For convenience of the reader, we include a self-contained proof (see also \cite[Theorem~6.4]{MR2801403} for a more general result).
\begin{proof}
  Let~$\Phi$ be an~$f$\dash linear autoequivalence,
  \(
    V ' \to U
  \)
  an affine scheme over~$U$, and let
  \begin{equation}
    \Perf\cX _{ V ' }
    =
    \langle\cA^1,\ldots,\cA^\ell\rangle
  \end{equation}
  be a semiorthogonal decomposition, which we will denote~$\mathcal{A}^\bullet$.
  Let
  \(
    V \to V '
  \)
  be a morphism of affine schemes over~$U$,
  and consider the base change semiorthogonal decomposition
  \begin{equation}
    \Perf\cX _V
    =
    \langle\cA_\phi^1,\ldots,\cA_\phi^\ell\rangle,
  \end{equation}
  where~$\cA_\phi^i$ is defined as in \Cref{proposition:base-change-sod}.
  We need to check that
  \begin{equation}\label{equation:autoequivalence and base change}
    \Phi _{ V }(\cA_\phi^i)
    =
    \left(
      \Phi _{ V ' } ( \cA^i )
    \right) _{ \phi }
  \end{equation}
  for
  \(
    i = 1, \dots, \ell
  \),
  where
  \(
    \Phi _{ V }
  \)
  and
  \(
    \Phi _{ V ' }
  \)
  are the base changes of~$\Phi$ to~$\Perf\cX _V$ and~$\Perf\cX _{ V ' }$, respectively.
  This follows combining the isomorphism of functors
  \begin{equation}
    \Phi _{ V } \circ \phi_{\cX _{ V ' }} ^{ \ast }
    \cong
    \phi_{\cX _{ V ' }} ^{ \ast } \circ \Phi _{ V ' }
  \end{equation}
  with the definition of base change of semiorthogonal decompositions given in~\eqref{def:A_phi}.
\end{proof}

\begin{remark}
  \label{remark:rigidity-of-autoequivalences}
  This group action, together with the fact that~$\sod_f^{\ell}$ is \'etale over~$U$
  leads to a conceptual proof of the fact that
  topologically trivial autoequivalences act trivially on semiorthogonal decompositions,
  as in~\cite[Corollary~3.15]{1508.00682v2}.
  We will sketch it here.

  Consider a smooth projective variety~$g\colon X\to\Spec\bfk$,
  where~$\bfk$ is an algebraically closed field.
  Consider the connected component of the identity~$\Auteq^0(g)$,
  which,
  by \cite[Th\'eor\`eme~4.18]{MR2806466} and the discussion in \cite[\S3]{MR2839458}
  is isomorphic to the group scheme~$U \defeq \Aut_{X/\bfk}^0 \times \Pic_{X/\bfk}^0$.
  Then we can consider the smooth projective family~$f = \pr_U\colon X\times U\to U$.
  Let $ \cL $ be the universal line bundle on $X \times\Pic_{X/\bfk}^0 $
  and $\Sigma\colon X \times\Aut_{X/\bfk}^0\simto X \times\Aut_{X/\bfk}^0 $ the universal automorphism.
  Then the $U$\dash linear autoequivalence $ \Phi \colon \Perf (X \times U) \to \Perf (X \times U) $ defined by
  \begin{equation}
    \label{equation:universal-autoequivalence}
    \Phi ( - )
    \defeq
    \left(\Sigma\times\id_{\Pic_{X/\bfk}^0}\right)_*(-)
    \otimes
    \pr_{X,\Pic_{X/\bfk}^0}^{\ast}\cL
  \end{equation}
  is the universal autoequivalence by which $ U $ represents the moduli functor,
  so that its restriction to the point $(\sigma,L)\in U(\bfk)$
  coincides with the autoequivalence $ \sigma_* ( - ) \otimes L$.

  Let us fix a semiorthogonal decomposition~$\derived^\bounded(X)=\langle\cA,\cB\rangle$.
  We wish to show that, for $\Phi$ as in \eqref{equation:universal-autoequivalence},
  we have that $ \cA = \Phi ( \cA ) \subset \Perf X$,
  so that every topologically trivial autoequivalence of $ \Perf X $ preserves $\cA$ as a subcategory of $ \Perf X$.

  The two $U$-linear semiorthogonal decompositions
  \begin{equation}
    \langle \pr_X^* \cA, \pr_X^* \cB \rangle
    = \Perf \left( X \times U \right)
    =
    \langle \Phi ( \pr_X^* \cA), \Phi (\pr_X^* \cB ) \rangle
  \end{equation}
  correspond to two sections $ U \rightrightarrows \sod_f$,
  and they coincide at the origin of $ U $.
  Now since $ \sod_f \to U $ is \'etale, it follows that they coincide on a Zariski open subset $ V \subset U $ containing the origin.

  Note that we have not yet proved the equality $ V = U $, since $ \sod_f \to U $ may not be separated.
  However, note that $ V $ is actually a \emph{subgroup scheme} of $ U$,
  as it is a stabiliser subgroup of the semiorthogonal decomposition.
  As $ V $ is open dense in $ U $, we can conclude $ U = V $.
\end{remark}

Let us for good measure show that topologically nontrivial autoequivalences can act nontrivially.
\begin{example}
  Consider the case~$U=\Spec\bfk$.
  To see that tensoring with line bundles can act nontrivially it suffices to consider~$\mathbb{P}^1$ and the exceptional collection given by~$\langle\mathcal{O}_{\mathbb{P}^1}(i),\mathcal{O}_{\mathbb{P}^1}(i+1)\rangle$.
  To see that automorphisms can act nontrivially, consider the surface $ \bP^1 \times \bP^1$.
  Then every involution exchanging the two components, which is topologically nontrivial,
  sends the exceptional line bundle $ \cO ( 1, 0 ) $ to $ \cO ( 0, 1 )$.
  In particular, such an involution does not preserve the semiorthogonal decomposition induced by $ \cO ( 1, 0 )$.

\end{example}

We will finally show that the two group actions introduced so far commute with each other. We expect the same result to hold for the action of the algebraic group action suggested in \Cref{remark:toen-vaquie-group}. In the absolute case this result is given (without proof) as \cite[Lemma~2.2(ii)]{MR3987870}. For the reader's sake we fill in some of the details in the relative case.

\begin{proposition}
  The actions of~$\Br_\ell$ and~$\Auteq(f)$ commute.
\end{proposition}

\begin{proof}
  It suffices to show that the right mutation~$\rightmutation_i(\mathcal{A}^\bullet)$
  commutes with an~$f$-linear autoequivalences~$\Phi$,
  the proof for the left mutation being dual.
  But, indeed, we have
  \begin{align*}
    \Phi
    \left(
      \rightmutation_i(\mathcal{A}^\bullet)
    \right)
    &=
    {}^\perp\langle \Phi(\mathcal{A}^1),\ldots,\Phi(\mathcal{A}^{i-2}),\Phi(\mathcal{A}^i)\rangle\cap\langle\Phi(\mathcal{A}^{i+1}),\ldots,\Phi(\mathcal{A}^\ell)\rangle^\perp
    \\
    &=
    \rightmutation_i(\Phi(\mathcal{A}^\bullet)),
  \end{align*}
  where both identities follow from \eqref{equation:right mutation}.
\end{proof}

\section{Example: families of cubic surfaces}
\label{section:del-pezzo-families}
In this final section we briefly discuss the precise relationship between
the moduli spaces~$\sodl{\ell}_f$ associated to a family of cubic surfaces
and the moduli spaces of lines in a family of cubic surfaces,
the motivating example from the introduction.

Let $ f \colon \cX \to U $ be a versal family of cubic surfaces,
with $U$ connected, and defined over an algebraically closed field.
One can associate to $ f $ the following \'etale morphisms over $ U $:
\begin{enumerate}
  \item The moduli space of nontrivial semiorthogonal decompositions $ \sodl{2}_f \to U $ of length~$ 2 $.
  \item The moduli space of $  (- 1) $-curves (or relative Fano scheme of lines) in the fibers of $f$,
    which will be denoted by $ t \colon \mathcal{F} \to U$.
    It is a finite \'etale morphism of degree~27.
\end{enumerate}

Let us explain how these two spaces are related.
Let us take the base change
\begin{equation}
  \begin{tikzcd}
    \cX_t\defeq
    \cX \times_{ f, U, t } \mathcal{F} \arrow{r}{t^* f} & \mathcal{F},
  \end{tikzcd}
\end{equation}
over which there exists the universal line $ e\colon\mathcal{L} \hookrightarrow \cX_t$.
Firstly, we can construct two open immersions of the form
\begin{equation}
  \label{equation:first-open-immersion}
  \begin{tikzcd}
    \mathcal{F}\arrow[hook]{r} & \sodl{2}_f
  \end{tikzcd}
\end{equation}
of algebraic spaces over~$U$.
To do this, we consider the~$f$\dash linear semiorthogonal decomposition
given by the~$t^*f$-exceptional object~$\mathcal{O}_{e(\mathcal{L})}$ and its complement.
The choice between the left and right orthogonal gives two different open immersions.
They are related by the action of the braid group $ \Br_2 $.

Next, let $ \ttilde \colon \tilde{\mathcal{F}} \to U $ be the Galois closure of $ t $.
It is well known that the Galois group of $ \ttilde $ is the Weyl group $ \mathrm{W} ( \mathrm{E}_6) $ of order~51840. See \cite[page~716]{MR0552521} for a proof, and the book \cite{MR0091260} by Camille Jordan for the first treatment of the topic.
The base change
\begin{equation}
  \begin{tikzcd}
    \cX_{ \ttilde }\defeq
    \cX \times_{ f, U, \ttilde } \tilde{\mathcal{F}} \arrow{r}{\ttilde^* f } & \tilde{\mathcal{F}}
  \end{tikzcd}
\end{equation}
admits 6 closed immersions $ e_1, \dots, e_6 \colon\mathcal{L}
\hookrightarrow \cX_{ \ttilde } $ over $ \tilde{\mathcal{F}} $ such that for each $ x \in \tilde{\mathcal{F}} $, the images over $ x $ are mutually disjoint $(-1)$-curves of the fiber $ \left( \cX_{ \ttilde } \right)_x
=
\cX_{ \ttilde ( x )}
$.
Similarly, there exists another closed immersion
\(
  \ell _{ 1 2 }
  \colon
  \mathcal{L}
  \hookrightarrow
  \cX_{ \ttilde }
\)
representing the unique $(-1)$-curve which intersects
\(
  e _{ 1 },
  e _{ 2 }
\)
but not
\(
  e _{ 3 },
  \dots,
  e _{ 6 }
\).
We put
\(
  \cO _{ \cX } ( H )
  \defeq
  \cO _{ \cX } ( e _{ 1 } ( \cL ) + e _{ 2 } ( \cL ) + \ell _{ 1 2 } ( \cL ) )
\)
and
\(
  \cO _{ \cX } ( 2 H )
  \defeq
  \cO _{ \cX } ( H ) ^{ \otimes 2 }
\).
Now let
\begin{equation}
  \begin{tikzcd}
    \iota \colon
    \tilde{\mathcal{F}} \arrow[hook]{r} & \sodl{9}_f
  \end{tikzcd}
\end{equation}
be the open immersion of algebraic spaces over~$U$ corresponding to the $ \ttilde^* f $-exceptional collection obtained by the $ f $-exceptional collection
\begin{equation}
  \cO_\cX,
  \cO_\cX ( H ),
  \cO_\cX ( 2 H ),
  \cO_{ e_1 ( \mathcal{L} ) },
  \dots,
  \cO_{ e_6 ( \mathcal{L} ) }
\end{equation}
via Orlov's blowup formula \cite[Theorem~4.3]{MR1208153}.

Let $ W \leq \Br_9 $ be the subgroup of elements which preserves the connected component $ \iota ( \tilde{\mathcal{F}} )
$.
It acts as covering transformations of the Galois cover $ \ttilde
$,
so there exists a natural homomorphism
\begin{equation}
  \label{eq:from W to W(E6)}
  \begin{tikzcd}
    W \arrow{r} & \mathrm{W} ( \mathrm{E}_6 ).
  \end{tikzcd}
\end{equation}

\begin{proposition}
  The natural morphism \eqref{eq:from W to W(E6)} is surjective.
\end{proposition}
\begin{proof}
  It is enough to show that the action of $W$ on a fiber of $ \ttilde $ is transitive. For this purpose, fix a point $ u \in U $ and let~$S$
  be the corresponding cubic surface.
  Then the points of the fiber $ \ttilde^{ -1 } ( u ) $ bijectively correspond to markings on $S$.
  Let $ \ell_1, \dots, \ell_6 $ and $ \ell '_1, \dots, \ell '_6 $ be two sequences of six disjoint lines on $ S $.
  By \cite{MR1286839}, the action of $ \Br_9 $ on the set of full exceptional collections of $\derived^\bounded(S)  $ is transitive (up to shifts).
  Hence there exists $ b \in \Br_9 $ such that
  \begin{equation}
    b \left( \cO_S, \cO_S ( H ), \cO_S ( 2 H ), \cO_{ \ell_1 }, \dots, \cO_{ \ell_6 } \right)
    =
    \left( \cO_S, \cO_S ( H ' ), \cO_S ( 2 H ' ), \cO_{ \ell'_1 }, \dots, \cO_{ \ell '_6 } \right),
  \end{equation}
  where~$H$ and~$H'$ are the pullbacks of the classes of lines of~$\mathbb{P}^2$ obtained by contracting the sets of~$6$ lines $ \ell_1, \dots, \ell_6 $ and $ \ell '_1, \dots, \ell '_6$, respectively.
  As the braid group action exchanges the connected components of $ \sodl{ 9 }_f$,
  it follows that $ b ( \iota ( \tilde{\mathcal{F}} ) ) = \iota ( \tilde{\mathcal{F}} )$, i.e., $ b \in W$.
\end{proof}

The group $ W $ is interesting
in that it is related to such a classical subject as the monodromy of lines though it arises from the categorical notion of mutations.
Note however that the homomorphism~\eqref{eq:from W to W(E6)} has a big kernel, since no element of $ W $ has finite order whereas $ \mathrm{W} ( \mathrm{E}_6 ) $ is finite. For this reason it is yet to be further investigated. Specifically we ask:

\begin{problem}
  Identify the elements of the group $ W $. For example, are they exactly those elements of $ \Br_9 $ which preserves the Orlov type semiorthogonal decomposition?
  Also, identify the kernel of \eqref{eq:from W to W(E6)}. Does it coincide with the subgroup $ \Br_6\hookrightarrow W $ acting on the last six components of the semiorthogonal decomposition?
\end{problem}

\appendix

\section{Lifting perfect complexes and morphisms thereof}
\label{app:liftings}
In this appendix we describe various lifting results for perfect complexes and morphisms of perfect complexes
in an inverse system of schemes.
\Cref{lemma3.1,proposition:3.2,proposition:3.3}
were suggested by a referee,
in support of their arguments in \cref{subsection:extending-sods}.

Let $f\colon \mathcal X \to U$ be a morphism of schemes. Fix a directed system of $U$-algebras $\set{A_i}_i$, namely a compatible system of affine maps $\Spec A_i \to U$, and set $A = \varinjlim A_i$. We denote the base changes of $f$ by
\begin{equation}
  \begin{tikzcd}
    \mathcal X_i \defeq \mathcal X \times_U\Spec A_i \arrow{r}{f_i} & \Spec A_i,\qquad   \mathcal X_A \defeq \mathcal X \times_U\Spec A \arrow{r}{f_A} & \Spec A.
  \end{tikzcd}
\end{equation}
If $j \geq i$, we have maps $A_i \to A_j \to A$ inducing $\Spec A \to \Spec A_j \to \Spec A_i$ and $\mathcal X_A \to \mathcal X_j \to \mathcal X_i$, and for an object $F \in \Perf \mathcal X_i$ we denote by $F|_{\mathcal X_j}$ and $F|_{\mathcal X_A}$ its derived pullback along the canonical projections $\mathcal X_j \to \mathcal X_i$ and $\mathcal X_A \to \mathcal X_i$, respectively. Note that these are perfect complexes on $\mathcal X_j$ and $\mathcal X_A$, as derived pullback always preserves perfect complexes \cite[\href{https://stacks.math.columbia.edu/tag/09UA}{Tag 09UA}]{stacks-project}.

\begin{lemma}
  \label{lemma3.1}
  Let $B \to A$ be a morphism of $U$-algebras, and let $f \colon \mathcal X \to U$ be a flat morphism. Form the cartesian square
  \begin{equation}
    \begin{tikzcd}
      \mathcal X_A \MySymb{dr} \arrow{r}{\xi}\arrow[swap]{d}{f_A} & \mathcal X_B \arrow{d}{f_B} \\
      \Spec A \arrow[swap]{r}{\eta} & \Spec B.
    \end{tikzcd}
  \end{equation}
  Then
  \begin{enumerate}
    \item
      \label{item-i}
      the canonical morphism of sheaves
      \begin{equation}
        \label{base-change-iso}
        \begin{tikzcd}
          f_B^\ast \eta_\ast \mathcal O_{\Spec A} \arrow{r} & \xi_\ast \mathcal O_{\mathcal X_A}
        \end{tikzcd}
      \end{equation}
      is an isomorphism, and
    \item
      \label{item-ii}
      for any perfect complexes $F,G \in \Perf \mathcal X_B$, we have
      \begin{equation}
        \RRHom(F,G \otimes \xi_\ast \mathcal O_{\mathcal X_A}) \cong \RRHom(F,G) \otimes_BA.
      \end{equation}
  \end{enumerate}
\end{lemma}

\begin{proof}
  Both sheaves $f_B^\ast \eta_\ast \mathcal O_{\Spec A}$ and $\xi_\ast \mathcal O_{\mathcal X_A}$ are pure (the first because $\eta$ is affine and $f_B$ is flat, the second because $\xi$ is affine), therefore the question is local and we may assume both $\mathcal X$ and $U$ are affine, say $\mathcal X = \Spec S$ and $U=\Spec C$. Then
  \begin{equation}
    \mathcal X_A = \Spec S \otimes_CA,\qquad \mathcal X_B = \Spec S\otimes_CB.
  \end{equation}
  In this situation, the morphism \eqref{base-change-iso} is the isomorphism $(S\otimes_CB)\otimes_BA \cong S \otimes_CA$, which proves \cref{item-i}.

  To prove \cref{item-ii}, we compute
  \begin{equation}
    \begin{aligned}
      \RRHom(F,G \otimes \xi_\ast \mathcal O_{\mathcal X_A})
      &\cong \RRHom(\mathcal O_{\mathcal X_B}, F^\vee \otimes G \otimes \xi_\ast \mathcal O_{\mathcal X_A})& \mbox{as }F\mbox{ is perfect} \\
      &\cong f_{B\ast}(F^\vee \otimes G \otimes \xi_\ast \mathcal O_{\mathcal X_A}) & \mbox{by adjunction} \\
      &\cong f_{B\ast}(F^\vee \otimes G \otimes f_B^\ast \eta_\ast \mathcal O_{\Spec A} ) & \mbox{by \ref{item-i}} \\
      &\cong f_{B\ast}(F^\vee \otimes G)\otimes \eta_\ast \mathcal O_{\Spec A} & \mbox{by projection formula} \\
      &\cong \RRHom(F,G)\otimes_BA & \mbox{by adjunction}
    \end{aligned}
  \end{equation}
  which completes the proof.
\end{proof}

The following lemma can be seen as a version of \cite[Proposition~2.2.1]{MR2177199} for arbitrary morphisms of perfect complexes, as opposed to isomorphisms thereof.

\begin{proposition}
  \label{proposition:3.2}
  Let $f\colon \mathcal X \to U$ be a flat morphism of schemes. Fix a directed system of $U$-algebras $\set{A_i}_{i \in I}$ and two perfect complexes $F,G \in \Perf \mathcal X_i$ for a fixed $i \in I$. Let $\phi \colon F|_{\mathcal X_A} \to G|_{\mathcal X_A}$ be a morphism in $\Perf \mathcal X_A$. Then there exists an index $j \geq i$ and a morphism $\phi_j \colon F|_{\mathcal X_j} \to G|_{\mathcal X_j}$ such that $\phi_j|_{\mathcal X_A} = \phi$.
\end{proposition}

\begin{proof}
  Let $\xi_i \colon \mathcal X_A \to \mathcal X_i$ be the natural morphism. Note that
  \begin{equation}
    \begin{aligned}
      \RRHom(F|_{\mathcal X_A} , G|_{\mathcal X_A})
      &= \RRHom(\xi_i^\ast F,\xi_i^\ast G) \\
      &\cong \RRHom(F,\xi_{i\ast}\xi_i^\ast G) & \mbox{by adjunction }\eqref{adjunction} \\
      &\cong \RRHom(F,G\otimes \xi_{i\ast}\mathcal O_{\mathcal X_A}) & \mbox{by projection formula} \\
      &\cong \RRHom(F,G) \otimes_{A_i}A & \mbox{by \Cref{lemma3.1}}.
    \end{aligned}
  \end{equation}
  A similar argument shows that
  \begin{equation}
    \RRHom(F|_{\mathcal X_j},G|_{\mathcal X_j}) \cong \RRHom(F,G) \otimes_{A_i}A_j,
  \end{equation}
  for every $j \geq i$. Now, our starting point is a morphism
  \begin{equation}
    \phi \in \mathrm{H}^0(\RRHom(F|_{\mathcal X_A} , G|_{\mathcal X_A})) = \mathrm{H}^0(\RRHom(F,G) \otimes_{A_i}A).
  \end{equation}
  We need to check that $\phi$ comes from an element in $\mathrm{H}^0(\RRHom(F,G) \otimes_{A_i}A_j)$ under the canonical map
  \begin{equation}
    \begin{tikzcd}
      \RRHom(F,G) \otimes_{A_i}A_j \arrow{r} & \RRHom(F,G) \otimes_{A_i}A,
    \end{tikzcd}
  \end{equation}
  for some $j \geq i$.

  To begin with, we claim that $\RRHom(F,G)$ is a bounded above complex. This follows from \cite[\href{https://stacks.math.columbia.edu/tag/0GEN}{Tag 0GEN}]{stacks-project} after observing that $F$ and $G$ are perfect, and that $\Spec A_i$, being affine, is quasicompact and semiseparated. Therefore $\RRHom(F,G)$ can be represented by a bounded above complex of free $A_i$-modules
  \begin{equation}
    \begin{tikzcd}
      (P^\bullet,d^\bullet)\colon \qquad
      \cdots \arrow{r} & P^{n-2}\arrow{r}{d^{n-2}} & P^{n-1}\arrow{r}{d^{n-1}} & P^n \arrow{r} & 0.
    \end{tikzcd}
  \end{equation}
  An element in $\mathrm{H}^0(\RRHom(F,G) \otimes_{A_i}A)$ is then represented by an element $h \in P^0 \otimes_{A_i}A$ such that $d(h) = 0$, where $d=d^0\colon P^0 \to P^1$ is the differential of $P^\bullet$. Note that $h$ is contained in the tensor product $A_i^{\oplus n_0} \otimes_{A_i}A$ for some finite free submodule $A_i^{\oplus n_0} \subset P^0$ and, moreover, $d(A_i^{\oplus n_0}) \subset A_i^{\oplus n_1}\subset P^1$ for some finite free submodule $A_i^{\oplus n_1}\subset P^1$. Thus, to lift $h$, it is enough to lift the corresponding element $h \in A_i^{\oplus n_0} \otimes_{A_i}A \cong A^{\oplus n_0}$ to an element $h_j \in A_i^{\oplus n_0} \otimes_{A_i}A_j \cong A_j^{\oplus n_0}$ for some $j$, and to note that the equality
  \begin{equation}
    d(h_j)|_A = d(h_j|_A) = dh = 0
  \end{equation}
  in $A_i^{\oplus n_1}\otimes_{A_i}A \cong A^{\oplus n_1}$ implies that $d(h_j)|_{A_k} = 0$ in $A_i^{\oplus n_1}\otimes_{A_i}A_k \cong A_k^{\oplus n_1}$ for appropriate $k$. This completes the proof.
\end{proof}

\begin{corollary}\label{corollary:bijective}
  Consider the setup of \Cref{proposition:3.2}, and assume furthermore that
  \(
    \cX _{ i }
  \)
  is quasicompact. Then the canonical map
  \begin{equation}
    \begin{tikzcd}
      \displaystyle\varinjlim_{j \geq i}
      \Hom_{\Perf \mathcal X_j}(F|_{\mathcal X_j},G|_{\mathcal X_j})
      \arrow{r} &
      \Hom_{\Perf \mathcal X_A}(F|_{\mathcal X_A},G|_{\mathcal X_A})
    \end{tikzcd}
  \end{equation}
  is bijective.
\end{corollary}

\begin{proof}
  The surjectivity follows from \Cref{proposition:3.2}. For the injectivity, let $\phi_j \colon F|_{\mathcal X_j} \to G|_{\mathcal X_j}$ be a morphism such that $\phi_j|_{\mathcal X_A} = 0$. Consider the following distinguished triangle
  \begin{equation}
    \begin{tikzcd}
      F|_{\mathcal X_j} \arrow{r}{\phi_j} & G|_{\mathcal X_j} \arrow{r}{\iota _{ j }} & C _{ j } \arrow{r} & F|_{\mathcal X_j}[1]
    \end{tikzcd}
  \end{equation}
  in $\Perf \mathcal X_j$, where $C _{ j }$ is the cone of $\phi_j$.
  The assumption implies that this triangle is split after restriction to $\mathcal X_A$, which means that there exists a morphism
  \(
    \pi \colon C _{ j } \vert _{ \cX _{ A } }
    \to
    G \vert _{ \cX _{ A } }
  \)
  such that
  \(
    \pi \circ \iota _{ j } \vert _{ \cX _{ A } }
    =
    \id _{ G \vert _{ \cX _{ A } } }
  \).
  By \Cref{proposition:3.2}, there exists an index $k \geq j$ and a morphism
  \(
    \pi _{ k }
    \colon
    C _{ j } \vert _{ \cX _{ k } }
    \to
    G \vert _{ \cX _{ k } }
  \)
  such that
  \(
    \pi _{ k } \vert _{ \cX _{ A } }
    =
    \pi
  \).
  In particular, we have
  \(
    \pi _{ k }\vert _{ \cX _{ A } } \circ \iota _{ j } \vert _{ \cX _{ A } }
    =
    \id _{ G \vert _{ \cX _{ A } } }
  \).
  Since
  \(
     \cX _{ i }
  \)
  and hence
  \(
     \cX _{ \ell }
  \)
  is quasicompact for every $\ell \geq i$, \Cref{corollary:a priori isomorphy} implies that
  \(
    \pi _{ \ell } \circ \iota _{ j } \vert _{ \cX _{ \ell } }
  \)
  is an isomorphism for some $\ell \geq k$. This, in turn, implies that
  \(
    \phi _{ j } \vert _{ \cX _{ \ell } } = 0
  \).
\end{proof}

\begin{proposition}
  \label{proposition:3.3}
  Let $f\colon \mathcal X \to U$ be a flat semiseparated quasicompact morphism of schemes. Fix a directed system of $U$-algebras $\set{A_i}_{i \in I}$, and set $A = \varinjlim A_i$. Then, for any $F \in \Perf \mathcal X_A$, there is an index $j$ and a lift of $F$ over $\mathcal X_j$, i.e.~a perfect complex $F_j \in \Perf \mathcal X_j$ such that $F \cong F_j|_{\mathcal X_A}$.
\end{proposition}

\begin{proof}
    Take an index
    \(
       i
    \).
    Since
    \(
       \cX _{ i }
    \)
    is quasicompact and semiseparated over the affine scheme
    \(
       \Spec A _{ i }
    \),
    it is itself quasicompact and semiseparated. Take a classical generator $G _{ i }$ of $\Perf \cX _{ i }$ by \Cref{lemma:bvdb-ko-p}\ref{item:cg-i},
  whose base change $G _{ A }$ is again a classical generator of $\Perf \cX _{ A }$ by \Cref{lemma:bvdb-ko-p}\ref{item:cg-iii}.
  Then
  \(
    F
  \)
  is the image of an idempotent endomorphism of an iterated cone of endomorphisms of $G _{ A }$.
  By \Cref{corollary:bijective}, the endomorphisms and the idempotent lift to $\cX _{ j }$ for some
  \(
    j \ge i
  \).
  Thus,
  the object
  \(
    F _{ j } \in \cX _{ j }
  \)
  we obtain by taking the corresponding iterated cone and the image of the idempotent satisfies
  \(
    F _{ j } \vert _{ \cX _{ A } }
    \cong
    F
  \)
  (a similar argument also appears in the proof of \Cref{theorem:sod-f-lfp} in \Cref{subsection:lfp-using-generators}).
\end{proof}

\bibliographystyle{amsplain}
\providecommand{\bysame}{\leavevmode\hbox to3em{\hrulefill}\thinspace}
\providecommand{\MR}{\relax\ifhmode\unskip\space\fi MR }
\providecommand{\MRhref}[2]{%
  \href{http://www.ams.org/mathscinet-getitem?mr=#1}{#2}
}
\providecommand{\href}[2]{#2}

\medskip
\small
\emph{Pieter Belmans}, \url{p.belmans@uu.nl} \\
\textsc{Mathematical Institute, Utrecht University, Budapestlaan 6, 3584 CD Utrecht, The Netherlands} \\

\emph{Shinnosuke Okawa}, \texttt{okawa@math.sci.osaka-u.ac.jp} \\
\textsc{Department of Mathematics, Graduate School of Science, Osaka University, Machikaneyama 1-1, Toyonaka, Osaka 560-0043, Japan} \\

\emph{Andrea T.~Ricolfi}, \texttt{aricolfi@sissa.it} \\
\textsc{Scuola Internazionale Superiore di Studi Avanzati (SISSA), Via Bonomea 265, 34136 Trieste, Italy}

\end{document}